\documentclass[12pt,reqno]{amsart}

\usepackage{amsmath,amsthm,amssymb}
\usepackage{graphicx}
\usepackage{tikz}
\usepackage[asymmetric]{geometry} 
\usepackage{hyperref}
\hypersetup{
    colorlinks=true,       
    linkcolor=black,          
    citecolor=black,        
    filecolor=black,      
    urlcolor=black           
}
\newcommand{\E}{\operatorname{\mathbb{E}}}
\newcommand{\ac}{{\mathcal{A}}}
\newcommand{\rc}{{\mathcal{R}}}
\newcommand{\ec}{{\mathcal{E}}}
\newcommand{\fc}{{\mathcal{F}}}
\newcommand{\rcnnp}{\rc(n_1,n_2,p)}
\newcommand{\Var} {\operatorname {Var}}
\newcommand{\Prob} {\mathbb {P}}
\newcommand{\prob}[1]{\Prob\left(#1\right)}
\newcommand{\probH}[1]{\Prob_H\left(#1\right)}
\newcommand{\Pc}[2]{\prob{#1\,|\,#2}}
\newcommand{\PcH}[2]{\Prob_H\left(#1\,|\,#2\right)}
\newcommand{\Ec}[2]{\E\left[#1\,|\,#2\right]}
\newcommand{\EcH}[2]{\E_H\left(#1\,|\,#2\right)}
\newcommand{\eps} {\varepsilon}
\newcommand{\R} {\mathbb R}
\newcommand{\Rnnp}{\R(n_1,n_2,p)}

\newcommand{\G} {\mathbb G}
\newcommand{\Gnnp} {\G(n_1,n_2,p)}
\newcommand{\Gnnm} {\G(n_1,n_2,m)}
\newcommand{\Knn} {K_{n_1,n_2}}

\newcommand{\I} {\mathbb I}

\newcommand{\Bi}{\operatorname{Bin}}

\newcommand{\Po}{\operatorname{Po}}
\newcommand{\Hyp}{\operatorname{Hyp}}
\newcommand{\dd}{\mathbf{d}}
\newcommand{\rcdd}{\rc(\dd_1,\dd_2)}
\newcommand{\cod}{\operatorname{cod}}
\newcommand{\bip}{\operatorname{bip}}

\newtheorem{theorem}{Theorem}
\newtheorem{lemma}[theorem]{Lemma}
\newtheorem{corollary}[theorem]{Corollary}
\newtheorem{proposition}[theorem]{Proposition}
\newtheorem{claim}[theorem]{Claim}
\newtheorem{conjecture}{Conjecture}
\theoremstyle{definition}
\newtheorem{remark}[theorem]{Remark}
\newtheorem{manuallemmainner}{Lemma}
\newenvironment{manuallemma}[1]{%
  \renewcommand\themanuallemmainner{#1}%
  \manuallemmainner
}
{
\endmanuallemmainner
}
\newtheorem{manualtheoreminner}{Theorem}
\newenvironment{manualtheorem}[1]{%
  \renewcommand\themanualtheoreminner{#1}%
  \manualtheoreminner
}
{
\endmanualtheoreminner
}

\newtheorem{manualpropinner}{Proposition}
\newenvironment{manualprop}[1]{%
  \renewcommand\themanualpropinner{#1}%
  \manualpropinner
}
{
\endmanualpropinner
}

\newcommand{\eb}{\eps_{B}}
\newcommand{\er}{\eps_{R}}
\newcommand{\ebr}{\eps_{B\cup R}}
\newcommand{\db}{d_{B}}
\newcommand{\dr}{d_{R}}

\newcommand{\x}{\rs{\sx}}
\newcommand{\y}{\rs{\sy}}
\newcommand{\sx}{X}
\newcommand{\scx}{\overline{X}}
\newcommand{\sy}{Y}
\newcommand{\scy}{\overline{Y}}

\newcommand{\rs}[1]{s(#1)}

\newcommand{\deltat}{\delta(t)}
\newcommand{\justify}[1]{\fbox{\tiny{#1}}\quad}
\newcommand{\lhs}{LHS}
\newcommand{\rhs}{RHS}
\newcommand{\ceil}[1]{\lceil{#1}\rceil}
\newcommand{\By}[2]{\overset{\mbox{\tiny\ensuremath{#1}}}{#2}}

\newenvironment{romenumerate}[1][-0.25em]
 {
  \vspace{#1}\begin{enumerate}
 }
 {\end{enumerate}}

\usetikzlibrary{arrows,positioning}
\tikzset{
    >=stealth',
    pil/.style={
           ->,
           thick,
           shorten <=2pt,
           shorten >=2pt,}
}

\begin{document}
\title{Sandwiching biregular random graphs}
\author{Tereza Klimo\v{s}ov\'{a}}
\address{Tereza Klimo\v{s}ov\'{a} \\ Faculty of Mathematics and Physics \\ Charles University \\ Prague, Czech Republic.}
\email{tereza@kam.mff.cuni.cz}
\thanks{Research of TK was supported by the grant no.~19-04113Y of the Czech
Science Foundation (GA\v{C}R) and the Center for Foundations of Modern Computer Science (Charles Univ. project UNCE/SCI/004).}
\author{Christian Reiher}
\address{Christian Reiher, Fachbereich Mathematik, Universit\"at Hamburg, Hamburg, Germany.}
\email{Christian.Reiher@uni-hamburg.de}
\author{Andrzej Ruci\'nski}
\address{Andrzej Ruci\'nski \\ Faculty of Mathematics and Computer Science \\ Adam Mickiewicz University \\
Pozna\'n, Poland.}
\email{rucinski@amu.edu.pl}
 \thanks{Research of AR was supported by Narodowe Centrum Nauki, grant 2018/29/B/ST1/00426.}
\author{Matas \v{S}ileikis}
\address{Matas \v{S}ileikis \\ The Czech Academy of Sciences \\ Institute of Computer Science \\ Prague, Czech Republic.}
\email{matas.sileikis@gmail.com}
\thanks{Research of M\v{S} was supported by the Czech Science Foundation, grant number GJ20-27757Y, with institutional support RVO:67985807}
\date{1 October 2021
}

\begin{abstract}
	Let $\Gnnm$ be a uniformly random $m$-edge subgraph of the complete bipartite graph $\Knn$ with bipartition $(V_1, V_2)$, where $n_i = |V_i|$, $i=1,2$. Given a real number $p \in [0,1]$ such that $d_1 := pn_2$ and $d_2 := pn_1$ are integers, let $\Rnnp$ be a random subgraph of $\Knn$ with every vertex $v \in V_i$ of degree $d_i$,  $i = 1, 2$. In this paper we determine sufficient conditions on $n_1,n_2,p$, and $m$ under which one can embed $\Gnnm$ into $\Rnnp$ and vice versa with probability tending to $1$.
 In particular, in the balanced case $n_1=n_2$, we show that if $p\gg\log n/n$ and $1 - p \gg \left(\log n/n \right)^{1/4}$,
 then for some $m\sim pn^2$, asymptotically almost surely one can embed $\Gnnm$ into $\Rnnp$, while for  $p\gg\left(\log^{3} n/n\right)^{1/4}$ and $1-p\gg\log n/n$  the opposite embedding holds. As an extension, we confirm the Kim--Vu Sandwich Conjecture for degrees growing faster than $(n \log n)^{3/4}$.
\end{abstract}

\maketitle

\section{Introduction}
\label{intro}

\subsection{History and motivation}

The Sandwich Conjecture of Kim and Vu~\cite{KV04} claims that if $d \gg \log n$, then for some sequences $p_1 = p_1(n) \sim d/n$ and $p_2 = p_2(n) \sim d/n$ there is a joint distribution of a random $d$-regular graph $\R(n,d)$ and two binomial random graphs $\G(n,p_1)$ and $\G(n,p_2)$ such that with probability tending to $1$
\begin{equation*}
  \G(n,p_1) \subseteq \R(n,d) \subseteq \G(n,p_2).
\end{equation*}
If true, the Sandwich Conjecture would essentially reduce the study of any monotone graph property of the random graph $\R(n,d)$ in the regime $d \gg \log n$ to the more manageable  $\G(n,p)$.

For $\log n \ll d \ll n^{1/3}(\log n)^{-2}$, Kim and Vu proved the embedding $\G(n,p_1) \subseteq \R(n,d)$ as well as an imperfect embedding $\R(n,d) \setminus H \subseteq \G(n,p_2)$, where $H$ is some pretty sparse subgraph of $\R(n,d)$. In~\cite{DFRSc} the lower embedding was extended to $d \ll n$ (and, in fact, to uniform hypergraph counterparts of the models $\G(n,p)$ and $\R(n,d)$). Recently, Gao, Isaev, and McKay~\cite{GIM} came up with a result which confirms the conjecture  for $d \gg n/\sqrt{\log n}$ (\cite{GIM} is the first paper that gives the (perfect) embedding $\R(n,d) \subseteq \G(n,p_2)$ for some range of $d$) and, subsequently, Gao~\cite{G20} widely extended this range to $d=\Omega(\log^7n)$.

Initially motivated by a paper of Perarnau and Petridis~\cite{PePe} (see Section~\ref{sec_PP}), we consider sandwiching for bipartite graphs, in which the natural counterparts of $\G(n,p)$ and $\R(n,d)$  are random subgraphs of the complete bipartite graph $\Knn$ rather than of $K_n$.

\subsection{New Results}
\label{ss_new_result}
We consider three models of random subgraphs of $\Knn$, the complete bipartite graph with bipartition $(V_1, V_2)$, where $|V_1| = n_1, |V_2| = n_2$. Given an integer $m \in [0, n_1n_2]$, let $\Gnnm$ be an $m$-edge subgraph of $\Knn$ chosen uniformly at random (the bipartite Erd\H os--R\'enyi model). Given a number $p \in [0,1]$, let $\Gnnp$ be the binomial bipartite random graph where each edge of $\Knn$ is included independently with probability~$p$. Note that in the latter model, $pn_{3-i}$ is the expected degree of each vertex in $V_i$, $i=1,2$.
If, in addition,
\[
	d_1 := pn_2\qquad\text{and}\qquad d_2 := pn_1
\]
are integers (we shall always make this implicit assumption), we let $\rcnnp$ be the class of subgraphs of $\Knn$ such that every $v \in V_i$ has degree $d_i$, for $i = 1, 2$ (it is an easy exercise to show that $\rcnnp$ is non-empty). We call such graphs \emph{$p$-biregular}. Let $\Rnnp$ be a random graph chosen uniformly from $\rcnnp$.

In this paper we establish an embedding of $\G(n_1,n_2,m)$ into  $\Rnnp$. This easily implies an embedding of $\Gnnp$ into $\Rnnp$.
	Moreover, by taking complements, our result translates immediately to the opposite embedding of $\Rnnp$ into $\G(n_1,n_2,m)$ (and thus into $\Gnnp$). This idea was first used in~\cite{GIM} to prove $\R(n,d) \subseteq \G(n,p)$ for $p \gg \frac{1}{\sqrt{\log n}}$. In particular, in the balanced case ($n_1=n_2:=n$), we prove this opposite embedding  for $p \gg \left( \log^3 n / n \right)^{1/4}$.

	The proof is far from a straightforward adaptation of the proof in~\cite{DFRSc}. The common aspect shared by the proofs is that the edges of $\Rnnp$ are revealed in a random order, giving a graph process which turns out to be, for most of the time, similar to the basic Erd\H{o}s--R\'enyi process that generates $\Gnnm$. The rest of the current proof is different in that it avoids using the configuration model. Instead, we focus on showing that both $\Rnnp$ and its random $t$-edge subgraphs are pseudorandom. We achieve this by applying the switching method (when $\min \left\{ p, 1-p \right\}$ is small) and otherwise via asymptotic enumeration of bipartite graphs with a given degree sequence proved in~\cite{CGM} (see Theorem~\ref{thm_enum}).

	In addition, for $p > 0.49$, we rely on a non-probabilistic result about the existence of alternating cycles in 2-edge-colored pseudorandom graphs (Lemma~\ref{lem_alternating}), which might be of separate interest.

	Throughout the paper we assume that the underlying complete bipartite graph $\Knn$ grows on both sides, that is, $\min\{n_1,n_2\} \to \infty$, and any parameters (e.g., $p$, $m$), events, random variables, etc., are allowed to depend on $(n_1,n_2)$.
In most cases we will make the dependence on $(n_1,n_2)$ implicit, with all limits and asymptotic notation like $O, \Omega, \sim$ considered with respect to $\min\{n_1,n_2\} \to \infty$. We say that an event $\mathcal E = \mathcal E(n_1,n_2)$  holds \emph{asymptotically almost surely} (\emph{a.a.s.}) if $\prob{\mathcal E} \to 1$.

Our main results are Theorem~\ref{thm_embed} below and its immediate Corollary \ref{cor_embed}. For a gentle start we first state an abridged version of both in the balanced case $n_1 = n_2 = n$. Note that $pn^2$ is the number of edges in $\R(n,n,p)$.
\begin{theorem}
  \label{thm_simple}
	If $p\gg\frac{\log n}{n}$ and $1 - p \gg \left( \frac{\log n}{n} \right)^{1/4}$, then for some $m \sim pn^2$ there is a joint distribution of random graphs $\G(n,n,m)$ and $\R(n,n,p)$ such that
    \begin{equation}
      \label{eq:1stemb}
	    \G(n,n,m) \subseteq \R(n,n,p) \qquad \text{a.a.s.}
    \end{equation}
    If $p\gg\left(\frac{\log^3 n}{n}\right)^{1/4}$, then for some $m \sim pn^2$ there is a joint distribution of random graphs $\G(n,n,m)$ and $\R(n,n,p)$ such that
    \begin{equation}
      \label{eq:2ndemb}
	    \R(n,n,p) \subseteq \G(n,n,m)  \qquad \text{a.a.s.}
    \end{equation}
    Moreover, in \eqref{eq:1stemb} and \eqref{eq:2ndemb} one can replace $\G(n,n,m)$ by the binomial random graph $\G(n,n,p')$, for some $p' \sim p$.
\end{theorem}
The condition $p \gg (\log n)/n$ is necessary for~\eqref{eq:1stemb} to hold with $m \sim pn^2$ (see Remark~\ref{rem_best_gamma}), since otherwise the maximum degree of $\G(n,n,m)$ is no longer $pn(1+o(1))$ a.a.s. We guess that the maximum degree is, vaguely speaking, the only obstacle for embedding $\G(n,n,m)$ (or $\G(n,n,p')$) into $\Rnnp$, see Conjecture~\ref{conj} in Section~\ref{sec_concluding}.

We further write $N := n_1n_2$, $q := 1 - p$, $\hat n := \min \left\{ n_1, n_2 \right\}$, $\hat p := \min \{p, q\}$, and let
\begin{equation}
  \label{eq_defI}
	\I :=\I(n_1,n_2,p)= \begin{cases}
		1, & \hat p < 2\frac{n_1n_2^{-1} + n_1^{-1}n_2}{\log N}\\
		0, & \hat p \ge 2\frac{n_1n_2^{-1} + n_1^{-1}n_2}{\log N}.
	\end{cases}
\end{equation}
Note that $\I=0$ entails that the vertex classes are rather balanced: the ratio of their sizes cannot exceed $\tfrac 1 4 \log N$. Moreover, $\I=0$ implies that $\hat p \ge 4/\log N$.

\begin{theorem}
  \label{thm_embed}
  For every constant $C > 0$, there is a constant $C^*$ such that whenever the parameter $p \in [0,1]$ satisfies
  \begin{equation}
\label{eq:thm_embed_q}
q \ge 680\left( \frac{3(C + 4)\log N}{\hat n}\right)^{1/4}
  \end{equation}
  and
	\begin{equation}
	  \label{eq_gamma}
1\ge\gamma :=
	\begin{cases}
		C^* \left( p^2 \I + \sqrt{\frac{\log N}{p\hat n }}\right), \quad & p \le 0.49 \\
	C^* \left( q^{3/2}\I + \left(\frac{\log N}{\hat n} \right)^{1/4} + \frac1q\sqrt{\frac{\log N}{\hat n}} \log \frac{\hat n}{\log N}\right), \quad & p > 0.49,
	\end{cases}
\end{equation}
there is, for $m := \lceil(1-\gamma)pN\rceil$, a joint distribution of random graphs $\Gnnm$ and $\Rnnp$ such that
\begin{equation}
  \label{eq_lower_embed}
	    \prob{\Gnnm \subseteq \Rnnp} = 1 - O(N^{-C}).
    \end{equation}
    If, in addition $\gamma \le 1/2$, then for $p' := (1 - 2\gamma)p$,
	there is a joint distribution of $\G(n_1,n_2,p')$ and $\Rnnp$ such that
\begin{equation}
\label{eq:embed_Gnp}
  \prob{\G(n_1,n_2,p') \subseteq \Rnnp} = 1 - O(N^{-C}).
    \end{equation}
\end{theorem}
 By inflating $C^*$, the constant $0.49$ in~\eqref{eq_gamma} can be replaced by any constant smaller than $1/2$.
It can be shown (see Remark~\ref{rem_optimality} in Section~\ref{sec_concluding}) that if $p \le 1/4$ and $\gamma = \Theta\left( \sqrt{\frac{\log N}{p\hat n }} \right)$, then $\gamma$ has optimal order of magnitude. For more remarks about the conditions of Theorem~\ref{thm_embed} and the role of the indicator $\I$, see Section~\ref{sec_concluding}.

\medskip
By taking the complements of $\Rnnp$ and $\G(n_1,n_2,m)$ and swapping $p$ and $q$, we immediately obtain the following consequence of Theorem~\ref{thm_embed} which provides the opposite embedding.

\begin{corollary}
  \label{cor_embed}
	For every  constant $C > 0$ there is a constant $C^*$ such that whenever the parameter $p \in [0,1]$ satisfies
    \begin{equation}
\label{eq:cor_embed_q}
      p \ge 680\left( 3(C + 4)\hat n^{-1}\log N\right)^{1/4}
    \end{equation}
 and
	\begin{equation*}
		1 \ge \bar\gamma:=
		\begin{cases}
			C^*\left( p^{3/2}\I +  \left(\frac{\log N}{\hat n} \right)^{1/4} + \frac 1p\sqrt{\frac{\log N}{\hat n}} \log \frac{\hat n}{\log N}\right), \quad  p < 0.51, \\
			C^*\left( q^2 \I + \sqrt{\frac{\log N}{q\hat n}}\right), \quad  p \ge 0.51,
		\end{cases}
	\end{equation*}
there is, for $\bar m=\lfloor(p+\bar\gamma q)N\rfloor$, a joint distribution of random graphs $\G(n_1,n_2,\bar m)$ and $\Rnnp$ such that
\begin{equation}
  \label{eq_inverse_embed}
	    \prob{\Rnnp \subseteq \G(n_1,n_2,\bar m)}=1 - O(N^{-C}).
    \end{equation}
    If, in addition, $\bar \gamma \le 1/2$, then  for $p'' := (p + 2\bar{\gamma} q)N$
	there is a joint distribution of $\G(n_1,n_2,p'')$ and $\Rnnp$ such that
\begin{equation}
\label{eq:inverse_Gnp}
  \prob{\Rnnp \subseteq \G(n_1,n_2,p'')} = 1 - O(N^{-C}).
    \end{equation}
\end{corollary}

\begin{proof}
	The assumptions of Corollary \ref{cor_embed} yield the assumptions of Theorem~\ref{thm_embed} with $\gamma=\bar\gamma$ and with $p$ and $q$ swapped. Note also that
	\[
	  m = \lceil(1-\bar\gamma)qN\rceil =\lceil N-pN-\bar\gamma q N\rceil= N - \lfloor(p+\bar\gamma q)N\rfloor=N-\bar m.
	\]
	Thus, by Theorem~\ref{thm_embed}, with probability $1 - O(N^{-C})$ we have $\G(n_1,n_2,N - \bar m) \subseteq \R(n_1,n_2,q)$, which, by taking complements, translates into~\eqref{eq_inverse_embed}. Similarly
	\[
	  p' = (1 - 2\bar \gamma)q N = N - (p + 2\bar \gamma q) N = N - p'' N.
	\]
	Thus, by Theorem~\ref{thm_embed}, with probability $1 - O(N^{-C})$ we have $\G(n_1,n_2,N-p''N) \subseteq \R(n_1,n_2,q)$, which, by taking complements yields embedding~\eqref{eq:inverse_Gnp}.
\end{proof}

\medskip

\begin{proof}[Proof of Theorem~\ref{thm_simple}]
	We apply Theorem~\ref{thm_embed} and Corollary \ref{cor_embed} with $C=1$ and the corresponding~$C^*$. Note that for $n_1=n_2=n$, the ratio in~\eqref{eq_defI} equals  $4/\log N = 2/\log n$, so
\begin{equation}
  \label{I=1}
\I=1\quad\text{ if and only if }\quad\hat p<2/\log n.
\end{equation}

  To prove~\eqref{eq:1stemb}, assume $p\gg\log n/n$ and $q \gg \left( \log n/n \right)^{1/4}$ and apply Theorem~\ref{thm_embed}.
  Note that  condition~\eqref{eq:thm_embed_q} holds and it is straightforward to check that, regardless of whether $p \le 0.49$ or $p > 0.49$, $\gamma \to 0$. In particular, $\gamma \le 1/2$. We conclude that, indeed, embedding~\eqref{eq:1stemb} holds with $m = \lceil (1 - \gamma)p N \rceil \sim pn^2$ and~\eqref{eq:1stemb} still holds if we replace $\G(n,n,m)$ by $\G(n,n,p')$ with $p' = (1 - 2\gamma)p \sim p$.

  For~\eqref{eq:2ndemb} first note that when $p \to 1$, embedding~\eqref{eq:2ndemb} holds trivially with $m = n^2$ (even though a nontrivial embedding follows in this case under an additional assumption $q \gg \log n/n$. Hence, we further assume $q = \Omega(1)$, $p\gg\left(\log^3 n / n\right)^{1/4}$ and apply Corollary~\ref{cor_embed}. Note that condition~\eqref{eq:cor_embed_q} holds. It is routine to check, taking into account~\eqref{I=1}, that $\bar\gamma\le 1/2$ and, moreover, $\bar\gamma q = o(p)$. We conclude that~\eqref{eq:2ndemb} holds with $m = \lfloor(p + \bar\gamma q)N \rfloor \sim Np = n^2p$ and~\eqref{eq:2ndemb} still holds if $\G(n,n,m)$ is replaced by $\G(n,n,p')$ with $p' = p + 2 \bar \gamma q \sim p$.
\end{proof}

\subsection{A note on the second version of the manuscript.}
\label{non}
This project was initially aimed at extending the result in~\cite{DFRSc} to bipartite graphs and, thus,  limited to the lower embedding $\G(n,n,p_1) \subseteq \R(n,n,p)$ only. While it was in progress, Gao, Isaev and McKay~\cite{GIM} made an improvement on the Sandwich Conjecture by using a surprisingly fruitful idea of taking complements to obtain the upper embedding $\R(n,d) \subseteq \G(n,p_2)$ directly from the lower embedding $\G(n,p_1) \subseteq \R(n,d)$. We then decided to borrow this idea (but nothing else) and strengthen some of our lemmas to get a significantly broader range of $p$ for which the upper embedding (i.e. Corollary~\ref{cor_embed}) holds. It turned out that our approach works for \emph{non-bipartite} regular graphs, too. Therefore, prompted by the recent substantial progress of Gao~\cite{G20} on the Sandwich Conjecture (which appeared on arXiv after the first version of this manuscript), in the current version of the manuscript we added Section~\ref{sec_extension}, which outlines how to modify our proofs to get a corresponding sandwiching for non-bipartite graphs. This improves upon the results in~\cite{GIM} (for regular graphs), but is now superseded by~\cite{G20}.

\subsection{Organization}
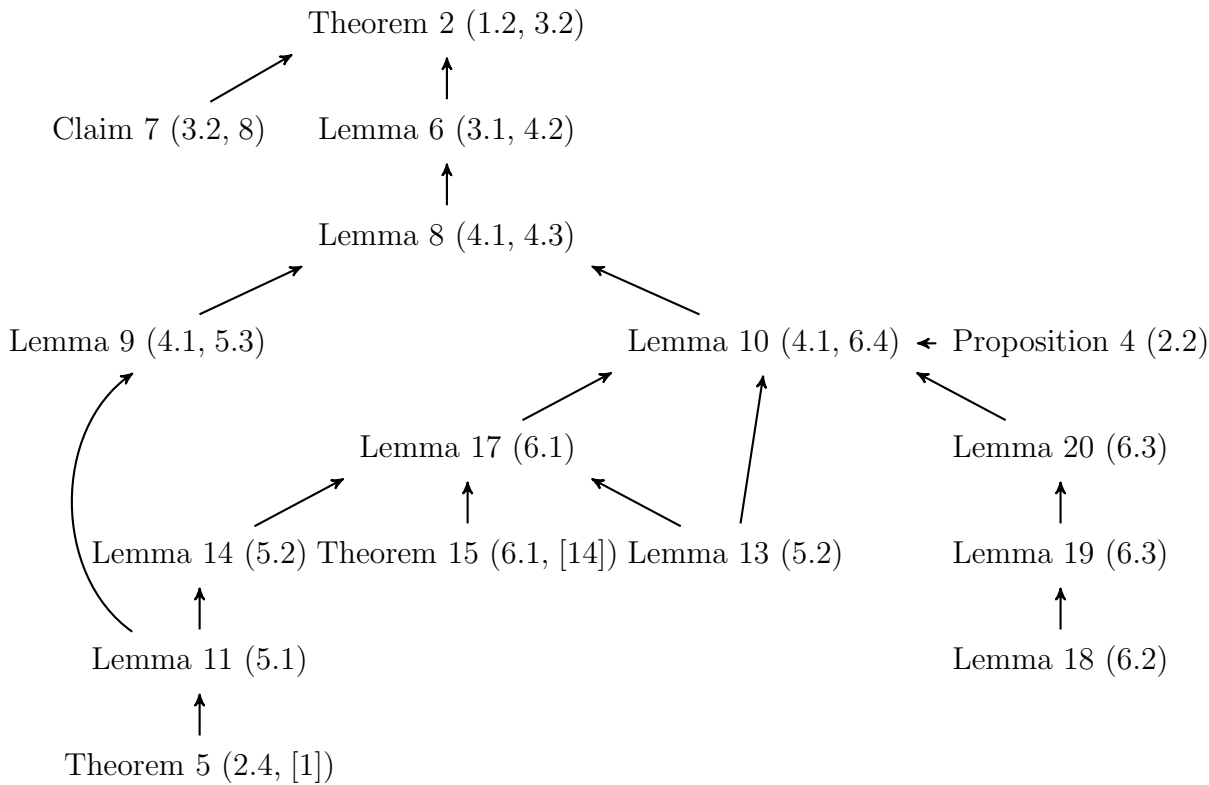
\begin{figure}
  \begin{tikzpicture}[node distance=7mm and 4mm, auto]
    \node (Thm2) {Theorem~\ref{thm_embed}~(\ref{ss_new_result},\;\ref{ss_proof_embed})};
    \node[below=of Thm2] (L8) {Lemma~\ref{lem_Stanislaw}~(\ref{ss_crucial_setup},\;\ref{ss_Stanislaw_proof})}
    edge[pil] (Thm2.south);
    \node[left=of L8] (Cl9) {Claim~\ref{clm_sumofgammas}~(\ref{ss_proof_embed},\;\ref{sec_technical})}
    edge[pil] (Thm2.south west);
    \node[below=of L8] (L10) {Lemma~\ref{lem_typical}~(\ref{ss_prep},\;\ref{ss_thetas})}
    edge[pil] (L8.south);
    \node[below left=of L10] (L11) {Lemma~\ref{lem_thetas}~(\ref{ss_prep},\;\ref{ss_proof_thetas})}
    edge[pil] (L10.south west);
    \node[below right=of L10] (L12) {Lemma~\ref{lem_ef}~(\ref{ss_prep},\;\ref{ss_ef_proof})}
    edge[pil] (L10.south east);
    \node[right=of L12] (Cl5) {Proposition~\ref{prop_switch}~(\ref{ss_switchings})}
    edge[pil] (L12.east);
    \node[below right=of L12] (Cor21) {Lemma~\ref{lem_alternating}~(\ref{ss_alt_walks_cycle})}
    edge[pil] (L12.south east);
    \node[below=of Cor21] (L20) {Lemma~\ref{lem_walks}~(\ref{ss_alt_walks_cycle})}
    edge[pil] (Cor21.south);
    \node[below=of L20] (L19) {Lemma~\ref{lem_ananas_prep}~(\ref{ss_blue_red})}
    edge[pil] (L20.south);

    \node[below left=of L12] (L18) {Lemma~\ref{lem_quasiRt}~(\ref{ss_jumbled})}
    edge[pil] (L12.south west);
    \node[below=of L18] (Thm16) {Theorem~\ref{thm_Thomason}~(\ref{ss_jumbled},\;\cite{T89})}
    edge[pil] (L18.south);
    \node[below right=of L18] (L14) {Lemma~\ref{lem_degrees}~(\ref{ss_Rt_cod})}
    edge[pil] (L12.south)
    edge[pil] (L18.south east);

    \node[below left=of L18] (L15) {Lemma~\ref{lem_codegsRt}~(\ref{ss_Rt_cod})}
    edge[pil] (L18.south west);

    \node[below=of L15] (L13) {Lemma~\ref{lem_codegs}~(\ref{ss_codegs})}
    edge[pil] (L15.south)
    edge[pil, bend left=55] (L11.south);
    \node[below=of L13] (Thm7) {Theorem~\ref{thm_enum}~(\ref{ss_count},\;\cite{CGM})}
    edge[pil] (L13.south);

  \end{tikzpicture}
  \caption{The structure of the proof of Theorem~\ref{thm_embed}. An arrow from  statement A to statement B means that A is used in the proof of B. The numbers in the brackets point to the section where a statement is formulated and where it is proved (unless the proof follows the statement immediately); external results have instead an article reference in square brackets.}
  \label{fig_flowchart}
\end{figure}
 In Section~\ref{sec_tools} we introduce the notation and tools used throughout the paper: the switching technique, probabilistic inequalities, and an enumeration result for bipartite graphs with a given degree sequence. In Section~\ref{sec_crucial} we state a crucial Lemma~\ref{lem_Stanislaw} and show how it implies Theorem~\ref{thm_embed}.

 In Section~\ref{sec_lemmaproof} we  give a proof of Lemma~\ref{lem_Stanislaw} based on two technical lemmas, one about the concentration of a degree-related parameter (Lemma~\ref{lem_thetas}), the other (Lemma~\ref{lem_ef}) facilitating the switching technique used in the proof of Lemma~\ref{lem_Stanislaw}. Lemma~\ref{lem_thetas} is proved in Section~\ref{sec_degs_codegs}, after giving some auxiliary results establishing the concentration of degrees and co-degrees in $\Rnnp$ as well as in its conditional versions.  In Section~\ref{sec_quasi}, we present a proof of Lemma~\ref{lem_ef}, preceded by a purely deterministic result about alternating cycles in 2-edge-colored pseudorandom graphs.
 We defer some technical but straightforward results and their proofs (e.g., the proof of Claim \ref{clm_sumofgammas}) to Section~\ref{sec_technical}.
 A flowchart of the results ultimately leading to the proof of Theorem~\ref{thm_embed} is presented in Figure~\ref{fig_flowchart}.

	The contents of Section \ref{sec_extension} were already described (see Subsection \ref{non} above).  Section~\ref{sec_PP} contains an application of our main Theorem~\ref{thm_embed}, which was part of the motivation for our research.
In Section~\ref{sec_concluding} we present some concluding remarks and our version of the (bipartite) sandwiching conjecture.

\subsection{Acknowledgement}
 We thank Noga Alon and Benny Sudakov who, at the conference Random Structures \& Algorithms 2019, suggested a way to show the existence of alternating paths in non-bipartite graphs. We are also thankful to the anonymous referee for useful remarks.

\section{Preliminaries}
\label{sec_tools}

\subsection{Notation}
Recall that
\[
	d_1 = pn_2\qquad\text{and}\qquad d_2 = pn_1
\]
are the degrees of vertices in a $p$-biregular graph, and thus, the number of edges in any $p$-biregular graph $H \in \rcnnp$ is
\[
	M:=pN\qquad\text{where}\qquad N=n_1n_2.
\]
Throughout the proofs we also use shorthand notation
\begin{equation}
  \label{eq_not}
	q = 1 - p, \quad \hat p = \min \left\{ p, q \right\}, \quad \hat n = \min \left\{ n_1,n_2 \right\},
\end{equation}
and $[n] := \left\{ 1, \dots, n \right\}$. All logarithms appearing in this paper are natural.

By $\Gamma_{G}(v)$ we denote the set of neighbors of a vertex $v$ in a graph $G$.

\subsection{Switchings}
\label{ss_switchings}
The switching technique is used to compare the size of two classes of graphs, say $\rc$ and $\rc'$, by defining an auxiliary bipartite graph $B:=B(\rc,\rc')$, in which two graphs $H \in \rc$, $H' \in \rc'$ are connected by an edge whenever $H$ can be transformed into $H'$ by some operation (a \emph{forward switching}) that deletes and/or creates some edges of $H$.
By counting the number of edges of $B(\rc,\rc')$ in two ways, we see that
\begin{equation}
  \label{eq_BRR}
	\sum_{H\in\rc}\deg_B(H) = \sum_{H'\in\rc'}\deg_B(H'),
\end{equation}
which easily implies that
\begin{equation}
  \label{eq_countingedges}
\frac{\min_{H' \in \rc'} \deg_B(H')}{\max_{H \in \rc} \deg_B(H)} \le \frac{|\rc|}{|\rc'|} \le \frac{\max_{H' \in \rc'} \deg_B(H')}{\min_{H \in \rc} \deg_B(H)}.
\end{equation}

The reverse operation mapping $H' \in \rc'$ to its neighbors in graph $B$, is called \emph{a backward switching}. Usually, one defines the forward switching in such a way that the backward switching can be easily described.

All switchings used in this paper follow the same pattern. For a fixed  graph $G \subseteq K$ (possibly empty), where $K := \Knn$, the families $\rc, \rc'$ will be subsets of
\begin{equation}
  \label{eq_rcG}
  \rc_G := \{ H \in \rc(n_1,n_2,p) : G \subseteq H \}.
\end{equation}
Every  $H \in \rc_G$ will be interpreted as a blue-red coloring of the edges of $K\setminus G$: those in $H\setminus G$ will be colored \emph{blue} and those in $K \setminus H$ --- \emph{red}. Given $H\in\rc$, consider a subset $S$ of the edges of $K\setminus G$ in which for every vertex $v \in V_1 \cup V_2$ the blue degree equals the red degree, i.e., $\deg_{(H\setminus G)\cap S}(v)=\deg_{(K\setminus H)\cap S}(v)$. Then switching the colors within $S$ produces another graph $H' \in \rc_G$. Formally, $E(H')$ is the symmetric difference $E(H) \triangle S$.
In each application of the switching technique, we will restrict the choices of $S$ to make sure that $H' \in \rc'$.

As an elementary illustration of this technique, which nevertheless turns out to be useful in Section~\ref{sec_quasi}, we prove here the following result. A cycle in $K\setminus G$ is \emph{alternating} if it is a union of a blue matching and a red matching. Note that the definition depends on $H$. We will omit mentioning this dependence, as $H$ will always be clear from the context.
 Given $e \in K\setminus G$, set
\begin{equation}
  \label{eq_rcGe}
	\rc_{G,e} := \left\{ H \in \rc_G : e \in H \right\}\qquad\text{and }\qquad \rc_{G,\neg e} := \left\{ H \in \rc_G : e \not\in H \right\}.
\end{equation}

\begin{proposition}
  \label{prop_switch}
  Let a graph $G\subseteq K$ be such that $\rc_G\neq\emptyset$ and let $e\in K\setminus G$. Assume that for some number $D > 0$ and every $H\in \rc_{G}$ the edge $e$ is contained in an alternating cycle of length at most $2D$. Then $\rc_{G,\neg e}\neq\emptyset$, $\rc_{G, e}\neq\emptyset$, and
	\[
		\frac{1}{N^D - 1} \le \frac{|\rc_{G,\neg e}|}{|\rc_{G,e}|} \le N^D - 1.
	\]
\end{proposition}

\begin{proof}
  Let $B = B(\rc_{G,e}, \rc_{G,\neg e})$ be the switching graph corresponding to the following forward switching: choose an alternating cycle $S$ of length at most $2D$ containing $e$ and switch the colors of edges within $S$. Note that the backward switching does precisely the same.
Note that by  the assumption, the minimum degree $\delta(B)$ is at least $1$.

Since  $\rc_G=\rc_{G, e}\cup\rc_{G, \neg e} \neq \emptyset$, one of the classes $\rc_{G, e}$ and $\rc_{G, \neg e}$ is non-empty and, in view of $\delta(B) \ge 1$, the other one is non-empty as well. For $\ell = 2, \dots, D$, the number of cycles of length $2\ell$ containing $e$ is, crudely, at most $n_1^{\ell - 1} n_2^{\ell - 1} = N^{\ell - 1}$, hence the maximum degree is
	\[
		\Delta(B)\le N + N^2 + \cdots + N^{D - 1} \le N^D - 1.
	\]
Thus, by~\eqref{eq_countingedges}, we obtain
 the claimed bounds. \end{proof}

\subsection{Probabilistic inequalities}
\label{ss_probineq}

We first state a few basic concentration inequalities that we  apply in our proofs.
For a sum $X $ of independent Bernoulli (not necessarily identically distributed) random variables, writing $\mu = \E X$, we have (see Theorem 2.8, (2.5), (2.6), and~(2.11) in~\cite{JLR} that
	\begin{equation}
		\label{eq_hypergeomUpTailExp}
		\prob{X \ge \mu + t} \le \exp \left\{ - \frac{t^2}{2\left( \mu + t/3 \right)} \right\}, \qquad t \ge 0,
	\end{equation}
\begin{equation}
		\label{eq_hypergeomLowTailExp}
		\prob{X \le \mu - t} \le \exp \left\{ - \frac{t^2}{2\mu} \right\}, \qquad t \ge 0.
	\end{equation}

	Let $\Gamma$ be a set of size $|\Gamma|=g$ and let  $A \subseteq \Gamma$,  $|A|=a \ge1$. For an integer $r \in [0,g]$, choose uniformly at random a subset $R \subseteq \Gamma$ of size $|R|=r$. The random variable $Y = |A \cap R|$ has then the \emph{hypergeometric} distribution $\Hyp(g, a, r)$ with expectation $\mu := \E Y = ar/g$.
By Theorem 2.10 in~\cite{JLR}, inequalities~\eqref{eq_hypergeomUpTailExp} and~\eqref{eq_hypergeomLowTailExp} hold for $Y$, too.

Moreover, by Remark 2.6 in~\cite{JLR}, inequalities~\eqref{eq_hypergeomUpTailExp} and~\eqref{eq_hypergeomLowTailExp} also hold for a random variable $Z$ which has Poisson distribution $\Po(\mu)$ with expectation $\mu$.
In this case, we also have the following simple fact.
For $k\ge0$, set $q_k=\prob{Z=k}$. Then $q_k/q_{k-1}=\mu/k$, and hence $k = \lfloor\mu\rfloor$ maximizes $q_k$ (we say that such $k$ is a \emph{mode} of $Z$). Since $\Var Z = \mu$, by Chebyshev's inequality, $\prob{|Z-\mu|<\sqrt{2\mu}}\ge1/2$. Moreover, the interval $(\mu - \sqrt{2 \mu}, \mu + \sqrt{2 \mu})$ contains at most $\lceil \sqrt{8\mu}\:\rceil$ integers, hence it follows that
\begin{equation}
  \label{eq_largestatom}
  q_{\lfloor \mu \rfloor} \ge \frac{1/2}{\lceil \sqrt{8\mu}\: \rceil}.
\end{equation}

\subsection{Asymptotic enumeration of dense bipartite graphs}
\label{ss_count}

To estimate co-degrees of $\R(n,n,p)$ we will use the following asymptotic formula by Canfield, Greenhill and McKay~\cite{CGM}. We reformulate it slightly for our convenience.

Given two vectors $\dd_1 = (d_{1,v}, v \in V_1)$ and $\dd_2 = (d_{2,v}, v \in V_2)$ of positive integers such that $\sum_{v \in V_1} d_{1, v} = \sum_{v \in V_2} d_{2,v}$, let $\rcdd$ be the class of bipartite graphs on $(V_1, V_2)$ with vertex degrees $\deg(v) = d_{i,v}, v \in V_i, i = 1,2$.
Let $|V_i| = n_i$ and write $\bar d_i = {n_i}^{-1}\sum_{v \in V_i} d_{i,v}$, $D_i = \sum_{v \in V_i} (d_{i,v} - \bar d_i)^2$,
$p = \bar d_1/n_2 = \bar d_2/n_1$, and $q = 1 - p$.
\begin{theorem}[\cite{CGM}]\label{thm_enum}
	Given any positive constants $a, b, C$ such that $a+b < 1/2$, there exists a constant $\epsilon > 0$ so that the following holds. Consider the set of degree sequences $\dd_1,\dd_2$ satisfying
\begin{romenumerate}
	\item\label{thm_enum_max} $\max_{v \in V_1}|d_{1,v} - \bar d_1| \le Cn_2^{1/2 + \epsilon}$, $\max_{v \in V_2} |d_{2,v} - \bar d_2| \le Cn_1^{1/2 + \epsilon}$
	\item\label{thm_enum_balance} $\max \{n_1, n_2\} \le C(pq)^2(\min \left\{ n_1, n_2 \right\})^{1 + \epsilon}$
\item\label{thm_enum_dense}
	$  \frac{(1 - 2p)^2}{4pq} \left( 1 + \frac{5n_1}{6n_2} + \frac{5n_2}{6n_1} \right) \le a \log \max \left\{ n_1, n_2 \right\}$.
\end{romenumerate}
If $\min\{n_1, n_2\} \to \infty$, then uniformly for all such $\dd_1,\dd_2$
\begin{multline}
  \label{eq_enum}
	|\rc(\dd_1,\dd_2)| = \binom{n_1n_2}{p n_1n_2}^{-1} \prod_{v \in V_1}\binom {n_2}{d_{1,v}} \prod_{v \in V_2}\binom {n_1}{d_{2,v}} \times \\
	\times \exp \left[ -\frac{1}{2}\left( 1- \frac{D_1}{pqn_1n_2} \right)\left(1- \frac{D_2}{pqn_1n_2} \right) + O\left((\max\left\{ n_1,n_2 \right\})^{-b}\right) \right],
\end{multline}
where the constants implicit in the error term may depend on $a,b,C$.
\end{theorem}
Note that condition~\ref{thm_enum_balance} of Theorem~\ref{thm_enum} implies the corresponding condition in~\cite{CGM} after adjusting~$\epsilon$.
Also, the uniformity of the bound is not explicitly stated in~\cite{CGM}, but, given $n_1, n_2$, one should take $\dd_1, \dd_2$ with the worst error and apply the result in~\cite{CGM}.

\section{A Crucial Lemma}
\label{sec_crucial}

\subsection{The set-up}
\label{ss_crucial_setup}
Recall that $K = \Knn$ has $N = n_1n_2$ edges.
Consider a sequence of graphs $\G(t) \subseteq K$ for $t = 0, \dots , N$, where $\G(0)$ is empty and, for $t < N$, $\G(t + 1)$ is obtained from $\G(t)$ by adding an edge $\eps_{t + 1}$ chosen from $K \setminus \G(t)$ uniformly at random, that is,
for every graph $G \subseteq K$ of size $t$ and  every edge $e \in K \setminus G$
\begin{equation}
  \label{eq_PG}
\Pc{\varepsilon_{t+1} = e}{\G(t) = G} = \frac{1}{N - t}.
\end{equation}
Of course, $(\varepsilon_1, \dots, \varepsilon_N)$ is just a uniformly random ordering of the edges of $K$.

Our approach to proving Theorem~\ref{thm_embed} is to represent the random regular graph $\Rnnp$ as the
outcome of a random process which behaves similarly to $(\G(t))_t$. Recalling that a $p$-biregular graph has $M = pN$ edges, let
\begin{equation*}
(\eta_1, \dots, \eta_M)
\end{equation*}
be a uniformly random ordering of the edges of $\Rnnp$.
By taking  the initial segments, we obtain a sequence of random graphs
\begin{equation*}
	\R(t) = \{\eta_1, \dots, \eta_t\}, \qquad t = 0, \dots, M.
\end{equation*}
For convenience, we shorten
\[
  \R:= \R(M) = \Rnnp\;.
\]

Let us mention here that for a fixed $H \in \rcnnp$, conditioning on $\R = H$, the edge set of the random subgraph  $\R(t)$ is a uniformly random $t$-element subset of the edge set of~$H$. This observation often leads to a hypergeometric distribution and will be utilized several times in our  proofs.

We say that a graph $G$ with $t$ edges is \emph{admissible}, if the family $\rc_G$ (see definition~\eqref{eq_rcG}) is nonempty, or, equivalently,
\begin{equation*}
	\prob{\R(t) = G} > 0.
\end{equation*}
For an admissible graph $G$ with $t$ edges and any edge $ e \in K \setminus G$, let
\begin{equation}
  \label{eq_pt}
p_{t+1}(e,G) := \Pc{\eta_{t+1} = e}{\R(t) = G}.
\end{equation}
The conditional space underlying~\eqref{eq_pt} can be described as first extending $G$ uniformly at random to an element of $\rc_G$ and then randomly permuting the new $M-t$ edges.

The main idea behind the proof of Theorem~\ref{thm_embed} is that the conditional probabilities in~\eqref{eq_pt} behave similarly to those in~\eqref{eq_PG}. Observe that $p_{t+1}(e,\R(t)) = \Pc{\eta_{t+1} = e}{\R(t)}$.
Given a real number $\chi \ge 0$ and $t \in \{0, \dots, M - 1\}$, we define an $\R(t)$-measurable event
\begin{equation*}
	\ac(t,\chi) := \left\{ p_{t+1}(e,\R(t)) \ge \frac{1 - \chi}{N - t} \text{ for every } e \in K \setminus \R(t) \right\}.
\end{equation*}
In the crucial lemma below, we are going to show that for suitably chosen $\gamma_0, \gamma_1, \dots$, a.a.s.\ the events $\ac(t,\gamma_t)$ occur simultaneously for all $t = 0,\dots,t_0-1$ where $t_0$ is quite close to $M$. Postponing the choice of $\gamma_t$ and $t_0$, we define an event
\begin{equation}
  \label{eq_ac}
	\ac := \bigcap_{t = 0}^{ t_0 - 1}\ac(t, \gamma_t).
\end{equation}
  Intuitively, the event $\ac$ asserts that up to time $t_0$ the process $\R(t)$ stays ``almost uniform'', which will enable us to embed $\Gnnm$ into $\R(t_0)$.

  To define time $t_0$, it is convenient to parametrize the time the by the proportion of edges of $\R$ ``not yet revealed'' after $t$ steps. For this, we define by
\begin{equation}
  \label{eq_tau}
  \tau = \tau(t) := 1 - \frac{t}{M} \in [0,1]\qquad\text{and so}\qquad t = (1 - \tau)M.
\end{equation}
Given a constant $C > 0$ and, we define (recalling the notation in~\eqref{eq_not}) the ``final'' value $\tau_0$ of~$\tau$~as
\begin{equation}
  \label{eq_tauzero}
	\tau_0 := \begin{cases}
		 \frac{3 \cdot 3240^2(C + 4)\log N}{p \hat n}, \quad & p \le 0.49, \\
		 700(3(C + 4))^{1/4} \left( q^{3/2}\I + \left(\frac{\log N}{\hat n}\right)^{1/4}\right),
		\quad & p > 0.49.
	\end{cases}
\end{equation}
(Some of the constants appearing here and below are sharp or almost sharp, but others have room to spare as we round them up to the nearest ``nice'' number.)

Consider the following assumptions on $p$ (which we will later show to follow from the assumptions of Theorem~\ref{thm_embed}):

\begin{equation}
	 \hat p \ge \frac{3 \cdot 3240^2(C+4)\log N}{\hat n},	
	\label{eq_phat_lower}
\end{equation}
\begin{equation}
  \label{eq_hatpI_ubound}
	\hat p \cdot \I \le \frac{49}{51}\cdot\frac1{340^2 (C + 4)^{1/6}},
	\end{equation}
and
\begin{equation}
  \label{eq_ef_qlbound}
  q \ge 680\left( \frac{3(C+4)\log N}{\hat{n}}\right)^{1/4}.
	\end{equation}
At the end of this subsection we show that these three assumptions imply
\begin{equation}
  \label{eq_tau_less}
	\tau_0 \le 1,
\end{equation}
so that
\begin{equation*}
	t_0 := \lfloor (1 - \tau_0)M \rfloor
\end{equation*}
is a non-negative integer.
Further, for $t = 0, \dots, M - 1$, define
\begin{equation}
  \label{eq_gamma_t}
	\gamma_t := 1080 \hat p^2 \I +
	\begin{cases}
		3240 \sqrt{\frac{2(C + 3)\log N}{\tau p \hat n}}, \quad & p \le 0.49, \\
	 25000  \sqrt{\frac{(C + 3)\log N}{ \tau^2q^2 \hat n}}
	, \quad & p > 0.49.
	\end{cases}
\end{equation}

Taking~\eqref{eq_tau_less}
for granted, we now state our crucial lemma, which is proved in Section~\ref{sec_lemmaproof}.

\begin{lemma}
\label{lem_Stanislaw}
For every constant $C > 0$, if assumptions~\eqref{eq_phat_lower},~\eqref{eq_hatpI_ubound}, and~\eqref{eq_ef_qlbound} hold, then
\begin{equation}
  \label{eq_largeprob}
  \prob{\ac} = 1- O(N^{-C}),
\end{equation}
where the constant implicit in the $O$-term in~\eqref{eq_largeprob} may also depend on $C$.
\end{lemma}

\medskip

It remains to show~\eqref{eq_tau_less}.
When $p \le 0.49$, inequality~\eqref{eq_tau_less} is equivalent to~\eqref{eq_phat_lower}. For $p > 0.49$, we have $q \le 51\hat p/49$, which together with assumptions~\eqref{eq_hatpI_ubound} and~\eqref{eq_ef_qlbound} implies that
\[
	\tau_0 \le 700 \left( 3(C+4) \right)^{1/4} \left( \frac{51}{49}\hat p \cdot \I \right)^{3/2} + 700 \cdot \frac{q}{680} \le \frac{700 \cdot 3^{1/4}}{340^3} + \frac{700 \cdot 0.51}{680} < 1.
\]

\medskip

\subsection{Proof of Theorem~\ref{thm_embed}}
\label{ss_proof_embed}

From Lemma~\ref{lem_Stanislaw} we are going to deduce Theorem~\ref{thm_embed} using a coupling argument similar to the one which was employed by Dudek, Frieze, Ruci\'nski and \v{S}ileikis~\cite{DFRSc}, but with an extra tweak (inspired by Kim and Vu~\cite{KV04}) of letting the probabilities~$\gamma_t$ of Bernoulli random variables depend on $t$, which reduces the error $\gamma$ in~\eqref{eq_gamma}.

It is easy to check that the assumptions of Lemma~\ref{lem_Stanislaw} follow from the assumptions of Theorem~\ref{thm_embed}. Indeed,~\eqref{eq_ef_qlbound} coincides with assumption~\eqref{eq:thm_embed_q}, while~\eqref{eq_phat_lower} and~\eqref{eq_hatpI_ubound} follow from the assumption $\gamma \le 1$ (see~\eqref{eq_gamma}) for sufficiently large $C^*$.

\medskip

Our aim is to couple $(\G(t))_t$ and $(\R(t))_t$ on the event $\ac$ defined in~\eqref{eq_ac}. For this we will define a graph process $\R'(t) := \{\eta'_1, \dots, \eta'_t\}, t = 0, \dots, t_0$ so that for every admissible graph~$G$ with $t \in [0, M -1]$ edges and every $e \in K \setminus G$
\begin{equation}\label{eq_etaprime}
    \Pc{\eta'_{t+1} = e}{\R'(t) = G} = p_{t+1}(e, G),
\end{equation}
where $p_{t+1}(e, G)$ was defined in~\eqref{eq_pt}.
Note that $\R'(0)$ is an empty graph. Since the distribution of the process $(\R(t))_t$ is determined by the conditional probabilities~\eqref{eq_pt}, in view of~\eqref{eq_etaprime}, the distribution of $\R'(t_0)$ is the same as that of $\R(t_0)$ and therefore we will identify $\R'(t_0)$ with $\R(t_0)$. As the second step, we will show
that a.a.s.\ $\Gnnm$ can be sampled from $\R'(t_0) = \R(t_0)$.

Proceeding with the definition, set $\R'(0)$ to be the empty graph and define graphs $\R'(t)$, $t = 1, \dots, t_0$, inductively, as follows. 
Hence, let us further
fix $t \in [0, t_0-1]$ and suppose that
\[
  \R'(t)=R_t\quad\mbox{and}\quad \G(t)=G_t
\]
have been already chosen.
Our immediate goal is to select  a random pair of edges $\eps_{t+1}$ and $\eta_{t+1}'$, according to, resp.,~\eqref{eq_PG} and~\eqref{eq_etaprime}, in such a way that the event $\eps_{t+1}\in \R'(t+1)$ is quite likely.

To this end, draw $\eps_{t + 1}$ uniformly at random from $K \setminus G_t$ and, independently, generate a
Bernoulli random variable $\xi_{t+1}$ with the probability of success  $1 - \gamma_t$ (which is in $[0,1]$ by~\eqref{eq_gammat_less1}). If event
$\ac(t, \gamma_t)$ has occurred, that is, if
\begin{equation}
  \label{eq_coupling}
p_{t+1}(e, R_t) \ge \frac{1 - \gamma_t}{N - t} \quad \text{ for every } \quad e \in K
\setminus R_t,
\end{equation}
then  draw a random edge $\zeta_{t+1} \in K \setminus R_t$ according to the distribution
\begin{equation*}
\Pc{\zeta_{t+1} = e}{\R'(t) = R_t} := \frac{p_{t+1}(e, R_t) - (1 - \gamma_t)/(N - t) }{\gamma_t} \ge 0,
\end{equation*}
where the inequality holds by~\eqref{eq_coupling}. Observe also that
\begin{equation*}
\sum_{e\in K \setminus R_t}\Pc{\zeta_{t+1} = e}{\R'(t) = R_t} = 1,
\end{equation*}
so $\zeta_{t+1}$ has a properly defined distribution. Finally, fix an arbitrary bijection
\[
  f_{R_t, G_t} : R_t \setminus G_t \to G_t \setminus R_t
\]
 between the sets of edges and define
\begin{equation*}
\eta'_{t+1} = \begin{cases}
\eps_{t+1}, &\text{ if } \xi_{t + 1} = 1, \eps_{t+1} \in K \setminus R_t,\\
f_{R_t, G_t}(\eps_{t+1}), &\text{ if } \xi_{t + 1} = 1, \eps_{t+1} \in R_t,\\
\zeta_{t+1}, &\text{ if } \xi_{t + 1} = 0. \\
\end{cases}
\end{equation*}
On the other hand, if event $\ac(t, \gamma_t)$ has failed, then $\eta'_{t+1}$ is sampled
directly (without defining $\zeta_{t+1}$) according to the distribution~\eqref{eq_etaprime}.
With this definition of $(\R'(t))_{t= 0}^{t_0}$, it is easy to check that for $\eta'_{t+1}$ defined above,~\eqref{eq_etaprime} indeed holds, so from now on we drop the prime $'$ and identify $\R'(t)$ with $\R(t)$, which is a subset of $\Rnnp$.

Most importantly, we conclude that, for $t = 0, \cdots, t_0 - 1$
\begin{equation}
  \label{eq_implies}
\ac(t, \gamma_t) \cap \{\xi_{t+1} = 1\} \quad \implies \quad \eps_{t+1} \in \R(t+1).
\end{equation}
 In view of this, define
\begin{equation*}
  S := \left\{ t \in [t_0] : \xi_t = 1 \right\}
\end{equation*}
and recall that $m=\lceil (1-\gamma)M\rceil$.
If $|S| \ge m$,  define $\Gnnm$ as, say, the edges indexed by the smallest $m$ elements of $S$ (note that since the vectors $(\xi_i)$ and $(\eps_i)$ are independent, after conditioning on $S$, these $m$ edges are uniformly distributed), and if $|S| < m$, then  define $\Gnnm$ as, say, the graph with edges $\{\eps_1, \dots, \eps_m\}$.
Recalling the definition~\eqref{eq_ac} of the event $\ac$, by~\eqref{eq_implies} we observe that $\ac$ implies the inclusion $\left\{ \eps_t : t \in S \right\} \subseteq \R(t_0) \subset \R(M) = \Rnnp$. On the other hand $|S| \ge m$ implies $\Gnnm \subseteq \left\{ \eps_t : t \in S \right\}$, so
\begin{equation*}
\prob{\Gnnm \subseteq \R} \ge \prob{\{|S| \ge m\}\cap \ac }.
\end{equation*}
Since, by Lemma~\ref{lem_Stanislaw}, event $\ac$ holds with probability $1 - O(N^{-C})$, to complete the proof of~\eqref{eq_lower_embed} it suffices to show that also
\begin{equation*}
	\prob{|S| \ge m}  = 1 - O(N^{-C}) .
\end{equation*}
For this we need the following claim whose technical proof is deferred to Section~\ref{sec_technical}.
 \begin{claim}
   \label{clm_sumofgammas}
  We have
\begin{equation}
  \label{eq_theta_p}
  \E |S| \ge t_0 - \theta M,
\end{equation}
where
 \begin{equation*}
	\theta := 1080 \hat p^2 \I +
	\begin{cases}
	  6480\sqrt{\frac{2(C + 3)\log N}{p\hat n}}, &\quad  p \le 0.49, \\
		6250\sqrt{\frac{(C + 3)\log N}{q^2\hat n }  }\log \frac{\hat n}{\log N}, &\quad  p > 0.49,
	\end{cases}
\end{equation*}
and, with $\gamma$ as in~\eqref{eq_gamma},
\begin{equation}
  \label{eq_thetaM}
  \gamma  \ge \tau_0 + \theta  + 2/M + \sqrt{\frac{2C \log N}{M}}.
\end{equation}
 \end{claim}
 \medskip
 Recalling that $t_0 = \lfloor (1 - \tau_0)M \rfloor$ and $m = \lceil(1-\gamma)M\rceil$, we have
 \begin{equation}
   \label{eq_tzero_theta_Mm}
   \begin{split}
     t_0 - \theta M - m &\ge (1-\tau_0)M - 1 - \theta M - (1-\gamma)M - 1 \\
     &= (\gamma - \tau_0 - \theta)M - 2 \By{\eqref{eq_thetaM}}{\ge} \sqrt{2C M \log N}.
   \end{split}
 \end{equation}
 Since $|S|$ is a sum of independent Bernoulli random variables,
\begin{align*}
  \prob{|S| < m} &= \prob{|S|  < \E|S|- (\E|S| - m) } \By{\eqref{eq_theta_p}}{\le} \prob{|S|  < \E|S| -(t_0 - \theta M - m)} \\
  \justify{\eqref{eq_hypergeomLowTailExp}, $\E |S| \le M$} &\By{}{\le} \exp \left( -\frac{(t_0 - \theta M - m)^2}{2M} \right)
 \By{\eqref{eq_tzero_theta_Mm}}{\le} \exp \left( - C \log N \right) = N^{-C},
\end{align*}
which as mentioned above, implies~\eqref{eq_lower_embed}.

\medskip

Finally, we prove~\eqref{eq:embed_Gnp} by coupling $\G(n_1, n_2, p')$ with $\G(n_1,n_2, m)$ so that the former is a subset of the latter with probability $1 - O(N^{-C})$.
	Denote by $X := e(\G(n_1,n_2,p'))$ the number of edges of $\G(n_1,n_2, p')$.
	Whenever $X \le m$,  choose $X$ edges of $\G(n_1,n_2,m)$ at random and declare them to be a copy of $\G(n_1,n_2,p')$. Otherwise sample $\G(n_1,n_2,p')$ independently from $\G(n_1,n_2,m)$. Hence it remains to show that $\prob{X > m} = O(N^{-C})$.

	Recalling that $m = \ceil{(1-\gamma)}pN$ and $p' = (1 - 2\gamma)p$, we have $m/N \ge (1 - \gamma)p = p' + \gamma p$. Since $X \sim \Bi(N,p')$, Chernoff's bound~\eqref{eq_hypergeomUpTailExp} implies that
	\begin{equation}
	  \label{eq:Xmorem}
		\prob{X > m} \le \prob{X > p'N + \gamma p N} \le \exp\left\{ -\frac{(\gamma pN)^2}{2(p'N + \gamma pN/3)} \right\} \le e^{-\gamma^2pN/2}.
	\end{equation}
	Choosing $C^*$ sufficiently large, the definition~\eqref{eq_gamma} implies $\gamma \ge \sqrt{\frac{2C\log N}{p N}}$ with a lot of room (to see this easier in the case $p > 0.49$, note that \eqref{eq:thm_embed_q} implies $\log(\hat n/\log N) \ge 1$, say). So~\eqref{eq:Xmorem} implies $\prob{X > m} \le N^{-C}$.

\section{Proof of Lemma~\ref{lem_Stanislaw}}
\label{sec_lemmaproof}

\subsection{Preparations}\label{ss_prep} Recall our notation $K = \Knn$, and $\rc_G$ defined in~\eqref{eq_rcG}. Fix an integer $t \in [0, t_0)$. For a graph~$G$ with $t$ edges and $e, f \in K\setminus G$, define
\begin{equation*}
	\rc_{G,e,\neg f} := \left\{ H \in \rc_G : e \in  H, f \notin H\right\}.
\end{equation*}
By skipping a few technical steps, we will see that the essence of the proof of Lemma~\ref{lem_Stanislaw} is to show that the ratio ${|\rc_{G,e,\neg f}|} / {|\rc_{G,f,\neg e}|}$ is approximately $1$ for any pair of edges $e,f$, where~$G$ is a `typical' instance of $\R(t)$.

Recalling our generic definition of  switchings from Subsection~\ref{ss_switchings}, let us treat the graphs in $\rc := \rc_{G,e,\neg f}$ and $\rc' := \rc_{G,f,\neg e}$ as blue-red edge colorings of the graph $K \setminus G$.
Recall that a path or a cycle in $K \setminus G$ is \emph{alternating} if no two consecutive edges have the same color.

When edges $e,f$ are disjoint, we define the switching graph $B = B(\rc,\rc')$ by putting an edge between $H \in \rc$ and $H' \in \rc'$ whenever there is an alternating $6$-cycle containing $e$ and $f$ in $H$ such that switching the colors in the cycle gives $H'$ (see Figure~\ref{fig_coupling}). If $e$ and $f$ share a vertex, instead of $6$-cycles we use alternating $4$-cycles containing $e$ and $f$.

It is easy to describe the vertex degrees in $B$.
For distinct $u,v\in V_i$, let $\theta_{G,H}(u,v)$ be the number of (alternating) paths $uxv$ such that $ux$ is blue and $xv$ is red.
Note that $\theta_{G,H}(u,v) = |\Gamma_{H\setminus G}(u)\cap\Gamma_{K\setminus H}(v)|$.
 Then, setting  $f = u_1u_2$ and $e = v_1v_2$, where $u_i,v_i\in V_i$, $i = 1,2$,
\begin{equation}
  \label{eq_deg_forw}
	\deg_B(H) = \prod_{i: u_i \neq v_i}\theta_{G,H}(u_i,v_i), \quad H \in \rc
\end{equation}
and
\begin{equation}
  \label{eq_deg_back}
	\deg_B(H') = \prod_{i: u_i \neq v_i}\theta_{G,H'}(v_i,u_i), \quad H' \in \rc'.
\end{equation}

\begin{figure}
\begin{center}
\begin{tikzpicture}[scale = 0.5]
    \tikzstyle{vertex} = [draw,circle,fill=white,fill opacity=1,minimum size=4pt, inner sep=2pt]

      \draw (3,-3) node (a) {$H$};
      \foreach \x in {1,2,3,4,5,6}
      {
        \draw (180-\x*60: 2) node[vertex] (\x) {} ;
      }
        \foreach \x/\y in {1/v_1, 2/v_2, 3/x_2, 6/x_1}
        {
          \draw node [above] at (\x.north) {$\y$};
      }
        \foreach \x/\y in {5/u_1, 4/u_2}
        {
          \draw node [below] at (\x.south) {$\y$};
      }
      \draw [very thick, auto] (1) to node {$e$} (2);
      \foreach \x/\y in {2/6, 1/3, 4/5}
      {
              \draw [dashed, very thick, auto] (\x) to (\y);
      }
      \foreach \x/\y in {3/5, 4/6}
      {
              \draw[very thick, auto]  (\x) to (\y);
      }

\draw (6,0) node () {\huge $ \Rightarrow$};

  \begin{scope}[shift={(12,0)}]
    \draw (3,-3) node (a) {$H'$};
      \foreach \x in {1,2,3,4,5,6}
      {
        \draw (180-\x*60: 2) node[vertex] (\x) {} ;
      }
        \foreach \x/\y in {1/v_1, 2/v_2, 3/x_2, 6/x_1}
        {
          \draw node [above] at (\x.north) {$\y$};
      }
        \foreach \x/\y in {5/u_1, 4/u_2}
        {
          \draw node [below] at (\x.south) {$\y$};
      }
      \draw [very thick, auto] (4) to node {$f$} (5);
      \foreach \x/\y in {2/6, 1/3}
      {
              \draw [very thick, auto] (\x) to (\y);
      }
      \foreach \x/\y in {1/2, 3/5, 4/6}
      {
	      \draw[dashed, very thick, auto]  (\x) to (\y);
      }
  \end{scope}
  \end{tikzpicture}
  \end{center}
  \caption{Switching between $H$ and $H'$ when $e$ and $f$ are disjoint: solid edges are in $H \setminus G$ (or $H' \setminus G$) and the dashed ones in $K \setminus H$ (or $K \setminus H'$).}
  \label{fig_coupling}
\end{figure}
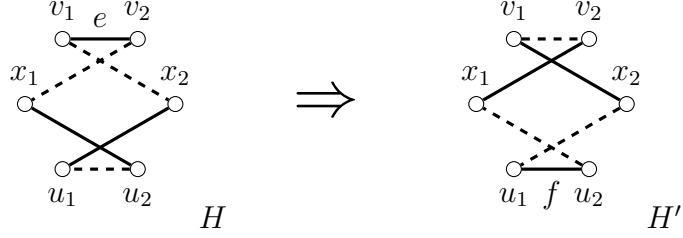

Note that equation~\eqref{eq_BRR} is equivalent to
\begin{equation}
  \label{eq_ratio_R}
	\frac{|\rc|}{|\rc'|} = \frac{\frac1{|\rc'|}\sum_{H'\in\rc'}\deg_B(H')}{\frac1{|\rc|}\sum_{H\in\rc}\deg_B(H)}.
\end{equation}
In view of \eqref{eq_deg_forw} and \eqref{eq_deg_back}, the denominator of the \rhs{} above is the (conditional) expectation of the random variable $\prod_{i : u_i \neq v_i} \theta_{G,\R}(u_i,v_i)$, given that $\R$ contains $G$ and~$e$, but not~$f$ (and similarly for the numerator).

To get an idea of how large that expectation could be, let us focus on one factor, say, $\theta_{G,\R}(u_1,v_1)$, assuming $u_1 \neq v_1$.
Clearly, the red degree of $v_1$ equals
$|\Gamma_{K\setminus\R}(v_1)| = n_2 - d_2 = qn_2$. Since $|\Gamma_{\R}(v_1)| = pn_2$, viewing $\R \setminus G$ as a $\tau$-dense subgraph of $\R$ (see \eqref{eq_tau}) we expect that the blue neighborhood $|\Gamma_{\R\setminus G}(u_1)|$ is approximately $\tau p n_2$.
It is reasonable to expect that for typical $G$ and $\R$ the red and blue neighborhoods intersect proportionally, that is, on a set of size about $q \cdot \tau p \cdot n_2 = \tau qd_1$.

 Inspired by this heuristic, we say that, for  $\delta > 0$, an admissible graph $G$ with $t$ edges is \emph{$\delta$-typical}  if
\begin{equation}
  \label{eq_typical}
  \max_{u_1u_2, v_1v_2 \in K \setminus G}	\Pc{\max_{i \in [2] : u_i \neq v_i}\left|\frac{\theta_{G,\R}(u_i,v_i)}{\tau qd_i } - 1\right| > \delta }{G \subseteq \R, v_1v_2 \in \R, u_1u_2 \notin \R} \le \tau^2 \delta,
\end{equation}
where the outer maximum is taken over distinct pairs of edges.
(The bound $\tau^2 \delta$ has hardly any intuition, but it is simple and sufficient for our purposes.)

Lemma~\ref{lem_Stanislaw} is a relatively easy consequence of the upcoming Lemma~\ref{lem_typical}, which states that for a suitably chosen function $\delta(t)$ it is very likely that the initial segments of $\R$ are $\delta(t)$-typical.

For each $t = 0, \dots, t_0 - 1$, define
\begin{equation}
  \label{eq_deltat}
	\deltat := 120\hat p^2 \I + 360 \sqrt{\frac{(C + 3)\lambda(t)}{6\tau pq\hat n }} ,
\end{equation}
where
\begin{equation}
  \label{eq_lambda}
	\lambda(t)  :=
\begin{cases}
  6 \log N, \quad &p \le 0.49, \\
  6\log N + \frac{64\log N}{\tau p q}, \quad &p > 0.49.
\end{cases}
\end{equation}
For future reference, note that
\begin{equation}
  \label{eq_deltat_gammat}
\deltat \le \gamma_t/9.
\end{equation}
Indeed, recalling~\eqref{eq_gamma_t}, for $p\le0.49$,
\[
	9\deltat \le 1080\hat p^2 \I + 3240 \sqrt{\frac{2(C + 3)\log N}{\tau p\hat n }} =\gamma_t,
\]
while, for $p > 0.49$, noting that $\lambda(t)\le 70\log N/(\tau pq)$, we have
\[
9\deltat \le 1080\hat p^2 \I + 3240 \sqrt{\frac{70(C + 3)\log N}{0.49^2 \cdot 6\tau^2q^2\hat n}} \le \gamma_t.
\]
\begin{lemma}
  \label{lem_typical}
	For every constant $C > 0$, under the conditions of Lemma~\ref{lem_Stanislaw},
\begin{equation*}
  \prob{\R(t) \text{ is } \deltat\text{-typical for all } t < t_0} = 1 - O(N^{-C}).
\end{equation*}
\end{lemma}
Note that the \lhs{} of~\eqref{eq_typical} is a function of graph $G$, say $f(G)$. Hence, Lemma~\ref{lem_typical} asserts that, with high probability, $f(\R(t))$ is small for all $t < t_0$.
The main idea of the proof of Lemma~\ref{lem_typical}, which we defer to Subsection~\ref{ss_thetas}, is to bound $f(\R(t))$ by a ratio of two simpler functions (again conditional probabilities) of $\R(t)$. Lemmas~\ref{lem_thetas} and~\ref{lem_ef} below bound each of these conditional probabilities separately.

For any $t = 0, \dots, M$ and distinct $u_i, v_i \in V_i$, $i = 1,2$, set
\begin{equation}
  \label{eq_theta_t}
\theta_t(u_i,v_i) := \theta_{\R(t),\R}(u_i,v_i),
\end{equation}
where $\theta_{G,H}(u_i,v_i) = |\Gamma_{H\setminus G}(u_i)\cap\Gamma_{K\setminus H}(v_i)|$ was introduced earlier in this subsection.

\begin{lemma}
\label{lem_thetas}
For every constant $C>0$, under the conditions of Lemma~\ref{lem_Stanislaw}, for $i \in [2]$ we have that with probability $1 - O(N^{-C})$, 
	\begin{equation}
	  \label{eq_condthetas}
	  \Pc{ \max_{u,v \in V_i, u \neq v} \left|\frac{\theta_t(u,v)}{\tau qd_{i} } - 1\right| > \deltat}{\R(t)} \le e^{ - 2\lambda(t) }, \qquad \text{for all  } t < t_0.
	\end{equation}
\end{lemma}
\begin{lemma}
  \label{lem_ef}
	For every constant $C > 0$, under the conditions of Lemma~\ref{lem_Stanislaw}, with probability $1 - O(N^{-C})$
\begin{equation}
  \label{eq_edgesNonnegligible}
	\min_{e, f \in K \setminus \R(t), e \neq f}\Pc{e \in \R, f \notin \R}{\R(t)} \ge e^{ - \lambda(t) }, \qquad \text{for all } t < t_0.
\end{equation}
\end{lemma}

\subsection{Proof of Lemma~\ref{lem_Stanislaw}}
\label{ss_Stanislaw_proof}
In view of Lemma~\ref{lem_typical}, it suffices to show that if $t < t_0$, and  a $t$-edge graph $G$ is $\deltat$-typical, then
\begin{equation*}
	\min_{e \in K \setminus G}p_{t+1}(e,G) \ge \frac{1 - \gamma_t}{N - t}.
\end{equation*}
Fix $f \in K \setminus G$ which \emph{maximizes} $p_{t+1}(f,G)$.
Since the average of
$p_{t+1}(e,G)$ over $e \in K \setminus G$ is exactly $\frac{1}{N - t}$, we have
$p_{t+1}(f,G) \ge \frac{1}{N - t}$
and therefore it is enough to prove that, for every $e \in K \setminus (G \cup \{f\})$,
\begin{equation*}
\frac{p_{t+1}(e,G)}{p_{t+1}(f,G)}
	\ge 1 - \gamma_t.
\end{equation*}
Recall the definitions of $\rc_G$ in~\eqref{eq_rcG}, $\rc_{G,e}, \rc_{G, \neg e}$ in~\eqref{eq_rcGe}, and $\rc = \rc_{G,e,\neg f}$ and $\rc' = \rc_{G, f, \neg e}$ from Subsection \ref{ss_prep}
and observe that $\rc = \rc_{G\cup \{e\},\neg f}$ too. In view of the remark immediately following~\eqref{eq_pt},
\begin{equation*}
  p_{t+1}(e,G) = \Pc{e \in \R}{\R(t) = G} \cdot \Pc{\eta_{t+1} = e}{\R(t) = G, e \in \R} = \frac{|\rc_{G,e}|}{|\rc_G|} \cdot \frac{1}{M-t}.
\end{equation*}
Therefore,
\begin{equation*}
  \label{eq_ratio}
	1 \ge \frac{p_{t+1}(e,G)}{p_{t+1}(f,G)}
	= \frac{|\rc_{G,e}|}{|\rc_{G,f}|}
	= \frac{|\rc_{G \cup \{e,f\}}| + |\rc_{G,e,\neg f}|}{|\rc_{G \cup \{e, f\}}| + |\rc_{G,f,\neg e}| } \ge \frac{|\rc|}{|\rc'|}.
\end{equation*}
Write $e = v_1v_2$ and $f = u_1u_2$ and for simplicity assume that both $u_1 \neq v_1$ and $u_2 \neq v_2$ (otherwise the proof goes mutatis mutandis and is, in fact, a bit simpler).
By~\eqref{eq_deg_forw},~\eqref{eq_deg_back}, and~\eqref{eq_ratio_R},
\begin{equation*}
	\frac{|\rc|}{|\rc'|} = \frac{\E_{\text{top}}}{\E_{\text{bottom}}},
\end{equation*}
where
\begin{equation*}
	\E_{\text{top}} := \frac1{|\rc'|}\sum_{H'\in\rc}\deg_B(H')= \Ec{\theta_{G,\R}(v_1,u_1)\theta_{G,\R}(v_2,u_2)}{G \subseteq \R, f \in \R, e \notin \R}
\end{equation*}
and
\begin{equation*}
	\E_{\text{bottom}} := \frac1{|\rc|}\sum_{H\in\rc}\deg_B(H)= \Ec{\theta_{G,\R}(u_1,v_1)\theta_{G,\R}(u_2,v_2)}{G \subseteq \R, e \in \R, f \notin \R}.
\end{equation*}
Since $G$ is $\delta$-typical with $\delta = \deltat$,
denoting the \lhs{} of~\eqref{eq_typical} as $p_*$,
we have $p_* \le \tau^2 \delta$, so
\begin{align}
  \notag \E_{\text{top}}
  &\ge (1 - \delta) \tau qd_1 \cdot (1 - \delta)\tau qd_2 \cdot (1 - p_*) + 0 \cdot p_*
  \\
  \label{eq_top_bound} &\ge (1 - \delta)^2 \tau^2p^2q^2n_1n_2(1 - \tau^2\delta )
  \ge (1 - \delta)^3 \tau^2p^2q^2N.
\end{align}
Moreover, since, deterministically,
\begin{equation*}
  \theta_{G,\R}(u_1, v_1)\theta_{G, \R}(u_2, v_2) \le \min\{p,q\} n_2\cdot \min\{p,q\}n_1 \le 4p^2q^2N,
\end{equation*}
using~\eqref{eq_typical} again, we infer that
\begin{align}
  \E_{\text{bottom}}
  \notag &\le (1 + \delta) \tau qd_1 \cdot (1 + \delta)\tau qd_2 \cdot (1 - p_*) + 4p^2q^2N \cdot p^*\\
  \label{eq_bottom_bound} &\le (1 + \delta)^2\tau^2p^2q^2N + 4p^2q^2N \cdot \tau^2\delta \le (1 + 6\delta + \delta^2)\tau^2p^2q^2N .
\end{align}
Finally, combining the bounds on $\E_{\text{top}}$ and $\E_{\text{bottom}}$ and using~\eqref{eq_deltat_gammat}, we conclude that
\begin{equation}
\label{eq_p_ratio}
  \frac{p_{t+1}(e,G)}{p_{t+1}(f,G)}
  \ge \frac{(1 - \delta)^3}{(1 + 6 \delta + \delta^2)} \ge 1 - 9\delta \ge  1 - \gamma_t,
\end{equation}
and the proof of Lemma~\ref{lem_Stanislaw} is complete. \qed

\subsection{Proof of Lemma~\ref{lem_typical}}
\label{ss_thetas}
By Lemmas \ref{lem_thetas} and \ref{lem_ef}  the events~\eqref{eq_condthetas} (for each $i \in [2]$) and~\eqref{eq_edgesNonnegligible} hold simultaneously with probability $1 - O(N^{-C})$. Hence it suffices to fix an arbitrary integer $t < t_0$ and a realization of $\R(t)$, say $\R(t) = G$ satisfying inequalities~\eqref{eq_condthetas} (for each $i \in [2]$) and~\eqref{eq_edgesNonnegligible} and to prove that $G$ is $\deltat$-typical. Fix any such $G$ and define events
\begin{equation*}
  \ec :=  \bigcup_{i \in [2]} \left\{ \max_{u,v \in V_i, u \neq v}\left|\frac{\theta_{G,\R}(u,v)}{\tau qd_i } - 1\right| > \deltat \right\} \quad\text{and}\quad \fc_{e,f} := \{ e\in \R, f \notin \R \}.
\end{equation*}
Since, conditioning on the event $\R(t) = G$, we have $\theta_{G,\R}(u_i,v_i) = \theta_t(u_i,v_i)$, by the choice of $G$ we have
\begin{equation}\label{eq_thetas_Rt}
  \Pc{\ec}{\R(t) = G} \le 2e^{-2\lambda(t)}
\end{equation}
and
\begin{equation}
  \label{eq_edges_Rt}
  \min_{e,f\in K \setminus G}\Pc{\fc_{e,f}}{\R(t) = G} \ge e^{-\lambda(t)}.
\end{equation}
Note that the probability on the LHS of~\eqref{eq_typical} does not change if we replace $G \subseteq \R$ by $\R(t) = G$, since conditioning on either event makes $\R$ uniformly distributed over $\rc_G$ (i.e., biregular graphs containing $G$) and the random variables $\theta_{G,\R}(u_i,v_i)$ do not depend on the random ordering of the edges of $\R$. Hence it remains to  prove that~\eqref{eq_thetas_Rt} and~\eqref{eq_edges_Rt} imply, for any distinct edges $e = v_1v_2, f = u_1u_2 \in K \setminus G$, that
\begin{equation}
  \label{eq_typical2}
  \Pc{\max_{i \in [2] : u_i \neq v_i}\left|\frac{\theta_{G,\R}(u_i,v_i)}{\tau qd_i } - 1\right| > \delta(t) }{\R(t) = G, e \in \R, f \notin \R} \le \tau^2 \delta(t).
\end{equation}
The probability on the {\lhs}  of~\eqref{eq_typical2} is at most $\Pc{\ec}{\R(t) = G, \fc_{e,f}}$.
Inequalities~\eqref{eq_thetas_Rt} and~\eqref{eq_edges_Rt} imply that
	\begin{equation*}
	  \Pc{\ec}{\R(t) = G, \fc_{e,f}} = \frac{\Pc{\ec\cap \fc_{e,f}}{\R(t) = G}}{\Pc{\fc_{e,f}}{\R(t) = G}} \le \frac{\Pc{\ec}{\R(t) = G}}{\Pc{\fc_{e,f}}{\R(t) = G}} \le 2 e^{ - \lambda(t)}.
	\end{equation*}
	Finally, even a quick glance at the definitions of  $\deltat$ and $\tau_0$ (see~\eqref{eq_deltat} and~\eqref{eq_tauzero}) reveals  that $\min \{ \tau_0,\delta(t) \} \ge 1/\hat n$. Thus,
we get, with a huge margin,
\begin{equation*}
	2e^{-\lambda(t)}\le 2/N^6 \le 1/N^2 \le 1/\hat n^4 \le \tau_0^2\deltat \le \tau^2\deltat.
\end{equation*}
\qed

\section{Degrees and Co-degrees}
\label{sec_degs_codegs}
In this section we prove facts about the neighborhood structure of $\R(t)$, one of which will be enough to deduce Lemma~\ref{lem_thetas}, while the other two will be used in the proof of Lemma~\ref{lem_ef} in Section~\ref{sec_quasi}. We start, in Subsection~\ref{ss_codegs}, with  a tail bound for the co-degrees in $\R = \Rnnp$ (Lemma~\ref{lem_codegs}). In Subsection~\ref{ss_Rt_cod} we analyze the process~$(\R(t))_t$. First, we show that the vertex degrees in the process grow proportionally until almost the very end (Lemma~\ref{lem_degrees}). Then, conditioning on $\R$ having concentrated co-degrees, we  prove that the co-degrees in $\R(t)$ do not exceed  their expectation too much (Lemma~\ref{lem_codegsRt}).
Finally, in Subsection~\ref{ss_proof_thetas}, we present a proof of Lemma~\ref{lem_thetas} based on Lemma~\ref{lem_codegs}.

 \subsection{Co-degrees in the random biregular graph}
 \label{ss_codegs}
 Recall that $\Gamma_F(v)$ is the set of neighbors of a vertex $v$ in a graph $F$. We define the \emph{co-degree} of two distinct vertices $u, v$ as
\begin{equation}
\label{eq_def_cod}
 \cod_F(u,v) := |\Gamma_F(u) \cap \Gamma_F(v)|.
\end{equation}
A few times we will use the following simple observation: for $F \subseteq K$ and distinct $u,v\in V_1$,
\begin{equation}
    \label{eq_coco}
  \begin{split}
    \cod_{F}(u, v) &= |\Gamma_{F}(u)| + |\Gamma_{F}(v)| - |\Gamma_{F}(u) \cup \Gamma_{F}(v)| \\
    &= \deg_F(u) + \deg_F(v) - \left( n_2 - \cod_{K \setminus F}(u,v) \right),
  \end{split}
\end{equation}
where, recall, $K = \Knn$ is the complete bipartite graph.
\smallskip

Due to symmetry, we prove the following concentration result for pairs of vertices on one side of the bipartition only.

\begin{lemma}
  \label{lem_codegs}
	Suppose that $\hat p\hat n  \to \infty$  and let $\lambda = \lambda(n_1,n_2)$ be such that $\lambda \to \infty$.
	Then, for any distinct $u_1, v_1 \in V_1$,
\begin{equation*}
  \prob{
	|\cod_{\R}(u_1,v_1) - p^2n_{2}| > 20\left( \hat p^3n_{2}\I + \frac{\hat pn_{2}}{\hat n} + \sqrt{\lambda \hat p^2 n_2} \right) + \lambda
      } = O\left( \sqrt N e^{ - \lambda } \right),
\end{equation*}
where $\I$ is defined in~\eqref{eq_defI}.
\end{lemma}

\begin{proof}
We first claim that it is sufficient to assume $p \le 1/2$ and prove
that for any distinct $u_1, v_1 \in V_1$
\begin{equation}
  \label{eq_codegconcnoq}
  \prob{ |\cod_{\R}(u_1,v_1) - p^2n_2| > 20\left(p^3n_2\I + \frac{pn_2}{\hat n} + \sqrt{\lambda p^2n_2 } \right) + \lambda } = O\left( \sqrt N e^{ - \lambda } \right).
\end{equation}
	To see this, note that by~\eqref{eq_coco}, recalling $q = 1 - p$,
	\begin{equation*}
		\cod_{\R}(u_1, v_1) - p^2n_2= 2pn_2 - (n_2 - \cod_{K \setminus \R}(u_1,v_1))- p^2n_2=\cod_{K \setminus \R}(u_1,v_1) - q^2n_2.
	\end{equation*}
	Further, $K \setminus \R = K \setminus \Rnnp$ has the same distribution as $\R(n_1,n_2,q)$. So, if $p > 1/2$, the lemma follows by applying~\eqref{eq_codegconcnoq} with $q$ instead of $p$.

	The crude proof idea comes from our anticipation that $\cod_\R(u_1,v_1)$ behaves similarly to $\cod_{\Gnnp}(u_1,v_1)$, which is distributed as $\Bi(n_2, p^2)$ or, approximately, as $\Po(p^2n_2)$. We will show that each tail of $\cod_\R(u_1,u_2)$ is comparable to the tail of $\Po(\mu)$ with expectation~$\mu$ fairly close to $p^2n_2$ and then apply the Chernoff bounds (this is packaged in Claim~\ref{clm_Poiss_bounds} below). Further we consider the cases $\I = 1$ and $\I = 0$ and apply Claim~\ref{clm_Poiss_bounds} analogously to the proof of Theorem 2.1 in~\cite{KSVW}: in the case $\I = 1$ we use switchings and in the case $\I = 0$ we use asymptotic enumeration (Theorem~\ref{thm_enum}).

	Fix distinct $u_1, v_1 \in V_1$ and set $X(u_1,v_1): = \cod_{\R}(u_1,v_1)$. Further, recall that $\rcnnp$ is the class of subgraphs of $\Knn$ such that every $v \in V_i$ has degree $d_i$, for $i = 1, 2$, and let
	\[
	  \rc_k(u_1,v_1) = \{H \in \rcnnp : \cod_H(u_1,v_1) = k \}, \qquad k = 0, \dots, d_1.
	  \]
	\begin{claim}
	  \label{clm_Poiss_bounds}
	 Fix two vertices $u_1,v_1\in V_1$. Suppose that positive numbers $r_k, k = 0, \dots, d_1,$ are such that
	  \begin{equation*}
	    |\rc_k(u_1,v_1)| \sim r_k, \text{ uniformly in $k$ as } (n_1, n_2) \to \infty,
	  \end{equation*}
	  and there exist numbers $0 \le \mu_- \le \mu_+ \le N$ such that
	  \begin{equation*}
	    \frac{r_k}{r_{k-1}} \le \frac{\mu_+}{k}, \text{ for } k \in [\mu_+, d_1]
	    \quad \text{ and } \quad
	    \frac{r_{k-1}}{r_k} \le \frac{k}{\mu_-}, \text{ for } k \in [1, \mu_-].
	  \end{equation*}
	  If $x \ge \sqrt{2\mu_+ \lambda} + \lambda$, then
	  \begin{equation}
	    \label{eq_twotailbound}
	    \prob{X(u_1,v_1) \ge \mu_+ + x} + \prob{X(u_1,v_1) \le \mu_- - x} = O(\sqrt{N} e^{-\lambda}).
	  \end{equation}
	\end{claim}
	\begin{proof}
	  Note that if $Z_+ \sim \Po(\mu_+)$, then $\mu_+/k = \prob{Z_+ = k}/\prob{Z_+ = k - 1}$. Setting $m = \lfloor \mu_+ \rfloor$,
	  for any integer $i\in\{m,\dots,d_1\}$, and abbreviating $X:=X(u_1,v_1)$ and $\rc_j:=\rc_j(u_1,v_1)$, we have
	  \begin{align*}
	    \prob{X \ge i}  &= \frac{\sum_{j \ge i}|\rc_j| }{|\rcnnp|}\le \sum_{j \ge i} \frac{|\rc_j|}{|\rc_{m}|} \sim \sum_{j \ge i} \frac{r_j}{r_m} = \sum_{j \ge i} \prod_{k = m + 1}^{j} \frac{r_k}{r_{k-1}} \le \sum_{j \ge i} \prod_{k = m + 1}^{j} \frac{\mu_+}{k}\\
	    &= \sum_{j \ge i} \prod_{k = m + 1}^{j}\frac{\prob{Z_+ = k}}{\prob{Z_+ = k-1}}  = \sum_{j \ge i} \frac{\prob{Z_+ = j}}{\prob{Z_+ = m}} = \frac{\prob{Z_+ \ge i}}{\prob{Z_+ = m}}.
	  \end{align*}
	  Since  $\mu_+ \le N$, by \eqref{eq_largestatom} we have $\prob{Z_+ = m} = \Omega(1/\sqrt{\mu_+}) = \Omega(1/\sqrt{N})$, and conclude that
	  \begin{equation}
	    \label{eq_Poisstail_up}
	    \prob{X \ge i} = O\left( \sqrt{N} \cdot \prob{Z_+ \ge i} \right), \qquad i \ge \mu_+.
	  \end{equation}
	  Similarly, if $Z_- \sim \Po(\mu_-)$, then for $i\le m := \lfloor \mu_- \rfloor$, using $k/\mu_- = \prob{Z_- = k-1}/\prob{Z_- = k}$,
	  \begin{equation*}
	    \prob{X \le i}\le \sum_{j \le i} \frac{|\rc_j|}{|\rc_{m}|} \sim \sum_{j \le i}\prod_{k = j + 1}^{ m} \frac{r_{k-1}}{r_{k}} \le \sum_{j \le i}\prod_{k = j + 1}^{ m} \frac{\prob{Z_- = k-1}}{\prob{Z_- = k}} = \frac{\prob{Z_- \le i}}{\prob{Z_- = m}},
	  \end{equation*}
	  and therefore, again using \eqref{eq_largestatom},
	  \begin{equation}
	    \label{eq_Poisstail_down}
	    \prob{X \le i} = O\left( \sqrt{N} \cdot \prob{Z_- \le i} \right), \qquad i \le \mu_-.
	  \end{equation}
	  Since the \rhs{} of \eqref{eq_twotailbound} does not depend on $x$, we can assume the equality: $x = \sqrt{2\mu_+\lambda} + \lambda$. Inequalities \eqref{eq_hypergeomUpTailExp} and $\eqref{eq_hypergeomLowTailExp}$ imply
	  \begin{equation*}
	    \begin{split}
	      \prob{Z_+ \ge \mu_+ + x} &+ \prob{Z_- \le \mu_- - x} \le \exp \left( - \frac{x^2}{2(\mu_+ + x/3)} \right) + \exp \left( -\frac{x^2}{2\mu_-} \right) \\
	      \justify{$\mu_- \le \mu_+, x > 0$} &\le 2 \exp \left( -\frac{x^2}{2(\mu_+ + x/3)} \right) \\
	      \justify{$x = \sqrt{2 \mu_+ \lambda} + \lambda$} &= 2 \exp \left( -\frac{\lambda(2\mu_+ + 2\sqrt{2\mu_+ \lambda} + \lambda)}{2\mu_+ + \sqrt{2\mu_+ \lambda}/3 + \lambda/3} \right) \le 2e^{-\lambda}.
	    \end{split}
	  \end{equation*}
	  Combining this with \eqref{eq_Poisstail_up} and \eqref{eq_Poisstail_down} yields \eqref{eq_twotailbound}.
	\end{proof}

	It remains to prove \eqref{eq_codegconcnoq}. We consider separately the cases $\I = 1$ and $\I = 0$.

\smallskip
	
	\textbf{Case $\I = 1$.}
	If $p \ge 1/5$, then deterministically
	$\cod_\R(u_1,u_2) \le pn_2\le 5p^2n_2 \le p^2 n_2 + 20 p^3 n_2$ and $\cod_\R(u_1, u_2) - p^2 n_2 \ge - p^2 n_2 \ge -5p^3 n_2$, and hence \eqref{eq_codegconcnoq} holds trivially. So we further assume $p < 1/5$.

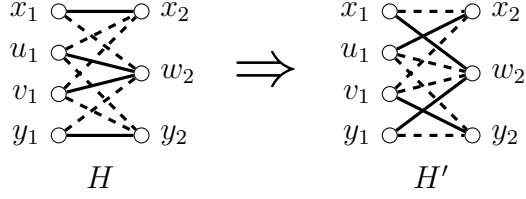
\begin{figure}
\begin{center}
\begin{tikzpicture}[scale = 0.55]
    \tikzstyle{vertex} = [draw,circle,fill=white,fill opacity=1,minimum size=4pt, inner sep=2pt]

    \foreach \x/\y/\n in {0/0/u_1, 0/-1/v_1, 0/1/x_1, 0/-2/y_1}
      {
	  \draw (\x,\y) node[vertex] (\n) {} ;
          \draw node [left] at (\n.west) {$\n$};
      }

      \foreach \x/\y/\n in {2/-0.5/w_2, 2/-2/y_2, 2/1/x_2}
      {
	  \draw (\x,\y) node[vertex] (\n) {} ;
          \draw node [right] at (\n.east) {$\n$};
      }
      \foreach \x/\y in {u_1/x_2, x_1/w_2, v_1/y_2, y_1/w_2, u_1/y_2, v_1/x_2}
      {
              \draw [dashed, very thick, auto] (\x) to (\y);
      }
      \foreach \x/\y in {u_1/w_2, x_1/x_2, v_1/w_2, y_1/y_2}
      {
              \draw[very thick, auto]  (\x) to (\y);
      }
     \draw (1,-3) node (a) {$H$};
  \begin{scope}[shift={(8,0)}]
    \foreach \x/\y/\n in {0/0/u_1, 0/-1/v_1, 0/1/x_1, 0/-2/y_1}
      {
	  \draw (\x,\y) node[vertex] (\n) {} ;
          \draw node [left] at (\n.west) {$\n$};
      }

      \foreach \x/\y/\n in {2/-0.5/w_2, 2/-2/y_2, 2/1/x_2}
      {
	  \draw (\x,\y) node[vertex] (\n) {} ;
          \draw node [right] at (\n.east) {$\n$};
      }
     \foreach \x/\y in {u_1/x_2, x_1/w_2, v_1/y_2, y_1/w_2}
      {
              \draw [very thick, auto] (\x) to (\y);
      }
      \foreach \x/\y in {u_1/w_2, x_1/x_2, v_1/w_2, y_1/y_2, u_1/y_2, v_1/x_2}
      {
              \draw[dashed, very thick, auto]  (\x) to (\y);
      }
     \draw (1,-3) node (a) {$H'$};
  \end{scope}
\draw (5,-0.5) node () {\huge $ \Rightarrow$};
  \end{tikzpicture}
  \end{center}
  \caption{Switching between $H \in \rc_k(u_1,v_1)$ and $H' \in \rc_{k-1}(u_1,v_1)$: solid edges are in $H$ and $H'$ and the dashed ones in $K \setminus H$ and $K \setminus H'$, respectively.}
  \label{fig_codegree}
\end{figure}
Setting $r_k = |\rc_k|$, we are going to prove bounds on $r_k/r_{k-1}$, $k\ge 1$, using a switching between $\rc = \rc_k$ and $\rc' = \rc_{k-1}$. To this end, recall our terminology from Subsection~\ref{ss_switchings} (with $G = \emptyset$): any graph $H \subseteq K$ is interpreted as a coloring of the edges of $K$, with the edges in $H$ blue and the rest red. We define a forward switching
as follows: pick a common blue neighbor $w_2 \in \Gamma_H(u_1) \cap \Gamma_H(v_1)$ and find two alternating $4$-cycles $u_1w_2x_1x_2$ and $v_1w_2y_1y_2$ so that $x_1 \neq y_1$, $x_2 \neq y_2$.
 Moreover, we restrict the choice of the cycles to those for which  $v_1x_2$ and $u_1y_2$ are red; this is to make sure that swapping of the colors on each of the two cycles (but keeping the color of $v_1x_2,u_1y_2$ red) indeed decreases $\cod_{\R}(u_1,v_1)$ by one, thus mapping $H \in \rc_k$ to $H' \in \rc_{k-1}$ (see Figure~\ref{fig_codegree}).

 Let $B_k := B(\rc_k, \rc_{k-1})$ be the auxiliary graph corresponding to the described switching.
Since the number of choices of $w_2$ is $k$ for any $H \in \rc_k$, we will show upper and lower bounds on $\deg_{B_k}(H)/k$ by fixing $w_2$ and proving upper and lower bounds on the possible choices of $(x_1, x_2, y_1, y_2)$. 
  
For the upper bound, since there are  $n_2 - 2d_1 + k$ common red neighbors of $u_1$ and $v_1$ (one can use \eqref{eq_coco} for this), the number of choices of distinct $x_2, y_2$ is exactly $(n_2 - 2d_1 + k)_2$, while the number of choices of $x_1$ and $y_1$, as blue neighbors of, resp., $x_2$ and $y_2$, is at most~$d_2^2$. Ignoring the requirements that $x_1\neq y_1$ and that both $x_1w_2$ and $y_1w_2$ should be red), we have shown
  \begin{equation}
    \label{eq_degB_upper}
    \deg_{B_k}(H)/k \le (n_2 - 2d_1 + k)_2d_2^2.
\end{equation}
For the lower bound, we subtract from the upper bound \eqref{eq_degB_upper} the number of choices of $(x_1, x_2, y_1, y_2)$ for which either  $x_1 = y_1$ or at least one of $x_1w_2$ and $y_1w_2$ is blue (the assumption $x_2 \neq y_2$ was taken into account in the upper bound). The number of choices with $x_1 = y_1$ is at most $n_1d_1^2$, and those with, say, $x_1w_2$ blue, is at most $d_2d_1 \cdot (n_2 - 2d_1 + k)d_2 $. Indeed, there are $\deg_H(w_2) \le d_2$ candidates for $x_1$ and, then, at most $d_1$ candidates for $x_2$, $\cod_{K \setminus H}(u_1,u_2) = n_2 - 2d_1 + k$ candidates for $y_2$, and then $d_2$ choices of $y_1$.
Therefore
\begin{align*}
  \deg_{B_k}(H)/k &\ge (n_2 - 2d_1 + k)_2d_2^2 - n_1d_1^2 - 2(n_2 - 2d_1 + k)d_1d_2^2 \\
&= (n_2 - 2d_1 + k)_2 d_2^2 \left( 1 - \frac{n_1d_1^2}{(n_2 - 2d_1 + k)_2 d_2^2} - \frac{2d_1}{n_2 - 2d_1 + k - 1}\right) \\
\justify{$d_1 = n_2p, d_2 = n_1p, k \ge 1$}& \ge  (n_2 - 2d_1 + k)_2 d_2^2 \left( 1 - \frac{n_2^2}{(n_2 - 2pn_2)^2 n_1} - \frac{2pn_2}{n_2 - 2pn_2}\right) \\
&= (n_2 - 2d_1 + k)_2 d_2^2 \left( 1 - \frac{1}{(1 - 2p)^2n_1} - \frac{2p}{1 - 2p}\right) \\
\justify{$p\le1/5$} &\ge (n_2 - 2d_1 + k)_2 d_2^2 \left( 1 - \left( \frac{25}{9pn_1} + \frac{10}{3} \right) p \right) \\
\justify{$\hat p \hat n \to \infty$} &\ge (n_2 - 2d_1 + k)_2 d_2^2 \left( 1 - 4p \right).
\end{align*}

 A moment of thought (and a glance at Figure~\ref{fig_codegree}) reveals that the backward switching corresponds to choosing a common red neighbor $w_2$ of $u_1$ and $v_1$, choosing alternating cycles $u_1w_2x_1x_2$ and $v_1w_2y_1y_2$ such that $x_1 \neq y_1$ and the edges $v_1x_2$ and $u_1y_2$ are red (note that these  assumptions  imply $x_2 \neq y_2$), and swapping the colors along the cycles.

 Since for every $H' \in \rc_{k-1}$ the number of choices of $x_2 \in \Gamma_{H'}(u_1) \setminus \Gamma_{H'}(v_1)$ and $ y_2 \in \Gamma_{H'}(v_1) \setminus \Gamma_{H'}(u_1)$ is exactly $(d_1 - k + 1)^2$, we will bound $\deg_{B_k}(H')/(d_1 - k + 1)^2$ from above and below by fixing $x_2, y_2$ and estimating the number of possible triplets $(w_2, x_1,y_1)$.
 For the upper bound, we can choose~$w_2$ in exactly $\cod_{K \setminus H'}(u_1,u_2) = n_2 - 2d_1 + k - 1$ ways, and a pair $x_1, y_1$ of distinct blue neighbors of $w_2$ in at most $(d_2)_2$ ways (we ignore the requirement that $x_1x_2$ and $y_1y_2$ are red). Thus,
 \begin{equation}
   \label{eq_back_upper}
   \deg_{B_k}(H')/(d_1 - k + 1)^2 \le (n_2 - 2d_1 + k - 1)(d_2)_2.
 \end{equation}
 For the lower bound, given $x_2, y_2$, we need to subtract from the upper bound \eqref{eq_back_upper} the number of choices of $w_2, x_1, y_1$ for which $x_1x_2$ or $y_1y_2$ is blue. In the former case, ignoring other constraints, there are at most $d_2$ choices of a blue  neighbor $x_1$ of $x_2$, at most $d_1$ choices of a blue neighbor $w_2$ of $x_1$, and at most $d_2$ choices of a blue neighbor $y_1$ of $w_2$. By symmetry, the total number of bad choices is at most $2d_1d_2^2$, whence
\begin{align*}
  \deg_{B_k}(H')/(d_1 - k + 1)^2 &\ge (n_2 - 2d_1 + k - 1)(d_2)_2 - 2d_1d_2^2 \\
  \justify{$k \ge 1$}	&\ge (n_2 - 2d_1 + k - 1)d_2^2\left( 1 - \frac{1}{d_2} - \frac{2d_1}{n_2 - 2d_1} \right) \\
  \justify{$d_1 = n_2p, p \le 1/5$} &\ge (n_2 - 2d_1 + k - 1)d_2^2\left( 1- d_2^{-1} - 4p \right).
\end{align*}
Since, by our assumptions, $d_2 \ge \hat p\hat n  \to \infty$  and $p \le 1/5$, it follows that the graphs $B_k$ for $k = 1, \dots, d_1$ have minimum degrees at least $1$. In particular, this means that starting with any $p$-biregular graph, by applying a certain number of switchings, we can obtain a $p$-biregular graph with arbitrary co-degree. Since we implicitly assume $\rcnnp$ is non-empty, this implies that all classes $\rc_k, k = 0, \dots, d_1,$ are non-empty, i.e., the numbers $r_k = |\rc_k|$ are all positive, satisfying one of the conditions of Claim~\ref{clm_Poiss_bounds}.

Using \eqref{eq_countingedges} and bounds on the degrees in $B_k$, for $k = 1, \dots, d_1,$ we have
\begin{align}
  \notag \frac{r_k}{r_{k-1}} \le \frac{\max_{H' \in \rc' }\deg_{B_k}(H')}{\min_{H \in \rc} \deg_{B_k}(H)} &\le \frac{(d_1 - k + 1)^2}{k(n_2 - 2d_1 + k)}\cdot\frac{1}{1 - 4p} \\
 \label{eq_rkUp} &\le \frac{(d_1 - k + 1)^2}{k(n_2 - 2d_1 + k)} \cdot \left( 1 + 20p \right),
\end{align}
(where the last inequality holds because $p \le 1/5$ implies $\frac{1}{1 - 4p} = 1 + \frac{4p}{1 - 4p} \le 1 + 20p$), and
\begin{equation}
  \label{eq_rkLow}
  \frac{r_k}{r_{k-1}} \ge \frac{\min_{H' \in \rc' }\deg_{B_k}(H')}{\max_{H \in \rc} \deg_{B_k}(H)} \ge \frac{(d_1 - k + 1)^2}{k(n_2 - 2d_1 + k)} \left( 1 - d_2^{-1}  -4p \right).
	\end{equation}
Let $\mu$ be a real number such that
\begin{equation}
  \label{eq_mu_ratio}
\frac{(d_1 - \mu + 1)^2}{\mu(n_2 - 2d_1 + \mu)} = 1.
\end{equation}
 After solving for $\mu$, we have
\begin{equation}
  \label{eq_mu_np_square}
	\mu = \frac{(d_1 + 1)^2}{n_2 + 2} \in [\:p^2n_2, p^2n_2(1 + 3d_1^{-1})\:].
	\end{equation}
  Define
	\begin{equation}
		\label{eq_mupm}
		\mu_+ := \mu(1 + 20p), \quad \text{and} \quad \mu_- := \mu(1 - d_2^{-1} -4p).
	\end{equation}
	From \eqref{eq_rkUp}, \eqref{eq_mu_ratio}, and \eqref{eq_mupm}, for $\mu \le k \le d_1$, we have
	\begin{equation*}
		\frac{r_k}{r_{k-1}} \le \frac{(d_1 - \mu + 1)^2\left( 1 + 20p \right)}{\mu (n_2 - 2d_1 + \mu)} \cdot \frac{\mu}{k} = \frac{\mu_+}{k}.
	\end{equation*}
	On the other hand, from \eqref{eq_rkLow}, \eqref{eq_mu_ratio}, and \eqref{eq_mupm} that, for $1\le k \le \mu$, we have
	\begin{equation*}
	  \frac{r_{k-1}}{r_k} \le \frac{\mu (n_2 - 2d_1 + \mu)}{(d_1 - \mu + 1)^2\left( 1 - 1/d_2 -4p \right)} \cdot \frac{k}{\mu} = \frac{k}{\mu_-}.
	\end{equation*}
Note that, since $\min\{d_1,d_2\} = p\hat n  \to \infty$ and $p\le1/5$, it follows by \eqref{eq_mu_np_square}, and \eqref{eq_mupm} that
\begin{equation}
\label{eq_mu_pm_ptwon}
  p^2n_2(1 - (p\hat n)^{-1} - 4p) \le \mu_-\le \mu_+\le p^2n_2(1 + 15(p\hat n)^{-1} + 20p)\le 6p^2n_2 .
\end{equation}
Therefore,
\begin{equation}
\label{eq_mu_pm_bounds}
  \mu_+ \le p^2n_2 + 20\left(p^3n_2 + \frac{pn_2}{\hat n} \right) \quad\text{and}\quad \mu_- \ge p^2n_2 - 20\left(p^3n_2 + \frac{pn_2}{\hat n} \right).
\end{equation}
Consequently, setting
\begin{equation}
  \label{eq_def_x}
	x := 20 \sqrt{\lambda p^2 n_2} + \lambda,
\end{equation}
we have, by \eqref{eq_mu_pm_bounds}
\begin{equation*}
     \prob{|X - p^2n_2| \ge 20\left(p^3n_2 + \frac{pn_2}{\hat n}\right)  + x} \le \prob{X \ge \mu_+ + x} + \prob{X\le \mu_- - x}.
   \end{equation*}
Noting that the last inequality in \eqref{eq_mu_pm_ptwon} implies $x \ge \sqrt{\mu_+n_2} + \lambda $, using Claim~\ref{clm_Poiss_bounds} we obtain~\eqref{eq_codegconcnoq}, completing the proof in the case $\I = 1$.

\medskip
\textbf{Case $\I = 0$.} For the switching argument used in the previous case it was crucial that $p$ was small, as otherwise several estimates would be negative and so meaningless. Therefore, it cannot be used now. Fortunately, in this case we are in a position to apply an asymptotic enumeration approach based on Theorem~\ref{thm_enum} and analogous to the proof of Theorem 2.1 in~\cite{KSVW}.

Recall the notation $\rcdd$ defined before Theorem~\ref{thm_enum}.
Writing $A = \Gamma_H(u_1)$ and $B = \Gamma_H(v_1)$, we note that every graph in $\rc_k$ induces an ordered partition of $V_2$ into four sets $A \cap B$, $A \setminus B$, $B \setminus A$ and $V_2 \setminus (A \cup B)$, of sizes, respectively, $k, d_1 - k, d_1 - k$, and $n_2 - 2d_1 + k$. There are exactly
$$\Pi(n_2,d_1,d_2,k):=\frac{n_2!}{k!(d_1 - k)!^2(n_2 - 2d_2 + k)!} $$
such partitions.
After removing vertices $u_1$ and $v_1$, we obtain a graph $H^* \in \rcdd$ on $(V_1\setminus\{u_1,v_1\}, V_2)$, with $\dd_1$ having all its $n_1 - 2$ entries equal to $d_1$, and entries of $\dd_2$ determined by the partition:  entries equal to $d_2 - 2$ on $A \cap B$, to $d_2 - 1$ on $A \triangle B$ and the remaining ones equal to $d_2$.
Since $H^*$ together with the ordered partition uniquely determines $H$, we have
\begin{equation*}
	|\rc_k| = \Pi(n_2,d_1,d_2,k)\times  |\rc(\dd_1,\dd_2)|.
\end{equation*}
Let us check that the three assumptions \ref{thm_enum_max}--\ref{thm_enum_dense} of Theorem~\ref{thm_enum} are satisfied. When doing so, we should remember that we have $n_1 - 2$ instead of $n_1$, but that the parameter $p = d_1/n_2$ remains intact.
With foresight, choose $a = 0.35$ and, say, $b = 0.1$, and, for convenience $C = 1$, which determine, via Theorem~\ref{thm_enum}, an $\epsilon > 0$.

Note that $\I = 0$ implies
\begin{equation}
  \label{eq_p4logN}
  \hat n \ge \frac{4\max \left\{ n_1,n_2 \right\}} {\log N} \quad \text{and} \quad \frac{n_1}{n_2} + \frac{n_2}{n_1} \le \frac{p \log N}{2} \le p \log \max \left\{ n_1, n_2 \right\}.
\end{equation}
As the degrees on one side are all equal and on the other side they differ from each other by at most $2$, assumption~\ref{thm_enum_max} holds true with a big margin.
Using $q \ge 1/2$ and the first inequality in \eqref{eq_p4logN}, we get
\[
  (pq)^2\min\{n_1 - 2,n_2\}^{1 + \epsilon} \ge \frac{p^2\hat n^{1 + \epsilon}}{4}(1 + o(1))  = \Omega \left( \frac{ \left(\max \left\{ n_1, n_2 \right\}\right)^{1 + \epsilon}}{(\log N)^{3 + \epsilon}}  \right) \gg \max \left\{ n_1, n_2 \right\},
\]
which implies assumption~\ref{thm_enum_balance} for large $\hat n$.
Finally, using elementary inequalities $(1 - 2p)^2 \le q$ (since $p \le 1/2$), $1 \le \left( n_1/n_2 + n_2/n_1 \right)/2$, and the second inequality in \eqref{eq_p4logN}, we obtain
\begin{align*}
  \frac{(1 - 2p)^2}{4pq} &\left(1 + \frac{5(n_1 - 2)}{6n_2} + \frac{5n_2}{6(n_1 - 1)}\right) \le \frac{1}{4p}\left( \frac{1}{2} + \frac{5}{6} \right)\left(\frac{n_1}{n_2} + \frac{n_2}{n_1}\right) (1 + o(1)) \\
  &\le \frac{1}{3}\log \max \left\{ n_1, n_2 \right\} (1 + o(1))  \ll 0.35 \cdot \log \max \left\{ n_1 - 2, n_2 \right\},
\end{align*}
which for large $\hat n$ implies assumption~\ref{thm_enum_dense} with $a = 0.35$.

Note that in \eqref{eq_enum} we have $D_1 = 0$, while $D_2 \le 4n_2 \ll pq n_1n_2$, by the assumption $\hat p\hat n\to\infty$.
Thus, uniformly over $k$, the exponent in \eqref{eq_enum} is $-1/2 + o(1)$ and, by Theorem~\ref{thm_enum},
\begin{equation*}
	|\rc_k| \sim
	\frac {e^{-1/2}n_2!\binom {n_2}{d_1}^{n_1 - 2} \binom {n_1 - 2}{d_{2} - 2}^{k}\binom {n_1 - 2}{d_{2} - 1}^{2d_1 - 2k}\binom {n_1 - 2}{d_{2}}^{n_1 - 2d_1 + k}} {k!(d_1 - k)!^2(n_2 - 2d_1 + k)!\binom{(n_1 - 2)n_2}{(n_1 - 2)d_1}} =: r_k.
\end{equation*}
Straightforward calculations yield
\begin{equation*}
\frac{r_k}{r_{k-1}} =
\frac {(d_1 - k + 1)^2(n_1 - 2 - d_2 + 1)(d_2 - 1)} {k(n_2 - 2d_1 + k)\cdot d_2(n_1 - d_2)} \le \frac{(d_1 - k + 1)^2}{k(n_2 - 2d_1 + k)}
\end{equation*}
and, because $d_2/n_1 = p\le1/2$,
\begin{equation*}
\frac{r_k}{r_{k-1}} \ge \frac{(d_1 - k + 1)^2}{k(n_2 - 2d_1 + k)}\left(1 - 2d_2^{-1} \right).
\end{equation*}
(Compare with \eqref{eq_rkUp} and \eqref{eq_rkLow} to note the absence of $\Theta(p)$ error terms.)
With  $\mu$ as in \eqref{eq_mu_np_square}, we redefine
\begin{equation}
  \label{mudef}
\mu_+ := \mu\quad\text{and}\quad\mu_- := \mu (1 - 2d_2^{-1}).
\end{equation}
From \eqref{eq_mu_np_square} and \eqref{mudef} it follows that
\begin{equation}
  \label{eq_mu_plus}
	\mu_{+} = \mu \le p^2n_2 (1 + 3d_1^{-1}) \le p^2n_2 + \frac{3pn_2}{n_1}\le p^2n_2 + \frac{20pn_2}{\hat n}
\end{equation}
and
\begin{equation}
	\label{eq_mu_minus}
	\mu_{-} = p^2n_2 (1 - 2d_2^{-1}) \ge p^2n_2 - 2p \ge  p^2n_2 - \frac{20pn_2}{\hat n}.
\end{equation}
From \eqref{eq_mu_plus} it follows $\mu_+ \le 6p^2n_2$, which implies that $x$ defined in \eqref{eq_def_x} satisfies $x \ge \sqrt{2\mu_+\lambda} + \lambda$.
Consequently, using \eqref{eq_mu_plus}, \eqref{eq_mu_minus} and Claim~\ref{clm_Poiss_bounds} we infer
\begin{equation*}
    \prob{|X - p^2n_2| \ge \frac{20pn_2}{\hat n} + x} \le \prob{X \ge \mu_+ + x} + \prob{X\le \mu_- - x} = O\left( \sqrt{N} e^{-\lambda} \right).
\end{equation*}
We have obtained \eqref{eq_codegconcnoq} in the case $\I = 0$.
\end{proof}

\subsection{Degrees and co-degrees in \texorpdfstring{$\R(t)$}{R(t)}}
\label{ss_Rt_cod}
For convenience, having fixed $H\in\rcnnp$, we will denote by $\Prob_H$ and $\E_H$ the conditional probability and expectation with respect to the event $\R = H$. In the conditional space defined by such an event, $(\eta_1, \dots, \eta_M)$ is just a uniformly random permutation of the edges of $H$ and therefore $\R(t)$ is a uniformly random $t$-subset of edges of $H$.

We first show that the degrees of vertices in the process $\R(t)$ grow proportionally almost until the end.
Recall that $d_1 = pn_2,\; d_2 = pn_1$, while $\tau=\tau(t)$ and $\tau_0$ are defined, respectively, in \eqref{eq_tau} and \eqref{eq_tauzero}.

\begin{lemma}
\label{lem_degrees}
If $\lambda = \lambda(n_1,n_2) \le \tau_0 p \hat n$, then for $i = 1, 2$, with probability $1 - O(N^2e^{-9\lambda/4})$
\begin{equation}
  \forall t < t_0 \quad \forall v \in V_i \quad |\deg_{\R(t)}(v) - (1 - \tau) d_i| \le 3 \sqrt{\lambda \tau d_i}.
\label{eq_degupper}
\end{equation}
\end{lemma}
\begin{proof}
  Fix $t < t_0$ and $v\in V_i$ and set $X_t(v) := \deg_{\R \setminus \R(t)}(v)$. For any $H \in \rcnnp$, conditioning on $\R = H$, the random variable $X_t(v) \sim \Hyp(M, d_i, M - t)$ has a hypergeometric distribution. Since the distribution does not depend on $H$, $X_t(v)$ has the same distribution unconditionally. In view of \eqref{eq_tau}, $\E X_t(v) = \tau d_i$, therefore combining \eqref{eq_hypergeomUpTailExp} and \eqref{eq_hypergeomLowTailExp}, for all $x > 0$,
\begin{multline}
  \label{Xtv_bound}
\prob{|X_t(v) - \tau d_i| \ge x}
\le 2 \exp \left\{ - \frac{x^2}{2\left(\tau d_i  + x/3 \right)} \right\}= 2 \exp \left\{ - \frac{x^2}{2\tau d_i\left(1  + x/(3\tau d_i) \right)} \right\}.
\end{multline}
Let $x := 3\sqrt{\lambda\tau d_i}$. Since $\tau \ge \tau_0$, by the assumption $\lambda\le \tau_0 p\hat n $,
\begin{equation}
	\label{eq_xtau}
	x/(3\tau d_i) = \sqrt{\lambda/(\tau d_i)} \le  \sqrt{\tau_0 \hat n p/(\tau_0 d_i)} \le 1.
\end{equation}
Consequently, taking the union bound over all $(1 - \tau_0)Mn_i \le N^2$ choices of $t$ and $v$ and using \eqref{Xtv_bound} and \eqref{eq_xtau},
we conclude that  \eqref{eq_degupper} fails with probability at most
\begin{equation*}
2N^2\exp \left\{ - \frac{x^2}{2\tau d_i\left(1  + x/(3\tau d_i) \right)} \right\}
\le 2N^2\exp \left\{ -\frac{9 \tau d_i \lambda}{4\tau d_i} \right\} = 2N^2e^{-9\lambda/4}.
\end{equation*}
\end{proof}

In the proof of Lemma~\ref{lem_ef}, to have a pseudorandom-like property of $\R \setminus \R(t)$ for $p > 0.49$, we will need a bound on the upper tail of co-degrees in $\R(t)$. Recall the definition of co-degree from \eqref{eq_def_cod}.

\newcommand{\constcodegsRt}{15}
\begin{lemma}
	\label{lem_codegsRt}
	Assume that $n_1 \ge n_2$ and $p > 0.49$.
	If $\lambda = \lambda(n_1,n_2) \ge \log N$, then,  with probability $1 - O(N^3e^{-\lambda/3 })$,
	\begin{equation}
		\label{eq_codegsRt}
		\forall t \le M \quad \forall u_1 \neq v_1 \quad\cod_{\R(t)}(u_1,v_1) \le(1 - \tau)^2p^2 n_{2} + 20 q^3n_2\I + \constcodegsRt\sqrt{\lambda n_{2}}.
	\end{equation}
\end{lemma}
\begin{proof}
If $\lambda \ge n_2$, inequality \eqref{eq_codegsRt} holds trivially, so let us assume $\lambda < n_2$.
 We can condition on $\R = H$ satisfying
	\begin{align}
		\notag 		\forall u_1,v_1 \in V_1 \quad\cod_{H}(u_1,v_1) & \le p^2n_2+ 20\left(\hat p^3n_2\I + n_2\hat p/\hat n + \sqrt{\hat p^2\lambda n_2}\right)  + \lambda\\
		\notag \justify{$\hat p \le \min \{q, 0.5\}, n_2 = \hat n$}&\le p^2n_2+ 20\left(q^3n_2\I + 0.5 + 0.5\sqrt{\lambda n_2}\right)  + \lambda\\
		\label{eq_codHconc}
		\justify{$1 \ll \lambda < n_2$} &\le p^2n_2 + 20q^3n_2 \I + 12\sqrt{\lambda n_2}.
	\end{align}
Indeed, the probability of the opposite event can be bounded, by Lemma~\ref{lem_codegs} and the union bound,  by $O\left(n_1^2\sqrt Ne^{-\lambda}\right) \ll N^3e^{-\lambda}$. Taking this bound into account, it thus suffices to show that if $H$ satisfies \eqref{eq_codHconc}, then for any distinct $u_1, v_1\in V_1$ and $t \le M$
\begin{equation}
		\label{enough}
		\probH{\cod_{\R(t)}(u_1,v_1) > (1 - \tau)^2p^2 n_{2} + 20 q^3n_2\I + \constcodegsRt\sqrt{\lambda n_{2}}} = 2e^{-\lambda/3},
	\end{equation}
	since then the proof is completed by a union bound over $O(N^3)$ choices of $t, u_1, v_1$.

	Fix $t \le M, u_1 \neq v_1$ and let $Y := \left| \Gamma_{\R(t)}(u_1) \cap \Gamma_{H}(v_1)\right|$ and  $\cod_H := \cod_H(u_1, v_1)$. The distribution of $Y$ conditioned on $\R = H$ is hypergeometric $\Hyp(M, \cod_H, t)$, and hence, by~\eqref{eq_codHconc},
	\[
		\mu_Y := \E_H{Y} = \frac{t\cod_H}{M} = (1 - \tau) \cod_H \le (1 - \tau)p^2n_2 + 20q^3n_2 \I + 12 \sqrt{\lambda n_2}.
	\]
	Using \eqref{eq_hypergeomUpTailExp}, a trivial bound $\mu_Y \le n_2$, and our assumption $\lambda < n_2$,
	\[
		\probH{ Y \ge \mu_Y + \sqrt{\lambda n_2}} \le \exp\left\{-\frac{\lambda n_2}{2(\mu_Y + \sqrt{\lambda n_2}/3)}\right\} \le e^{-3\lambda/8}\le e^{-\lambda/3}.
	\]
	Since also trivially $Y \le \min\{d_1, t\}$, we have shown that given $\R = H$, with probability at least $1 - e^{-\lambda/3}$,
\begin{equation}
  \label{eq_Y}
	Y \le y_0 := \min \{\mu_Y + \sqrt{\lambda n_2}, d_1, t\} \le (1 - \tau)p^2n_2 + 20q^3n_2 \I + 13 \sqrt{\lambda n_2}.
\end{equation}

Let $\ec_\eqref{enough}$  be the event that the inequality in \eqref{enough}  holds. By the law of total probability,
\begin{align*}
	\probH{\ec_\eqref{enough}} &= \sum_y \PcH{\ec_\eqref{enough}}{Y = y}\probH{Y = y} \\
	\justify{\eqref{eq_Y}} &\le \sum_{y \le y_0}\PcH{\ec_\eqref{enough}}{Y = y}\probH{Y = y} + e^{-\lambda/3}.
\end{align*}
Hence, to prove \eqref{enough}, it suffices to show that, for any $ y = 0, \dots, y_0$,
\begin{equation}
	\label{enougher}
\PcH{\ec_\eqref{enough}}{Y = y} \le e^{-\lambda/3}.
\end{equation}
Fix an integer $y \in [0, y_0]$ and a set $S\subseteq V_2$ of size $|S| = y \le y_0$. Under an additional condi\-tioning $\Gamma_{\R(t)}(u_1)\cap \Gamma_H(v_1) = S$, set $\R(t)$ is the union of a fixed $y$-element set $\{u_1w : w \in S\}$ and a random $(t - y)$-element subset of $E(H) \setminus \{u_1w : w \in \Gamma_H(u_1)\cap \Gamma_H(v_1)\}$.
Thus, in this conditional space, $X := \cod_{\R(t)}(u_1,v_1)$ counts how many of these $t-y$ random edges fall into the set $\{v_1w : w \in S\}$ and therefore $X \sim \Hyp(M - \cod_H, y, t - y)$. Moreover, the distribution of $X$ is the same for all $S$ of size $y$, so $X$ has the same distribution when conditioned on $Y = y$. In particular,
	\begin{align*}
	  \mu_X := \EcH{X}{Y = y} &= \frac{y(t-y)}{M-\cod_H} \\
	  \justify{$y \ge 0$} &\le \frac{yt}{M} \left(1 + \frac{\cod_H}{M - \cod_H}\right) \\
	  \justify{$M = pn_1n_2, \cod_H \le pn_2$} &\le (1-\tau) y \left(1 + \frac{1}{n_1 - 1}\right) \\
		\justify{$y \le y_0$ and \eqref{eq_Y}}	&\le (1-\tau)^2p^2n_2 + 20q^3n_2 \I + 13 \sqrt{\lambda n_2} + \frac{n_2}{n_1 - 1}.
		\\
	\justify{$n_1 \ge n_2$}	&\le (1-\tau)^2p^2n_2 + 20q^3n_2 \I + 14 \sqrt{\lambda n_2}.
	\end{align*}
	Using again \eqref{eq_hypergeomUpTailExp} as well as inequalities $\mu_X \le n_2$ and $\lambda \le n_2$, we infer that
	 \begin{align*}
\PcH{X \ge \mu_X + \sqrt{\lambda n_2}}{Y = y}  \le \exp\left\{-\frac{\lambda n_2}{2(\mu_X + \sqrt{\lambda n_2}/3)}\right\} \le e^{-\lambda/3} ,
\end{align*}
which, together with the above upper bound on $\mu_X$, implies \eqref{enougher}. This completes the proof of Lemma~\ref{lem_codegsRt}.
\end{proof}

\subsection{Proof of Lemma~\ref{lem_thetas}}
\label{ss_proof_thetas}
Recall the definitions of $\tau=\tau(t)$ in \eqref{eq_tau}, $\delta(t)$ in \eqref{eq_deltat}, and $\lambda(t)$ in~\eqref{eq_lambda}.
In this proof we utilize some technical bounds on $\tau:=\tau(t)$ and $\gamma_t$ proved in Section~\ref{sec_technical} (see Proposition \ref{tech}).
In particular, the  bound \eqref{eq_gammat_less1} on $\gamma_t$, together with \eqref{eq_deltat_gammat}, implies that
\begin{equation}
	\deltat \le 1/9\le 1.
	\label{eq_delta_upper}
\end{equation}
(We do not even need to remember what $\gamma_t$ is to see that.)

We now derive an upper bound on $\lambda(t)$.
By \eqref{eq_taupn}, with a huge margin,
\[
	6 \log N \le \frac{\tau \hat p \hat n}{16(C + 3)} \le \frac{\tau pq\hat n }{8(C+3)}.
\]
On the other hand, for $p > 0.49$, squaring and rearranging the inequality \eqref{eq_tauq_lower} implies
\begin{equation*}
	\frac{64\log N}{\tau p q} \le
	\frac{\tau pq\hat n }{8(C+3)}.
\end{equation*}
Summing up, for any $p$,
\begin{equation}
  \label{eq_lambda_upper}
\lambda(t)\le \frac{\tau pq \hat n}{4(C + 3)}.
\end{equation}

Without loss of generality, we assume that $i = 1$. Let $\fc_t$ be the family of graphs $H$ satisfying, for any distinct $u_1, v_1 \in V_1$,
\begin{equation}
  \label{eq_Sienkiewicz}
	|d_1 - \cod_H(u_1, v_1) - pqn_2|
	\le 20\hat p n_{2} \left(\hat p^2\I + \frac{1}{\hat n}  + \sqrt{\frac{(C+3)\lambda(t)}{n_{2}}} + \frac{(C+3)\lambda(t)}{20\hat pn_2}\right).
\end{equation}
Lemma~\ref{lem_codegs} with $\lambda = (C+3)\lambda(t)$ and a union bound over the $O(N^2)$ choices of $u_1, v_1$ imply
\begin{equation}
	\label{eq_good_codegs}
	\prob{\R \in \fc_t} = 1 - O(N^{2.5} e^{-(C+3)\lambda(t)}) \By{\lambda(t) \ge 6\log n}{\ge} 1 - O(N^{-(C+1)}e^{-2\lambda(t)}).
\end{equation}

Writing $\delta_* = \sqrt{\frac{(C+3)\lambda(t)}{\tau p q \hat n}}$ and noting that \eqref{eq_lambda_upper} implies $\delta_* \le 1/2$, the sum of the last three terms in the parentheses in \eqref{eq_Sienkiewicz} is at most
\begin{equation*}
  2 \cdot \sqrt{\frac{(C+3)\lambda(t)}{\hat n}} + \frac{(C+3)\lambda(t)}{20\hat p \hat n} {\le} 2\sqrt{pq}\delta_* + \frac{pq\delta_*^2}{20\hat p}   {\le}  (1 + 1/40) \delta_*\le  \frac{3\delta_*}{\sqrt{6}}.
\end{equation*}
In addition, the factor in front of the brackets is at most $40pq n_2$. Thus, \eqref{eq_Sienkiewicz} implies
\begin{equation}
  \label{eq_Tuwim}
  |d_1 - \cod_H(u_1, v_1) - pqn_2| \le 40 pqn_2  \left( \hat p^2 \I + \frac{3\delta_*}{\sqrt{6}} \right) = pqn_2 \cdot \deltat/3.
\end{equation}
	
We claim that for $t  = 0, \dots, t_0 - 1$
\begin{equation}
  \label{eq_worst_tail}
		\max_{H \in \fc_t}\probH{\max_{u_1 \ne v_1} \left|\frac{\theta_t(u_1,v_1)}{\tau pqn_2 } - 1\right| \ge \deltat} \le 2N^2e^{-(C+3)\lambda(t)},
	\end{equation}
	and deferring its proof to the end we first show how \eqref{eq_good_codegs} and \eqref{eq_worst_tail} imply the lemma.

Inequalities \eqref{eq_worst_tail} and \eqref{eq_good_codegs} imply
\begin{align}
\label{eq_Epit}
	\notag \prob{\max_{u_1 \ne v_1} \left|\frac{\theta_t(u_1,v_1)}{\tau pqn_2 } - 1\right| \ge \deltat} &\le 2N^2e^{-(C+3)\lambda(t)} + \prob{\R \notin \fc_t} \\
      &= O(N^{-(C+1)}e^{-2\lambda(t))}).
\end{align}
Consider random variables, for $t = 0, \dots, t_0 - 1$,
	\[
		Y_t := \Pc{\max_{u_1 \ne v_1} \left|\frac{\theta_t(u_1,v_1)}{\tau pqn_2 } - 1\right| > \deltat}{\R(t)}.
	\]
	Using Markov's inequality and \eqref{eq_Epit},
   we infer that
	\begin{align*}
	  \prob{Y_t > e^{-2\lambda(t)}} &\le e^{2\lambda(t)}\E Y_t = e^{2\lambda(t)}\prob{\max_{u_1 \ne v_1} \left|\frac{\theta_t(u_1,v_1)}{\tau pqn_2 } - 1\right| \ge \deltat} = O( N^{-(C+1)}),
	\end{align*}
	which, taking the union bound over the $O(N)$ choices of $t$, implies that \eqref{eq_condthetas} holds with the desired probability, completing the proof of lemma.

	\smallskip
	Returning to the proof of \eqref{eq_worst_tail}, fix $t < t_0, H \in \fc_t$ and two distinct vertices $u_1,v_1\in V_1$. Conditioning on $\R = H$, note that from \eqref{eq_theta_t} that random variable $X := \theta_t(u_1,v_1) = |\Gamma_{H \setminus \R(t)}(u_i)\cap\Gamma_{K\setminus H}(v_i)|$ counts elements in the intersection of two subsets of $E(H)$: a fixed set
	$  \left\{ (u_i,w) : w \in \Gamma_H(u_i) \cap \Gamma_{K \setminus H}(v_i) \right\}$ of size $d_1 - \cod_H(u_1, v_1)$
 and a random set $H \setminus \R(t)$ of size $M - t$. Hence $X \sim \Hyp(M, d_1 - \cod_H(u_1,v_1), M -t)$ has the hypergeometric distribution with expectation
\begin{equation*}
	\mu_X := \E X = (d_1 - \cod_H(u_1,v_1))(M - t)/M  = \tau (d_1 - \cod_H(u_1,v_1)).
\end{equation*}
Note that by \eqref{eq_Tuwim},
\begin{equation}
  \label{eq_theta_exp}
|\mu_X - \tau p q n_2| \le \tau pqn_2 \cdot \deltat/3.
\end{equation}
Let $\lambda^* := (3C + 9)\lambda(t)$. From \eqref{eq_theta_exp},  \eqref{eq_delta_upper}, and \eqref{eq_lambda_upper} it follows that
\begin{equation}
  \label{eq_muHlower}
	\mu_X \ge \frac{26}{27}\tau pqn_{2} \ge \lambda^*.
\end{equation}
Note that \eqref{eq_deltat} implies
  \begin{equation}
    \label{eq_deltasquared}
  \delta(t)^2\ge7200\frac{\lambda^*}{\tau pqn_2}\ge 10\frac{\lambda^*}{\tau pqn_2}.
  \end{equation}
By \eqref{eq_hypergeomUpTailExp}, \eqref{eq_hypergeomLowTailExp}, and \eqref{eq_muHlower},
\begin{equation*}
  \begin{split}
    \probH{|X - \mu_X| \ge \sqrt{\mu_X \lambda^*} } &\le 2 \exp \left\{ - \frac{\lambda^*}{2\left(1  + \frac{1}{3}\sqrt{\lambda^*/\mu_X } \right)} \right\} \\
    &\le 2 e^{-3\lambda^*/8}\le 2 e^{-\lambda^*/3} = 2e^{-(C + 3)\lambda(t)}.
  \end{split}
\end{equation*}
Thus, with probability $1 - 2e^{-(C+3)\lambda(t)}$
\begin{align*}
  |\theta_t(u_1,v_1) - \tau p q n_2| &\le |X - \mu_X| + |\mu_X - \tau p q n_2| \\
  &\le \sqrt{\mu_X \lambda^*} + |\mu_X - \tau p q n_2| \\
  \justify{\eqref{eq_theta_exp}, \eqref{eq_deltasquared}} &\le \sqrt{\left( 1 + \delta(t)/3 \right)\tau p q n_2 \cdot \delta^2(t)\tau p q n_2 / 10} + (\delta(t)/3) \tau p q n_2\\
  \justify{\eqref{eq_delta_upper}} &\le \delta(t)\tau pqn_2 (\sqrt{(1 + 1/27)/10} + 1/3) \\
  &\le \delta(t)\tau pqn_2.
\end{align*}
Hence, applying also the union bound over all $n_1^2 \le N^2$ choices of $u_1,v_1$,  we infer \eqref{eq_worst_tail}.
\qed

\section{Alternating cycles in regularly 2-edge-colored jumbled graphs}
\label{sec_quasi}

The ultimate goal of this section is to prove Lemma~\ref{lem_ef}. While for $p\le0.49$ the proof follows relatively easily from Lemma~\ref{lem_degrees} by a standard switching technique, the case $p > 0.49$ is much more involved. To cope with it, we first study the existence of
alternating walks and cycles in a class of 2-edge-colored pseudorandom graphs.

In Subsection~\ref{ss_jumbled},  we define an appropriate notion of pseudorandom bipartite graphs (\emph{jumbledness}), inspired by a similar notion introduced implicitly by Thomason in~\cite{T89}. We show that for $p > 0.49$ and suitably chosen parameters, the random graph $K\setminus \R(t)$ is jumbled with high probability (Lemma~\ref{lem_quasiRt}).

 The next two subsections are devoted to 2-edge-colored jumbled graphs which are almost regular in each color. After proving a technical Lemma~\ref{lem_ananas_prep} in Subsection~\ref{ss_blue_red}, in Subsection~\ref{ss_alt_walks_cycle} we show the existence of alternating short walks between any two vertices of almost regular 2-edge-colored jumbled graphs (Lemma~\ref{lem_walks}).

 An immediate consequence of Lemma~\ref{lem_walks} is Lemma~\ref{lem_alternating}, which states that every edge belongs to an alternating short cycle. The latter result together with a standard switching argument (Proposition~\ref{prop_switch}) will be used in the proof of  Lemma~\ref{lem_ef} for $p > 0.49$.
 That proof, for both cases $p\le 0.49$ and $p > 0.49$,  is presented in Subsection~\ref{ss_ef_proof}.

\subsection{Jumbled graphs}
\label{ss_jumbled}
Let $K:= \Knn$ be the complete bipartite graph with partition classes $V_1$ and $V_2$, where $|V_i| = n_i$.
Given a bipartite graph $F \subseteq K$ and two subsets $A \subseteq V_1$, $B \subseteq V_2$, denote by $e_F(A,B)$ the number of edges of $F$ between $A$ and $B$.
Recall that $N = n_1n_2$ and $M = pN$.

Given real numbers $\pi,\delta\in(0,1)$, we say that a graph $F \subseteq K$ is \emph{$(\pi, \delta)$-jumbled} if
for every $A \subseteq V_1$ and $B \subseteq V_2$
\begin{equation*}
  |e_F(A,B) - \pi|A||B|| \le \delta \sqrt{N|A||B|}.
\end{equation*}

The following result of Thomason~\cite[Theorem 2]{T89}, which quantifies a variant of jumbledness in terms of the degrees and co-degrees of a graph, will turn out to be crucial for~us.

\begin{theorem}[\cite{T89}]
  \label{thm_Thomason}
  Let $F \subseteq K$ be a bipartite graph and  $\rho \in (0,1)$ and $\mu \ge 0$ be given. If
	\begin{equation}
\label{eq_Thomason_assumption}
		\min_{v \in V_1} \deg_F(v) \ge \rho n_2\quad\text{and}\quad\max_{u,v\in V_1} \cod_F(u,v) \le \rho^2n_2 + \mu,
	\end{equation}
then, for all $A \subseteq V_1$ and $B \subseteq V_2$,
	\begin{equation*}
		|e_F(A,B) - \rho|A||B|| \le \sqrt{(\rho n_2 + \mu |A|)|A||B|} + |B|\I_{\left\{ |A|\rho < 1 \right\}}.
	\end{equation*}
\end{theorem}

\begin{remark}\rm
	The proof in~\cite{T89} is given only in the case $n_1 = n_2 = n$.
 However, it carries over in this more general setting, as in~\cite{T89} $n$ always refers to $|V_2|$.
\end{remark}

Recall that $K\setminus \R(t)$ has precisely $N - t = N - (1-\tau)M = (\tau p + q)N$ edges.
The following technical result states that, under the conditions of Lemma~\ref{lem_Stanislaw}, with high probability, $K\setminus \R(t)$ is jumbled for parameters which are tailored for Lemma~\ref{lem_alternating}.

\begin{lemma}
	\label{lem_quasiRt}
Let  $\alpha = \min \{ \tau p, q \}$ and $\pi := \tau p + q$. For every constant $C > 0$ and $p > 0.49$, if assumptions \eqref{eq_hatpI_ubound} and \eqref{eq_ef_qlbound} hold,	then, with probability $1 - O(N^{-C })$, for all $t < t_0$
\begin{equation*}
	K \setminus \R(t)\quad \text{is} \quad(\pi, \alpha/16)\text{-jumbled}.
\end{equation*}
and for any $e \in K \setminus \R(t)$
\begin{equation*}
  K \setminus (\R(t) \cup \{ e \}) \quad \text{is} \quad(\pi, \alpha/16)\text{-jumbled}.
\end{equation*}
\end{lemma}
\begin{proof}
  W.l.o.g., we assume that $n_1\ge n_2$. Let $\lambda := 3(C + 4)\log N$, and
		\begin{equation}
		  \label{eq_cond_quasiRt}
		\delta := 20 \left( \lambda/n_2 \right)^{1/4} + 10q^{3/2}\I.
		\end{equation}
The plan is to show that
\begin{equation}
	\delta \le \frac{\alpha}{16}
	\label{eq_delta_alpha}
\end{equation}
and that with the correct probability, $K \setminus \R(t)$ and $K \setminus \R(t) \cup \{e\}$ are in fact $(\pi, \delta)$-jumbled.

We start with the proof of \eqref{eq_delta_alpha}.
Notice that,
by \eqref{eq_ef_qlbound},
\begin{equation}
  \label{lnq}
\sqrt{\lambda/{n_2}}\le\left(\lambda/{n_2}\right)^{1/4}\le\frac q{680}\le\frac\pi{680}\le\frac1{680}.
\end{equation}
Since $\hat p > \tfrac{49}{51}q$,
condition \eqref{eq_hatpI_ubound} implies that
$1/320 \ge q^{1/2} \I$. After multiplying both sides by  $10q$, we get
\begin{equation*}
  \frac q{32} \ge 10\cdot q^{3/2} \I,
\end{equation*}
which together with the second inequality in \eqref{lnq} implies
\begin{equation*}
	\frac{q}{16} = \frac{q}{32} + \frac{q}{32} \ge 20 \left( \lambda/{n_2} \right)^{1/4} + 10\cdot q^{3/2} \I = \delta.
\end{equation*}
On the other hand, using $p > 0.49$, $\tau \ge \tau_0$  and the definition \eqref{eq_tauzero} of $\tau_0$, we infer that
\begin{align*}
  \frac{\tau p}{16} &\ge \frac{0.49 \cdot \tau_0}{16} = \frac{0.49 \cdot 700\cdot (3(C + 4))^{1/4}}{16} \left( \left(\frac{\log N}{n_2}\right)^{1/4} + q^{3/2} \I\right) \\
	&\ge 20 \left( \lambda/{n_2} \right)^{1/4} + 10q^{3/2} \I = \delta.
\end{align*}
Hence $\alpha/16 = \min \left\{ \tau p/16, q/16 \right\} \ge \delta$, implying \eqref{eq_delta_alpha}.

\smallskip
We now prove the jumbledness, first focusing on $K \setminus \R(t)$ and then indicating the tiny change in calculation for $K \setminus (\R(t) \cup \{e\})$. Fixing an arbitrary $t < t_0$, we will first show that, with probability $1 - O(N^{-C -1})$, conditions \eqref{eq_Thomason_assumption} of Theorem~\ref{thm_Thomason} are satisfied by $F = K \setminus \R(t)$ and $F = K \setminus (\R(t) \cup \{e\})$ for suitably chosen $\rho$ and $\mu$. Then we will apply Theorem~\ref{thm_Thomason} to deduce that $K \setminus \R(t)$ is $(\pi, \delta)$-jumbled.
Lemma~\ref{lem_quasiRt} will follow by applying the union bound over all (at most $t_0\le M\le N$) choices of $t$.

By \eqref{eq_taupn},  $\lambda \le \tau_0 p \hat n$ with a big room to spare.
Note that \begin{equation}
  \label{pip}
  (1-\tau)p = 1 - \pi,
\end{equation}
which implies that  $|\deg_{K \setminus \R(t)}(v) - \pi n_2|=|\deg_{\R(t)}(v) - (1-\tau)d_1|$.
Hence, by Lemma~\ref{lem_degrees}, with probability $1 - O(N^2e^{-\lambda})$,
\begin{equation}
  \label{eq_deg_Rt}
  \max_{v \in V_1} |\deg_{K \setminus \R(t)}(v) - \pi n_2|
  \le 3\sqrt{\tfrac{4}{9}\lambda\tau d_1} \le 2\sqrt{\lambda n_2}
    \le 3\sqrt{\lambda n_2}.
\end{equation}

Moreover, recalling that $n_1\ge n_2$ and, again using \eqref{pip},
	Lemma~\ref{lem_codegsRt} implies  that, with probability $1 - O(N^3e^{-\lambda/3})$,
	\begin{equation}
	  \label{eq_codeg_Rt}
	  \max_{u, v \in V_1, u \ne v} \cod_{\R(t)}(u,v) \le (1 - \pi)^2 n_2 + 20q^3n_2\I + \constcodegsRt \sqrt{\lambda n_2}.
	\end{equation}
	Since $\lambda = 3(C + 4)\log N$, the intersection of events \eqref{eq_deg_Rt} and \eqref{eq_codeg_Rt} holds with probability $1 - O(N^3e^{-\lambda/3}) = 1 - O(N^{-C - 1})$.
	
Note that for distinct $u, v \in V_1$, by \eqref{eq_coco} and \eqref{eq_deg_Rt}
	\begin{equation}
	  \label{eq_cod_compl}
	  \cod_{K \setminus \R(t)}(u, v) \le \cod_{\R(t)}(u,v) + (2 \pi - 1)n_2 + 6\sqrt{\lambda n_2},
	\end{equation}
	which, by \eqref{eq_codeg_Rt}, implies that
	\begin{equation}
	  \label{eq_codeg_KRt}
	  \max_{u, v \in V_1, u \ne v} \cod_{K \setminus \R(t)}(u,v) \le \pi^2 n_2 + 20q^3n_2\I + (\constcodegsRt + 6) \sqrt{\lambda n_2}.
	\end{equation}
	Set
\[
  \rho := \pi - 3\sqrt{\lambda/n_2} \quad \text{and} \quad \mu :=  20q^3n_2\I + (\constcodegsRt + 12) n_2\sqrt{\lambda/ n_2},
\]
and note that by the inequality $\pi\le1$ and by \eqref{lnq}, we have $0 < \rho < 1$. Furthermore, $\rho^2\ge\pi^2-6\sqrt{\lambda/ n_2}$.
Hence, \eqref{eq_deg_Rt} and \eqref{eq_codeg_KRt} imply the  assumptions \eqref{eq_Thomason_assumption} for $F = K \setminus \R(t)$ with the above $\rho$ and $\mu$.
Consequently, by Theorem~\ref{thm_Thomason} (using $a \le n_1$ and $\rho\le\pi$),
		\begin{equation*}
		  |e_{K \setminus \R(t)}(A,B) - \rho ab| \le \sqrt{ \left( \pi n_2 + 20q^3n_1n_2\I + (\constcodegsRt + 12)n_1n_2\sqrt{\lambda /n_2}\right)ab} + b.
	\end{equation*}
	Further, since $N=n_1n_2$, $n_1\ge n_2 \ge b$ and $a,\lambda \ge 1$, we have, with a big margin,
\[
  \pi \le1 \le n_1 \sqrt{\lambda/n_2}\quad\text{and}\quad b \le \sqrt{b n_2} \le \sqrt{Nab}(\lambda/n_2)^{1/4}.
\]
 It follows, applying the inequality $\sqrt{x + y} \le \sqrt x+ \sqrt y$ as well as \eqref{eq_cond_quasiRt}, that
	\begin{equation}
\label{eq:e_rho}
		|e_{K \setminus \R(t)}(A,B) - \rho ab| \le \left( (\sqrt{\constcodegsRt + 13} + 1)(\lambda/n_2)^{1/4} + \sqrt{20} q^{3/2}\I\right) \sqrt{Nab} \le \frac{\delta}{2}\sqrt{Nab}.
	\end{equation}
	Moreover, note that using $ab \le n_1n_2 = N$ and the first inequality in \eqref{lnq},
\begin{equation}
\label{eq:pi_rho}
(\pi - \rho)ab = 3\sqrt{\lambda/n_2} \cdot ab \le 3\left(\lambda/n_2\right)^{1/4} \cdot \sqrt{Nab} \le \frac{\delta}{2}\sqrt{Nab},
\end{equation}
Hence, \eqref{eq:e_rho} and \eqref{eq:pi_rho} imply
\begin{equation*}
	|e_{K \setminus \R(t)}(A,B) - \pi ab| \le |e_{K \setminus \R(t)}(A,B) - \rho ab| + (\pi - \rho)ab \le \delta \sqrt{Nab},
\end{equation*}
meaning that $K \setminus \R(t)$ is $(\pi, \delta)$-jumbled.

If above we replace $K \setminus \R(t)$ by $K \setminus (\R(t) \cup \left\{ e \right\})$, the upper bound in \eqref{eq_deg_Rt} and the bound in \eqref{eq_codeg_KRt} still hold trivially, while the lower bound $\pi n_2 - 3 \sqrt{\lambda n_2}$ in \eqref{eq_deg_Rt} remains correct, since $\deg_{K \setminus (\R(t) \cup \{e\})}(v) \ge \deg_{K \setminus \R(t)}(v) - 1$ and we have plenty of room in \eqref{eq_deg_Rt}. Hence Theorem~\ref{thm_Thomason} applies with the same $\rho$ and $\mu$, implying that $K \setminus (\R(t) \cup \{e\})$ is also $(\pi, \delta)$-jumbled.

\end{proof}

\subsection{A technical inequality for blue-red graphs}
\label{ss_blue_red}
We find it convenient to introduce \emph{relative} counterparts of basic graph quantities. Below $i\in\{1,2\}$. As before, let $K:= \Knn$ be the complete bipartite graph with partition classes $V_1$ and $V_2$, where $|V_i| = n_i$.
The {\em relative size} of a subset of vertices $S\subseteq V_i$ is
\[
	\rs{S} := \frac{|S|}{n_i}.
\]
Further, for $X\subseteq V_1$ and $Y\subseteq V_{2}$ and a subgraph $F \subseteq K$, we define \emph{the relative edge count}
\[
  \eps_F(X,Y) = \eps_F(Y,X) = \frac{e_F(X,Y)}{n_1n_2}.
\]
Moreover, for $v\in V_i$ and  $Y\subseteq V_{3-i}$, we define {\em the relative degree}
\[
  d_F(v,Y) = \frac{e_F(\left\{ v \right\}, Y)}{n_{3-i}}.
\]
If  $Y = V_{3-i}$, we shorten $d_F(v,Y)$ to $d_F(v)$.

In this notation, a graph $F$ is $(\pi,\delta)$-jumbled if for every $X \subseteq V_1$, $Y \subseteq V_2$
\begin{equation}
  \label{eq_quasi_rel}
	|\eps_F(X,Y) - \pi\rs{X}\rs{Y}| \le \delta \sqrt{\rs{X}\rs{Y}}.
\end{equation}

Now, let the edges of a graph $F$ be 2-colored by blue and red, and let $B$ and $R$ be the subgraphs of $F$ induced by the edges of color, resp., blue and red. We then call $F = B \cup R$ \emph{a blue-red graph}.
Note that
\[
	\ebr(X,Y) = \eps_F(X,Y)= \er(X,Y) + \eb(X,Y).
\]
We say that a blue-red graph $F$ is \emph{$(r,b,\delta)$-regular}, if
\begin{equation}
  \label{eq_delta_regular}
	b - \delta \le \db (v) \le b + \delta, \quad\text{and}\quad r - \delta \le \dr (v) \le r + \delta \quad \text{ for every } v\in V_1\cup V_2.
\end{equation}
If $F$ is at the same time $(r+b, \delta)$-jumbled and $(r,b,\delta)$-regular, as in the technical lemma below, we will sometimes loosely refer to such a graph as \emph{regularly jumbled}.

Finally, for every $S\subseteq V_i$, $i = 1,2$,  set $\overline{S}:=V_i\setminus S$.

\begin{lemma}
	\label{lem_ananas_prep}
	Let $r, b \in (0,1)$ be real numbers and define $\alpha := \min \{r,b\}$.
	Let $\nu<\alpha/16$ and $\delta \leq \alpha/16$ be positive reals and
	let $F \subseteq K$ be a $(r, b, \delta)$-regular, $(b + r, \delta)$-jumbled bipartite blue-red graph. If sets $X \subseteq V_i$, $Y \subseteq V_{3-i}$ satisfy
\begin{equation}
  \label{eq_small_cross}
	\eb(X,\overline{Y}) + \er(\overline{X},Y) \le \nu,
\end{equation}
and
\begin{equation}
  \label{eq_less_one}
	\min \{b\x,r\y\} \leq \frac{rb}{r + b},
\end{equation}
then
\begin{equation}
	\label{eq_one_smaller}
	\max \{b\x, r\y\} \le \frac{\nu}{1 - 7\delta/\alpha}.
\end{equation}
\end{lemma}

\begin{proof}
Since
\[
  \er(\sx,\sy) = \frac{e_R(X,Y)}{n_1n_2} = \frac{\sum_{v\in X}e_R(\left\{ v \right\}, Y)}{n_1n_2} = \frac1{n_1}\sum_{v\in X}d_R(v,Y),
\]
from \eqref{eq_delta_regular} we have
\begin{equation}
  \label{eq_birblue2}
\er(\sx,\sy) + \er(\sx,\scy) \leq (r + \delta)\x
\end{equation}
and, similarly,
\begin{equation}
  \label{eq_birred2}
	\eb(\sx,\sy) + \eb(\scx,\sy) \leq (b + \delta)\y.
\end{equation}
By summing~\eqref{eq_small_cross},~\eqref{eq_birblue2} and~\eqref{eq_birred2}, we infer that
\begin{equation}
  \label{eq_ebr_upper}
  \eps_F(\sx,\sy) + \eps_F(\sx,\scy) + \eps_F(\scx,\sy)\leq \nu + (r + \delta)\x  + (b + \delta)\y.
  \end{equation}
On the other hand, by \eqref{eq_delta_regular},
\begin{equation*}
	\eps_F(\sx,\scy) = \eps_F(\sx, V_2) - \eps_F(\sx,\sy)
		\ge (b + r - 2\delta)\x - \eps_F(\sx,\sy)
\end{equation*}
and
\begin{equation*}
	\eps_F(\scx,\sy) = \eps_F(V_1, \sy) - \eps_F(\sx,\sy)
		\ge (b + r - 2\delta)\y - \eps_F(\sx,\sy). 		
\end{equation*}
Hence, by \eqref{eq_quasi_rel} with $\pi = b + r$,
\begin{multline*}
	\eps_F(\sx,\sy) + \eps_F(\sx,\scy) + \eps_F(\scx,\sy)
		\geq (b + r - 2\delta) (\x + \y) - \eps_F(\sx, \sy) \\
	\ge (b + r  - 2\delta) (\x + \y) - (b + r)\x\y - \delta \sqrt{\x\y}.
 \end{multline*}
 Comparing with \eqref{eq_ebr_upper}, we obtain the inequality
\begin{equation*}
	b \x + r \y - \frac{(b + r)b\x r\y}{br}
	\le \nu + \delta\left(3\x + 3\y + \sqrt{\x\y}\right).\\
\end{equation*}
Denoting $x := b\x$ and $y := r\y$ and $h := rb/(b + r)$, this becomes
\begin{equation}
  \label{eq_quadratic}
	x +  y - \frac{x  y}h \le \nu + \delta \left( \frac{3x}b + \frac{3y}r +\sqrt{\frac{xy}{br}} \right) =: \psi.
\end{equation}
Trivially, by the definitions of $s(X)$ and $\alpha$, and by our assumptions on $\nu$ and $\delta$, we have
 \[
	 \psi \le \nu + 7\delta < \frac12\alpha\le \frac{1}{1/b + 1/r} = h.
 \]
Since our goal --- inequality \eqref{eq_one_smaller} --- now reads as
\[
	\max \left\{ x, y \right\} \le \frac{\nu}{1 - 7\delta /\alpha},
\]to complete the proof it is enough to assume, without loss of generality, that $\max \{x, y\} = y$ and show, equivalently, that
\begin{equation}
	y \le \nu + 7\delta y/\alpha.
	\label{eq_ybound}
\end{equation}
By \eqref{eq_less_one} we have $x = \min\{x,y\} \le h$. Note that $x = h$ cannot hold, since then the \lhs{} of \eqref{eq_quadratic} would equal $h$, contradicting the fact that $\psi < h$.
Hence, we have $x < h$, which, together with $\psi < h$ and \eqref{eq_quadratic}, implies that
\begin{align*}
  y &\le \frac{\psi - x}{1 - x/h} = h - \frac{h - \psi}{1 - x/h} \le h - (h - \psi) = \psi = \nu + \delta \left( \frac{3x}b + \frac{3y}r + \sqrt{\frac{xy}{br}} \right) \le \nu + \frac{7\delta y}{\alpha},	
\end{align*}
and \eqref{eq_ybound} is proved. \end{proof}

\subsection{Alternating walks and cycles}
\label{ss_alt_walks_cycle}
A cycle in a blue-red bipartite graph is said to be \emph{alternating} if it is a union of a red matching and a blue matching, that is, every other edge is blue and the remaining edges are red.
The ultimate goal of this subsection is to show that for every edge in a regularly jumbled blue-red bipartite graph, there is an alternating cycle of bounded length containing that edge. We are going to achieve it by utilizing walks.

Given $x,y\in V_1\cup V_2$, \emph{an alternating walk from $x$ to $y$} in a blue-red graph $F$ is a sequence of (not necessarily distinct) vertices $(v_1 = x,\dots,v_s = y)$ such that for each $i = 1,\dots,s-1$, $v_iv_{i+1}\in F$, every other edge is blue and the remaining edges are red. There is no restriction on the color of the initial edge $v_1v_2$. \emph{The length} of a walk is defined as the number of edges, or $s-1$.

If the vertices $v_1,\dots,v_s$ are all distinct, an alternating walk is called \emph{an alternating path}.
Note also that if $F$ is bipartite, $x\in V_1,\;y\in V_2$, and the edge $xy$ is, say, blue, then every alternating path from $x$ to $y$ which begins (and thus ends) with a red edge together with $xy$ forms an alternating cycle containing $xy$.

The first result of this section asserts that regularly jumbled blue-red bipartite graphs  have a short alternating walk between any pair of vertices.

\begin{lemma}
  \label{lem_walks}
	Given $r, b \in (0,1)$ such that $\alpha := \min \{r,b\}$, let $\delta \in (0,\alpha/16]$ and let $F \subseteq K$ be an $(r, b, \delta)$-regular, $(b + r, \delta)$-jumbled blue-red graph. Let $L = 4\lceil 16/rb \rceil + 1$.
For any  $x\in V_i$ and $y\in V_{3-i}$  there exist at least two alternating walks from $x$ to $y$ of length at most $L$, one starting with a blue edge and another starting with a red edge.
\end{lemma}

\begin{proof}
	For  $w \in V_1 \cup V_2$ and an integer $k \ge 1$, define $R^w_k$ and $B^w_k$ as  the sets of vertices $v\in V(F)$ such that there is an alternating walk from $v$ to $w$ of length $\ell \le k$, $\ell\equiv k \pmod{2}$, starting with, respectively, a red edge and a blue edge.
(Note that these definitions concern walks {\em ending} with $w$.)

 Clearly, for every $k\ge3$,  $B^w_{k-2}\subseteq B^w_k$ and $R^w_{k-2}\subseteq R^w_k$.
 Observe also that for any $k \geq 2$, by definition the sets  $R^w_{k-1}$ and $B^w_k$ are contained in opposite sides of the bipartition $(V_1,V_2)$ and, moreover,
\begin{equation}
  \eb(\overline{B^w_k}, R^w_{k-1}) = 0.
\label{eq_2}
\end{equation}
By symmetry, $B_{k-1}^w$ and $R_k^w$ are contained in opposite sides of $(V_1,V_2)$ and
 \begin{equation}
	 \er(\overline{R^w_k}, B^w_{k-1}) = 0.
\label{eq_3}
\end{equation}

 Set $\nu = rb/16$ and note that, since $r,b < 1$, we have $\nu < \alpha/16$.
There exists an integer $t\leq T := \lceil 1/\nu \rceil = \lceil 16/rb \rceil$ such that
\begin{equation}
  \label{eq_eta}
\rs{R^w_{2t + 1}\setminus R^w_{2t - 1}}\leq \nu,
\end{equation}
since otherwise $1 \ge \rs{R^w_{2T + 1}} \ge \sum_{i = 1}^{T}  \rs{R^w_{2i+1} \setminus R^w_{2i-1}}  > \nu T \ge 1$, which is a contradiction.

By \eqref{eq_2} and \eqref{eq_eta},
 \begin{equation}
	 \eb(R^w_{2t+1}, \overline{B^w_{2t}}) = \eb(R^w_{2t - 1},\overline{B^w_{2t}})+ \eb(R^w_{2t + 1}\setminus R^w_{2t - 1}, \overline{B^w_{2t}})\leq\rs{R^w_{2t + 1}\setminus R^w_{2t - 1}}\leq \nu.
\label{eq_4}
\end{equation}
Combining \eqref{eq_3} for $k = 2t + 1$ and \eqref{eq_4}, we get
\begin{equation}
  \label{ass}
	\eb(R^w_{2t + 1}, \overline{B^w_{2t}})+ \er(\overline{R^w_{2t + 1}}, B^w_{2t}) \leq \nu.
\end{equation}

 Set $X = R^w_{2t + 1}, Y = B^w_{2t}$, for convenience.
 We claim that
 \begin{equation}
\label{eq:sXsYlower}
	 \x > r/(r + b)\quad\text{and}\quad\y > b/(r + b).
 \end{equation}
 Assuming the contrary, we have
 \[
   \min \left\{ b\x, r\y \right\} \le br/(r + b),
 \]
  which, together with \eqref{ass}, constitute the assumptions of
Lemma~\ref{lem_ananas_prep}. Applying it, we get
 \begin{equation}
   \label{max}
 \max \left\{ b\x, r\y \right\} \le \frac{\nu}{1 - 7\delta/\alpha} = \frac{rb}{16(1 - 7\delta/\alpha)} \le \frac{rb}{16(1 - 7/16)} = \frac{rb}{9},
 \end{equation}
 where the second inequality follows by our assumption $\delta \le \alpha/16$.

  On the other hand, since $X$ contains the set $R^w_1 = \Gamma_R(w)$ of red neighbors of~$w$, by $(r, b, \delta)$-regularity of $F$ we have $\x \ge r - \delta \ge \frac{15}{16}r$,  a contradiction with \eqref{max}.
  Hence, we have shown \eqref{eq:sXsYlower}. Since $X = R^w_{2t + 1}\subseteq R^w_{2T + 1}$ and $Y = B^w_{2t}\subseteq B^w_{2T}$, we also have
\begin{equation}
  \label{eq_RBx}
\rs{R^w_{2T + 1}} > \frac{r}{r + b}, \quad \rs{B^w_{2T}} > \frac{b}{r + b}.
\end{equation}
Since we chose $w$ arbitrarily, \eqref{eq_RBx} holds for $w \in \left\{ x,y \right\}$, implying

\begin{equation}
  \label{eq_RBy}
  \rs{R^y_{2T + 1}} > \frac{r}{r + b}, \quad \rs{B^x_{2T}} > \frac{b}{r + b}.
\end{equation}

Let us assume, without loss of generality, that $x \in V_1$ and $y \in V_2$. Then, $B^x_{2T}, R^y_{2T + 1}\subseteq V_1$ and, in particular, for every vertex $v\in B^x_{2T}$ there is a walk (of even length at most $2T$) from $v$ to $x$ starting with a blue and thus ending with a red edge. By \eqref{eq_RBy}, $\rs{B^x_{2T}} + \rs{R^y_{2T + 1}} > 1$, so there exists $v \in B^x_{2T} \cap R^y_{2T + 1}$.  This means that there is an alternating walk of length at most $2T + (2T + 1)= 4\lceil 16/rb \rceil + 1 = L$ from $x$ to $y$ (through $v$) that starts with a red edge.

By an analogous reasoning with the roles of the colors red and blue swapped,  $R^x_{2T}\cap B^y_{2T + 1}\neq \emptyset$, and thus there is an alternating walk  of length at most $2T + (2T + 1) = L $ from $x$ to $y$ which starts with a blue edge.
\end{proof}

The following result is an easy consequence of Lemma~\ref{lem_walks}.

\begin{lemma}
  \label{lem_alternating}
	Let $r, b \in (0,1)$, $\alpha := \min \{r,b\}$, and $\delta\in (0,\alpha/16]$. If $F \subseteq K$ is an $(r, b, \delta)$-regular, $(r + b, \delta)$-jumbled  blue-red bipartite graph, then every edge of $F$ belongs to an alternating cycle of length at most $2D$, where $D = 2\lceil 16/rb \rceil + 1$.
\end{lemma}
\begin{proof}
  Let $xy$ be an edge with $x\in V_1$, $y\in V_2$. Without loss of generality, we assume that $xy$ is blue. Then, by Lemma~\ref{lem_walks}, there exists at least one alternating walk from $x$ to $y$ of length at most $4\lceil 16/rb \rceil + 1 = 2D - 1$  starting with a red edge. Consider a shortest such walk $W$. We claim $W$ is a path. Indeed, assume that $W$ is not a path and let $w$ be the first repeated vertex on $W$. If we remove the whole segment of the walk between the first two occurrences of $w$, what remains is still an alternating walk from $x$ to $y$ starting with a red edge (since this segment has an even number of edges), contradicting the minimality of $W$.
  The path $W$ is not just a single edge $xy$ (since $xy$ is blue) and therefore $W$ and $xy$ form an alternating cycle of length at most $2D$.
\end{proof}

\subsection{Proof of Lemma~\ref{lem_ef}}
\label{ss_ef_proof}

Let $\alpha = \min \left\{ \tau p, q \right\}$. 
Applying Lemma~\ref{lem_degrees} with $\lambda = (C + 1)\log N$ (note that the condition $\lambda \le \tau_0p\hat n$ follows generously from \eqref{eq_taupn}) we have that, with probability $1 - O(N^{-C})$, for every $t < t_0$
\begin{equation}\label{eq_degreestweeked}
  \forall v_i \in V_i \quad \tau d_i (1 - \delta) \le d_i - \deg_{\R(t)}(v_i) \le \tau d_i (1 + \delta), \qquad i \in \left\{ 1,2 \right\},
\end{equation}
where
\[
  \delta = 3\sqrt{(C+1) \log N/(\tau p\hat n )}
\]
Whenever $p > 0.49$, by Lemma~\ref{lem_quasiRt}, with probability $1 - O(N^{-C})$ for every $t < t_0$ we have that
\begin{equation}
\label{eq_jumbled_G}
\begin{split}
  K \setminus \R(t) \quad &\text{is} \quad(\tau p + q, \alpha/16)\text{-jumbled}, \\
  \forall e \in K \setminus \R(t) \quad  K \setminus (\R(t)  \cup \{e\}) \quad &\text{is} \quad(\tau p + q, \alpha/16)\text{-jumbled}.
\end{split}
\end{equation}

Fix $\R(t) = G$ satisfying \eqref{eq_degreestweeked} and \eqref{eq_jumbled_G}. It remains to prove that for every pair of distinct edges $e, f \in K \setminus G$
\begin{equation}
  \label{eq_ef_G}
  \Pc{e \in \R, f \notin \R}{\R(t) = G} \ge e^{ - \lambda(t) } = N^{-2D},
\end{equation}
where $D = 3$ if $p \le 0.49$ and $D = 32/\tau p q + 3$, if $p > 0.49$.
For this, fix distinct edges $e,f\in K\setminus G$.
Aiming to apply Proposition~\ref{prop_switch},  we need to verify the assumption on the existence of alternating cycles containing a given edge.

Given a graph $G'$ and $H\in\rc_{G'}$, recall our convention to call the edges of $H\setminus G'$ blue and the edges of $K\setminus H$ red.
\begin{claim}
  \label{clm_alt}
  Let $G'\in\{ G,\;G\cup \{e\}\}$.
  For every $H \in \rc_{G'}$, every edge $g \in K \setminus G'$ is contained in an alternating cycle of length at most $2D$.
\end{claim}
From Claim~\ref{clm_alt} we complete the proof of \eqref{eq_ef_G} as follows.
Since $G$ is admissible, we have $\rc_G \ne \emptyset$. Therefore Proposition~\ref{prop_switch} implies
\begin{equation}
  \label{eq_ratio1}
  \rc_{G, e} \neq \emptyset, \quad \text{and} \quad \frac{|\rc_{G,\neg e}|}{|\rc_{G,e}|} \le N^{D} - 1.
\end{equation}
Since \eqref{eq_ratio1} implies $\rc_{G \cup \left\{ e \right\}} = \rc_{G, e} \ne \emptyset$ , Proposition~\ref{prop_switch}, applied to graph $G \cup \left\{ e \right\}$ and edge $f$, implies
\begin{equation}
  \label{eq_ratio2}
  \rc_{G \cup \left\{ e \right\}, \neg f} \neq \emptyset, \quad\text{and} \quad
  \frac{|\rc_{G \cup\{e\},f}|}{|\rc_{G \cup \{e\},\neg f}|} \le N^{D} - 1.
\end{equation}
Using \eqref{eq_ratio1} and \eqref{eq_ratio2}, we infer
\begin{align*}
  \notag	&\frac{1}{\Pc{e \in \R, f \notin \R}{\R(t) = G}} = \frac{|\rc_G|}{|\rc_{G,e,\neg f}|} = \frac{|\rc_G|}{|\rc_{G,e}|} \cdot \frac{|\rc_{G,e}|}{|\rc_{G,e,\neg f}|} \\
	&= \left( 1 + \frac{|\rc_{G,\neg e}|}{|\rc_{G,e}|} \right) \cdot \left( 1 + \frac{|\rc_{G,e,f}|}{|\rc_{G,e,\neg f}|} \right) \\
	&=  \left( 1 + \frac{|\rc_{G,\neg e}|}{|\rc_{G,e}|} \right) \cdot \left( 1 + \frac{|\rc_{G \cup\{e\},f}|}{|\rc_{G \cup \{e\},\neg f}|} \right)
	\le N^{2D},
\end{align*}
which implies \eqref{eq_ef_G}.

\smallskip

It remains to prove Claim~\ref{clm_alt}. As a preparation, we derive bounds on the vertex degrees in $G\cup\{e\}$.
Note that the inequality \eqref{eq_taupn} implies
\begin{equation}
  \label{eq_deltas}
  \delta \le 0.001,
\end{equation}
and
\begin{equation}
  \label{eq_delta_lower}
	\delta \tau d_i \ge 3 \sqrt{C\tau p\hat n  \log N } \ge 3C \log N \ge 1.
\end{equation}
The latter, together with \eqref{eq_degreestweeked}, implies that, for an arbitrary $e \in K \setminus G$,
\begin{equation}
  \label{eq_degreestweeked_cup_e}
	\forall v_i \in V_i \quad \tau d_i (1 - 2\delta) \le d_i - \deg_{G\cup \{e\}}(v_i) \le \tau d_i (1 + \delta), \qquad i \in \left\{ 1,2 \right\}
\end{equation}

We consider two cases with respect to $p$.

\medskip

 {\bf Case  $\mathbf{p \le 0.49}$}.
 We first claim that for any two vertices $x_i \in V_i, i = 1,2$, there is an alternating path $x_1y_2y_1x_2$ such that $x_1y_2$ is red (and thus $y_1y_2$ is blue and $y_1x_2$ is red).
 The number of ways to choose a blue edge $y_1y_2$ is
 \[
   M - |G'| \ge M - t - 1 = \tau pN - 1.
 \]
 We bound the bad choices of $y_1y_2$ which do not give a desired alternating path. These correspond to the walks (we must permit $y_1 = x_1$ and $y_2 = x_2$) $x_1y_2y_1$ and $x_2y_1y_2$ whose first edge in non-red, i.e., it belongs to $H$,  while the second one is blue, i.e., it belongs to $H\setminus G'$. By second inequalities in \eqref{eq_degreestweeked} and in \eqref{eq_degreestweeked_cup_e}, there are  at most $d_1 \cdot\tau(1 + \delta)d_2$ choices of such $x_1y_2y_1$ and at most $d_2 \cdot \tau(1 + \delta)d_1$ choices of such $x_2y_1y_2$, so altogether there are at most
 \[
   2p^2\tau(1 + \delta)N \le 0.98(1 + \delta) \tau p N
 \]
 bad choices of $y_1y_2$.
 Thus, noting that $1/\tau p N \le \delta\le 0.001$ (cf. \eqref{eq_delta_lower} and \eqref{eq_deltas}), the number of good choices of $y_1y_2$ is
\begin{equation*}
  \tau pN - 1 - 2p^2\tau(1 + \delta)N \ge \tau p N ( 1 - 2 p - 3\delta) > 0,
\end{equation*}
implying that there exists a desired path $x_1y_2y_1x_2$.

This immediately implies that if $g = x_1x_2$ is blue, then $g$ is contained in an alternating $4$-cycle. If $g = u_1u_2$ is red, then we choose blue neighbors $x_1 \in \Gamma_{H\setminus G'}(u_2)$ and $x_2 \in \Gamma_{H \setminus G'}(u_1)$ (which exist due to the lower bounds in \eqref{eq_degreestweeked} and \eqref{eq_degreestweeked_cup_e} being positive). Since there exists an alternating path $x_1y_2y_1x_2$ starting with a red edge, we obtain an alternating $6$-cycle containing $g$.
	This proves Claim~\ref{clm_alt} in the case $p \le 0.49$.

\medskip

{\bf Case  $\mathbf{p > 0.49}$}.
We aim to apply Lemma~\ref{lem_alternating}.
We first verify that, for $G'\in\{G, G\cup\{e\}\}$ and every $H\in\rc_{G'}$, blue-red graph $K\setminus G'$ is $(q, \tau p, \alpha/16)$-regular.
This assumption is trivial for the red graph $K\setminus H$, which is $q$-biregular regardless of $G'$.
 In view of \eqref{eq_degreestweeked} and \eqref{eq_degreestweeked_cup_e}, the relative degrees $d_{H\setminus G'}(v)$ in the blue graph lie in the interval $[\tau p - 2\delta \tau p, \tau p + 2\delta \tau p]$. Since~\eqref{eq_ef_qlbound} implies
 \[
	 \delta\tau p = 3\sqrt{\frac{(C+1)\tau p \log N }{\hat n}} \le 3\sqrt{\frac{(C+1)\log N }{ \hat n}} \le \sqrt3\left(\frac q{680}\right)^2 \le \frac q{32}
 \]
 and \eqref{eq_deltas} implies $\delta \tau p  \le \tau p / 32$, we obtain $2 \delta \tau p \le \alpha/16$. Hence, indeed, $K\setminus G'$ is $(q, \tau p, \alpha/16)$-regular.

 On the other hand, by \eqref{eq_jumbled_G} is also $(\tau p + q, \alpha/16)$-jumbled.
 Hence, by Lemma~\ref{lem_alternating} with $F = K\setminus G'$, $r = q$ and $b = \tau p$, the edge $g$ belongs to an alternating cycle of length at most $2D$ with $D = 2\lceil 16/\tau pq \rceil + 1 \le 32/\tau pq + 3$. Claim~\ref{clm_alt} is proven.
\qed

\section{Extension to non-bipartite graphs}
\label{sec_extension}

Given integers $n$ and $d$, $0\le d\le n-1$ such that $nd$ is even, define the random regular graph $\R(n,d)$ as a graph selected uniformly at random from the family $\rc(n,d)$ of all $d$-regular graphs on an $n$-vertex set $V$. To make the comparison with the binomial model $\G(n,p)$ easier, similarly as in the bipartite case, we set $p=\tfrac d{n-1}$ and define $\R(n,p):=\R(n,d)$. In what follows, we often suppress the parameter $d$  and instead just assume that $0\le p\le 1$, $p(n-1)$ is an integer, and $p(n-1)n$ is even.

As described below, our proof of Theorem~\ref{thm_embed} can be adjusted to yield its non-bipartite version and, consequently, also a non-bipartite version of Corollary \ref{cor_embed}. Instead of formulating these two quite technical results, we limit ourselves to just stating their abridged version, analogous to Theorem~\ref{thm_simple}. It confirms the Sandwich Conjecture of Kim and Vu \cite{KV04} whenever~$d \gg \left( n\log n \right)^{3/4}$ and $n - d \gg n^{3/4} (\log n)^{1/4} $.

\begin{manualtheorem}{\ref{thm_simple}$'$}
  \label{thm_simple_nonbip}
  If
\begin{equation}
  \label{eq_cond_nonbip}
p \gg \frac{\log n}{n}\quad\mbox{and}\quad
  1-p \gg \left( \frac{\log n}{n}\right)^{1/4},
	\end{equation}
	then for some $m \sim p\binom n2$, there is a joint distribution of random graphs $\G(n,m)$ and $\R(n,p)$ such that
    \begin{equation*}
	    \G(n,m) \subseteq \R(n,p) \qquad \text{a.a.s.}
    \end{equation*}
    If
    \[
      p\gg\left(\frac{\log^3 n}{n}\right)^{1/4},
    \]
    then for some $m \sim p\binom n2$, there is a joint distribution of random graphs $\G(n,m)$ and $\R(n,p)$ such that
    \begin{equation*}
	    \R(n,p) \subseteq \G(n,m)  \qquad \text{a.a.s.}
    \end{equation*}
    Moreover, in both inclusions, one can replace $\G(n,m)$ by the binomial random graph $\G(n,p')$, for some $p' \sim p$.
\end{manualtheorem}

\medskip

To obtain a proof of Theorem~\ref{thm_simple_nonbip}, one would modify the proof of Theorem~\ref{thm_embed} and its prerequisites fixing, say, $C = 1$. For the bulk of the proof (see Sections \ref{sec_crucial}--\ref{sec_quasi}) the changes are straightforward and consist mainly of replacing $K = K_{n_1, n_2}$ by $K_n$, both $n_1$ and $n_2$ by $n$, $N = n_1 n_2$  by $\binom{n}{2}$, both $d_1$ and $d_2$ by $p(n-1)$, as well as of setting $\I = 0$.
In particular, we redefine  (cf. \eqref{eq_tauzero})
\begin{equation}
  \label{eq_nonbip_tauzero}
	\tau_0 := C_1\begin{cases}
		  \frac{\log n}{p n}, \quad & p \le 0.49, \\
		  \left(\frac{\log n}{n}\right)^{1/4},
		 \quad & p > 0.49,
	\end{cases}
\end{equation}
\begin{equation*}
  \gamma_t = C_2 \begin{cases}
    \sqrt{\frac{\log n}{\tau p n}}, \quad & p \le 0.49, \\
\sqrt{\frac{\log n}{\tau^2 q^2 n}}, \quad & p > 0.49,
  \end{cases}
\end{equation*}
and
\begin{equation*}
	\gamma := C_3
	\begin{cases}
	  \sqrt{\frac{\log n}{p n}}, \quad & p \le 0.49, \\
		  \left(\frac{\log n}{n}\right)^{1/4},
		 \quad & p > 0.49,
	       \end{cases}
\end{equation*}
for some appropriately chosen constants $C_1, C_2, C_3 > 0$ and replace assumptions \eqref{eq_phat_lower}--\eqref{eq_ef_qlbound} by conditions \eqref{eq_cond_nonbip}.
Some other constants appearing in various definitions, might need to be updated, too.

The proofs of non-bipartite versions of Theorem~\ref{thm_embed}, Claim~\ref{clm_sumofgammas}, Lemmas \ref{lem_typical}, \ref{lem_thetas}, \ref{lem_degrees}, and \ref{lem_codegsRt} follow the bipartite ones in a straightforward way.

The proof of Lemma~\ref{lem_Stanislaw} is modified also in a straightforward way except for one technical change. In the switching graph $B$ we consider $6$-circuits rather than $6$-cycles (that is, we allow the vertices $x_1$ and $x_2$ in Figure~\ref{fig_coupling} coincide). With this definition the degrees in the switching graph $B$ (cf.  \eqref{eq_deg_forw}--\eqref{eq_deg_back}) are now as follows. If edges $f = u_1u_2$ and $e = v_1v_2$ are disjoint, then
\begin{equation}
  \label{eq_deg_forw_nonbip}
  \deg_B(H) = \theta_{G,H}(u_1,v_1)\theta_{G,H}(u_2,v_2) + \theta_{G,H}(u_1,v_2)\theta_{G,H}(u_2,v_1), \quad H \in \rc
\end{equation}
and
\begin{equation}
  \label{eq_deg_back_nonbip}
	\deg_B(H') = \theta_{G,H}(v_1,u_1)\theta_{G,H}(v_2,u_2) + \theta_{G,H}(v_1,u_2)\theta_{G,H}(v_2,u_1), \quad H' \in \rc'.
\end{equation}
If $e$ and $f$ share a vertex (without loss of generality, $u_1 = v_1$), then
\begin{equation*}
  \deg_B(H) = \theta_{G,H}(u_2,v_2), \quad H \in \rc, \quad \text{and} \quad
	\deg_B(H') = \theta_{G,H'}(v_2,u_2), \quad H' \in \rc'.
\end{equation*}
We modify the definition \eqref{eq_typical} of a $\delta$-typical graph by taking the maximum over all pairs $(u, v) \in f \times e$ of distinct vertices. Since we now have two terms in \eqref{eq_deg_forw_nonbip} and \eqref{eq_deg_back_nonbip}, the bounds in \eqref{eq_top_bound} and \eqref{eq_bottom_bound} get some extra factors $2$, which cancel out, leading to a bound similar to \eqref{eq_p_ratio}, but with constant $9$ inflated.

\smallskip

There is also a little inconvenience related to the analog of formula \eqref{eq_coco}. Namely,  now
\begin{multline}
  \label{coco_nb}
	  \cod_{F}(u, v) = |\Gamma_{F}(u) \setminus \left\{ v \right\}| + |\Gamma_{F}(v) \setminus \left\{ u \right\}| - |\Gamma_{F}(u) \cup \Gamma_{F}(v) \setminus \left\{ u, v \right\}| \\
	= \deg_F(u)+\deg_F(v) - 2\I_{uv \in F} - (n - 2 - \cod_{K_n \setminus F}(u,v)),
	\end{multline}
	so, the formula gets an extra factor $O(1)$, which turns out to be negligible.

	Set $\R:=\R(n,p)$. More substantial modifications needed to prove Theorem~\ref{thm_simple_nonbip} (which we discuss in detail below) are the following.
\begin{itemize}
  \item Lemma~\ref{lem_codegs} (only the case $\I=0$ left). Instead of using the asymptotic enumeration formula of Canfield, Greenhill and McKay (Theorem~\ref{thm_enum}), we apply the one of Liebenau and Wormald~\cite{LW2017}.
  \item Lemma~\ref{lem_ef} in the case $p > 0.49$. Instead of directly showing the existence of short alternating cycles in $K_n\setminus \R(t)=(K_n\setminus\R)\cup(\R\setminus \R(t))$, we create a blue-red auxiliary bipartite graph from $K_n \setminus \R(t)$ and apply unchanged Lemma~\ref{lem_alternating} to it.
\end{itemize}

\subsection{Sketch of co-degree concentration for regular graphs}
The non-bipartite version of Lemma~\ref{lem_codegs} below is obtained by just setting $\I=0$ and replacing $\hat n$ by $n$ (the resulting term $\hat p$ is swallowed by $\lambda$).

\medskip

\begin{manuallemma}{\ref{lem_codegs}$'$}
  \label{lem_codegs_nonbip}
  Suppose that $\hat p n  \to \infty$  and  $\lambda = \lambda(n)\to \infty$.
  Then, for any distinct $u, v \in [n]$,
  \begin{equation}
    \label{eq_codegconc_nonbip}
    \prob{
    |\cod_{\R}(u,v) - p^2 n| \le 20\hat p\sqrt{ \lambda n}  + \lambda} = O\left(  n e^{ - \lambda } \right).
  \end{equation}
\end{manuallemma}

\medskip

\noindent As before, it is sufficient to assume that $p \le 1/2$ and thus replace $\hat p$ by $p$ in \eqref{eq_codegconc_nonbip}. Indeed, by \eqref{coco_nb} and the identity $2p-1=p^2-q^2$,
\begin{equation*}
  \begin{split}
    \cod_{\R}(u, v)-p^2n = 2(p(n-1) - \I_{uv \in \R}) - (n - 2 - \cod_{K_n \setminus \R}(u,v))-p^2n \\
    = \cod_{K_n \setminus \R}(u_1,v_1) - q^2n + O(1).
  \end{split}
	\end{equation*}
This allows, for $p>1/2$, to replace $p$ by $q$, as explained in the proof of Lemma~\ref{lem_codegs} (the error $O(1)$ is absorbed by the term $\lambda$).

The proof of \eqref{eq_codegconc_nonbip} for $p \le 1/2$ is based on the following enumeration  result by Liebenau and Wormald~\cite{LW2017} (see Cor.~1.5 and Conj.~1.2 therein), proved for some ranges of $d$ already  by McKay and Wormald~\cite{MW90,MW91}.

 Given a sequence $\mathbf{d}=(d_1,\dots,d_n)$, let $g(\mathbf{d})$ denote the number of graphs $G$ on the vertex set $V=(v_1,\dots,v_n)$ whose degree sequence is $\mathbf{d}$, that is, $\deg_G(v_i)=d_i$, $i=1,\dots,n$. Further, let
\[
  \bar d=\frac1n\sum_{i=1}^nd_i\;,\quad\mu=\frac{\bar d}{n-1}\;,\quad \gamma_2=\frac1{(n-1)^2}\sum_{i=1}^n(d_i-\bar d)^2\;,\quad\hat d=\min(\bar d,n-1-\bar d).
\]

\begin{theorem}[\cite{LW2017}]
  \label{thm_enum_nonbip}
	For some absolute constant $\eps > 0$, if $\mathbf{d}=\mathbf{d}(n)$ satisfies $\max_i|d_i-\bar d|=o(n^\eps\hat d)$, $n\hat d\to\infty$, as $n\to\infty$, and $\sum_{i=1}^nd_i$ is even, then
\[
  g(\mathbf{d})\sim\sqrt2\exp\left(-\frac14-\frac{\gamma_2^2}{4\mu^2(1-\mu)^2}\right)\left(\mu^\mu(1-\mu)^{1-\mu}\right)^{\binom n2}\prod_{i=1}^n\binom{n-1}{d_i}.
\]
\end{theorem}

The proof of Lemma~\ref{lem_codegs_nonbip} follows the lines of the proof of Lemma~\ref{lem_codegs} in the case $\I=0$. For fixed distinct $u,v\in V$ we define, as before, $\rc_k = \{ G \in \rc(n,d) : \cod_G(u,v) = k \}$, $0\le k\le d$. Using Theorem~\ref{thm_enum_nonbip}, it is tedious but straightforward to find a sequence $r_k:=r_k(n,d)$ such that $|\rc_k|\sim r_k$. We have $r_k = r_k^0 + r_k^1$, as the formula for $|\rc_k|$ breaks into two, according to whether $uv$ is an edge ($r_k^1)$ or not ($r_k^0$). It can be checked that $r_0, \dots, r_d > 0$ and, for $1\le k\le d$,
\[
  \frac{r_k^0}{r_{k-1}^0}=\frac{(d-k+1)^2}{k(n-2-2d+k)}\left(1-\frac1d\right)\left(1-\frac1{n-1-d}\right)
\]
 as well as
 \[
   \frac{r_k^1}{r_{k-1}^1}=\frac{(d-k)^2}{k(n-2d+k)}\left(1-\frac1d\right)\left(1-\frac1{n-1-d}\right).
 \]
Let
$\mu^0 := (d+1)^2/n$ and $\mu^1 := d^2/n$. For $\mu^0\le k\le d$ and, respectively, for $\mu^1\le k\le d$,
\[
  \frac{r_k^0}{r_{k-1}^0}\le\frac{\mu^0}k\quad\mbox{and}\quad\frac{r_k^1}{r_{k-1}^1}\le\frac{\mu^1}k\le\frac{\mu^0}k.
\]
Thus, for $\mu^0\le k\le d$,
\[
\frac{r_k}{r_{k-1}}=\frac{r_k^0+r_k^1}{r_{k-1}^0+r_{k-1}^1}\le\frac{\mu^0}k.
\]
Similarly, setting $\rho=\left(1-\frac1d\right)\left(1-\frac1{n-1-d}\right)$, for $1\le k\le\mu^0$ and, respectively, $1\le k\le\mu^1$,
\[
  \frac{r_{k-1}^0}{r_{k}^0}\le\frac k{\mu^0\rho}\le\frac k{\mu^1\rho}\quad\mbox{and}\quad\frac{r_{k-1}^1}{r_{k}^1}\le\frac k{\mu^1\rho}.
\]
So, for $1\le k\le\mu^1$,
\[
\frac{r_{k-1}}{r_k}=\frac{r_{k-1}^0+r_{k-1}^1}{r_k^0+r_k^1}\le\frac k{\mu^1\rho}.
\]
Setting conveniently $\mu_+:=\mu^0$ and $\mu_-:=\mu^1\rho$, we may now apply Claim~\ref{clm_Poiss_bounds}, which extends straighforwardly to the non-bipartite setting.

We have, using that $d \ge 1$,
\begin{equation}
\label{eq_mu_plus_nonb}
   \mu_+ \le \frac{4d^2}{n} \le 4p^2n.
\end{equation}
Also, using $p \le 1/2$,
\begin{equation*}
  \mu_+\le \frac{(pn+1)^2}n\le p^2n+2 \quad \text{ and } \quad \mu_-=\mu^1\rho\ge\frac{d^2}n\left(1-\frac1d\right)^2\ge\frac{(pn-2)^2}n\ge p^2n-2,
\end{equation*}
whence, setting $X = \cod_{\R}(u_1,v_1)$, and $x = 20\sqrt{p^2 n \lambda} + \lambda - 2$,

\begin{equation}
\label{eq_twotailbound_nonbip}
    \prob{|X - p^2n| \ge 20p\sqrt{\lambda n} + \lambda} \le \prob{X \ge \mu_+ + x} + \prob{X\le \mu_- - x}.
  \end{equation}
We now bound the \rhs{} of \eqref{eq_twotailbound_nonbip} using Claim~\ref{clm_Poiss_bounds} with $\lambda - 2$ instead of $\lambda$. Noting that \eqref{eq_mu_plus_nonb} implies $x \ge \sqrt{2 \mu_+ (\lambda - 2)} + (\lambda - 2)$, we have
\begin{equation*}
     \prob{X \ge \mu_+ + x} + \prob{X\le \mu_- - x} = O\left( \sqrt{N} e^{-(\lambda - 2)} \right).
\end{equation*}
Since $\sqrt{N} e^{-(\lambda -2)} = \Theta(n e^{-\lambda})$,
this completes the proof of Lemma~\ref{lem_codegs_nonbip}.

\subsection{Sketch of the proof of new Lemma~\ref{lem_ef}, case \texorpdfstring{$p > 0.49$}{p > 0.49}}
For completeness, we state here the non-bipartite counterpart of Lemma~\ref{lem_ef} which on the surface looks almost identical.

\begin{manuallemma}{\ref{lem_ef}$'$}
  \label{lem_ef_nonbip}
  Assuming \eqref{eq_cond_nonbip}, we have, a.a.s.,
  \begin{equation*}
    \min_{e, f \in K_n \setminus \R(t), e \neq f}\Pc{e \in \R, f \notin \R}{\R(t)} \ge e^{ - \lambda(t) }, \qquad \text{for } t \le t_0.
  \end{equation*}
\end{manuallemma}

Looking at the diagram in Figure~\ref{fig_flowchart}, we see that the proof of Lemma~\ref{lem_ef} relies on several other results, most notably Lemmas~\ref{lem_quasiRt} and~\ref{lem_alternating}.
It would be, however, a very tedious task to come up with non-bipartite counterpart of Lemma~\ref{lem_alternating}, and, consequently, ones of Lemmas~\ref{lem_walks} and~\ref{lem_ananas_prep}. Instead, we rather convert the non-bipartite case into the bipartite one by a standard probabilistic construction and use Lemmas~\ref{lem_quasiRt} and~\ref{lem_alternating} practically unchanged.

Given a blue-red graph $H \subseteq K_n$, let $\bip(H)$ be a random blue-red bipartite graph with bipartition $V_1 = \left\{ u_1, \dots, u_n \right\}, V_2 = \left\{ v_1, \dots, v_n \right\}$, such that for each edge $ij \in E(H)$ we flip a fair coin and include into $\bip(H)$ either $u_iv_j$ or $u_jv_i$ (colored the same color as $ij$), with the flips being independent. In particular, $|E(\bip(H))| = |E(H)|$, so the density of $\bip(H)$ is exactly half of that of $H$, while, if the densities of the blue and red subgraphs of $H$ are $b$ and $r$, then the \emph{expected} densities of the blue and red subgraphs of $\bip(H)$ are $b/2$ and $r/2$.

Note that if there is an instance of $\bip(H)$ in which an edge $u_iv_j$ is contained in an alternating cycle of length at most $D$, then $ij$ is contained in an alternating \emph{circuit} of the same length. It is not, in general, a cycle, since some vertices can be repeated, but edges are not, as the edges of $\bip(H)$ correspond to different edges of $H$. Luckily, Proposition~\ref{prop_switch} actually works  for alternating circuits too, since in a circuit the blue and red degrees of each vertex equal each other (see the paragraph following equation \eqref{eq_rcG}) and, similarly as for cycles, in $K_n$ there are at most $n^{2\ell-2}$ circuits of length $\ell$ containing a given edge $e$. Defining $\rc_G$, $\rc_{G,e}$, and $ \rc_{G,\neg e}$ analogously to the bipartite case (cf. \eqref{eq_rcG} and \eqref{eq_rcGe}), and making obvious modifications of the proof of Proposition~\ref{prop_switch}, we obtain the following.
\begin{manualprop}{\ref{prop_switch}$'$}
  \label{prop4_nonbip}
  Let a graph $G\subseteq K_n$ be such that $\rc_G\neq\emptyset$ and let $e\in K_n\setminus G$. Assume that for some number $D > 0$ and every $H\in \rc_{G}$ the edge $e$ is contained in an alternating circuit of length at most $2D$. Then $\rc_{G,\neg e}\neq\emptyset$, $\rc_{G, e}\neq\emptyset$, and
	\[
		\frac{1}{n^{2D} - 1} \le \frac{|\rc_{G,\neg e}|}{|\rc_{G,e}|} \le n^{2D} - 1.
	\]
\end{manualprop}

The plan to adapt the proof of Lemma~\ref{lem_ef} is to condition on $K_n\setminus\R(t)$ having its degrees and co-degrees concentrated (using the non-bipartite counterparts of Lemmas~\ref{lem_degrees} and~\ref{lem_codegsRt}) and then show that there is an instance $F$ of $\bip(K_n \setminus \R(t))$ in which the degrees and co-degrees are similarly concentrated, with just a negligibly larger error.

In particular, such an $F$ is $(q/2,\tau p/2,\alpha/32)$-regular (as before, we denote $\alpha = \min \left\{ \tau p, q \right\}$). Then, applying Theorem~\ref{thm_Thomason} along the lines of the proof of Lemma~\ref{lem_quasiRt}, we show that $F$ is also $(\pi/2, \alpha/32)$-jumbled, where, as before, $\pi = \tau p + q$. Hence, we are in position to apply Lemma~\ref{lem_alternating}, obtaining for every edge of $F$ an alternating cycle of length $O\left( 1/(\tau p q) \right)$ in $F$. As explained above, this implies alternating circuits in $K_n \setminus \R(t)$ of the same length, and so we may complete the proof of Lemma~\ref{lem_ef_nonbip}, based on Proposition~\ref{prop4_nonbip}, in the way we did it in the bipartite case.

Let us now present some more details.
We condition on $\R(t) = G$ such that (cf. \eqref{eq_degupper} and~\eqref{eq_codegsRt})
\begin{equation}
\label{eq_blue_degs}
  \forall v \in [n] \quad |\deg_{G}(v) - (1 - \tau) pn| = O(\sqrt{n \log n})
\end{equation}
and
\begin{equation}
  \label{codegs}
  \max_{u, v \in [n], u \ne v} \cod_{G}(u,v) = (1 - \tau)^2p^2 n + O(\sqrt{n \log n}).
\end{equation}

Fix an arbitrary $H \in \rc_G$. Since the red graph $K_n \setminus H$ is $q(n-1)$-regular, we have $\deg_{\bip(K_n \setminus H)} (v) \sim \Bi(q(n-1), 1/2)$. Thus, by a routine application of the Chernoff bound  a.a.s.
\begin{equation}
\label{eq_red_deg}
\max_{v \in [n]} \left| \deg_{\bip(K_n \setminus H)}(v) - qn/2 \right| = O(\sqrt{n \log n}).
\end{equation}
By \eqref{eq_blue_degs}, the blue graph $H \setminus G$ has degrees $pn - (1 - \tau)pn + O\left( \delta \tau p n  \right) = \tau p n + O\left( \sqrt{n \log n} \right)$, so the Chernoff bound implies that (using $p > 0.49$) a.a.s.
\begin{equation}
\label{eq_blue_deg}
\max_{v \in [n]} \left|\deg_{\bip(H \setminus G)}(v) - \tau pn/2 \right| =O\left( \sqrt{n \log n} \right).
\end{equation}
 By \eqref{eq_cond_nonbip} and \eqref{eq_nonbip_tauzero}, since $p>0.49$ and $\tau\ge \tau_0$,
 \[
   \alpha=\min\{\tau p,q\} \gg \left(\frac{\log n}n\right)^{1/4}.
 \]
 Consequently, $\sqrt{n\log n} \ll \alpha n$ and, with a big margin, we conclude that a.a.s. $\bip(K_n \setminus G)$ is $(q/2, \tau p/2, \alpha/32)$-regular.

 Next, we check that $\bip(K_n \setminus G)$ is $(\pi/2, \alpha/32)$-jumbled.
Inequalities \eqref{eq_red_deg} and \eqref{eq_blue_deg} imply
\begin{equation}
\label{eq_deg_nonbip}
  \deg_{\bip(K_n\setminus G)}(v)
  = \frac{\pi n}{2}  + O(\sqrt{n \log n}).
\end{equation}
Further, by \eqref{coco_nb}, \eqref{eq_blue_degs}, and \eqref{codegs}, we have that (cf. \eqref{eq_cod_compl})
	\begin{align*}
	  \max_{u \ne v}\cod_{K_n \setminus G}(u, v) &\le 2\max_{v \in [n]} \deg_{K_n \setminus G}(v) - (n - 2 - \cod_{G}(u,v))\\
&=(2 \pi - 1)n  + (1 - \pi)^2 n + O(\sqrt{n\log n}) = \pi^2 n + O\left( \sqrt{n\log n} \right).
	\end{align*}
Since for every $u, v \in [n], u \ne v$,
$  \cod_{\bip(K_n \setminus G)}(u,v) \sim \Bi( \cod_{K_n \setminus G}(u,v), \tfrac{1}{4})$,
	by a simple application of Chernoff's inequality and the union bound we show that a.a.s.
\begin{equation}
  \label{eq_complementcodegs_nonbip}
  \max_{u,v \in V_1} \cod_{\bip \left( K_n \setminus G \right)} (u,v) = \frac{\pi^2n}{4}  + O( \sqrt{n\log n} ).
\end{equation}
Now, fix an instance $F$ of the graph $\bip(K_n \setminus G)$ for which \eqref{eq_red_deg}, \eqref{eq_blue_deg} and \eqref{eq_complementcodegs_nonbip} hold.

Applying Theorem~\ref{thm_Thomason}, by calculations similar to those in the proof of Lemma~\ref{lem_quasiRt}, from \eqref{eq_deg_nonbip} and \eqref{eq_complementcodegs_nonbip} we deduce that $F$ is $( \pi/2, \alpha/32)$-jumbled.
Since we earlier showed that $F$ is $(q/2, \tau p/2, \alpha/32)$-regular, Lemma~\ref{lem_alternating} implies that in $F$  every edge is contained in an alternating cycle of length $O\left( 1/(\tau p q) \right)$ and therefore in $K_n \setminus G$ every edge is contained in an alternating circuit of the same length. The same argument implies alternating cycles in $K_n \setminus (G \cup \{e\})$ for an arbitrary edge $e$, since only the lower bound in \eqref{eq_blue_degs} has to be decreased by a negligible quantity $1$.

The rest of the proof of Lemma~\ref{lem_ef_nonbip} in the case $p > 0.49$ follows along similar lines.

	\section{Technical facts}
	\label{sec_technical}
	Here we collect a few technical or very plausible facts with their easy proofs. Most of them have been already utilized in the paper. An exception is Proposition~\ref{prop_maxdeg} to be used only in Remark \ref{rem_best_gamma}, Section  \ref{sec_concluding}.

\medskip

We begin with convenient consequences of the assumptions of Lemma~\ref{lem_Stanislaw}.
\begin{proposition}\label{tech}
  For $t = 0, \dots,  t_0 - 1$,  the conditions of Lemma~\ref{lem_Stanislaw}, namely,~\eqref{eq_phat_lower}, \eqref{eq_hatpI_ubound}, and \eqref{eq_ef_qlbound}, imply that
  \begin{equation}
    \label{eq_tauq_lower}
	  \tau q \ge \tau_0 q \ge 700 \cdot 680 \sqrt{\frac{3(C+4)\log N}{\hat n}}, \quad \text{whenever } p > 0.49,
  \end{equation}
  \begin{equation}
    \label{eq_gammat_less1}
    \gamma_t \le \gamma_{t_0}\le 1,
  \end{equation}
and
\begin{equation}
  \label{eq_taupn}
	\tau \hat p \hat n \ge \tau_0 \hat p\hat n \ge 3000^2(C + 4)\log N.
\end{equation}
\end{proposition}
\begin{proof}
  Since $\tau = \tau(t) \ge \tau(t_0) \ge \tau_0$ and $\gamma_t$ is increasing in $t$, the first inequalities in \eqref{eq_tauq_lower}--\eqref{eq_taupn} are immediate and it is enough to prove the second inequalities.

	Inequality \eqref{eq_tauq_lower} follows from the definition \eqref{eq_tauzero} of $\tau_0$ and \eqref{eq_ef_qlbound}.

  To show \eqref{eq_gammat_less1}, for $p\le0.49$, using the definition of $\tau_0$ (see \eqref{eq_tauzero}) and \eqref{eq_hatpI_ubound}, we get
  \[
	  \gamma_{t_0} \le \frac{1080}{340^4} + \sqrt\frac23 < 0.01 + 0.82 < 1,
  \]
  while for $p > 0.49$
  \[
	  \gamma_{t_0} \le 0.01 + \frac{25000}{680 \cdot 700} < 1.
  \]

  To see \eqref{eq_taupn} first note that for $p \le 0.49$ it is straightforward from the definition of $\tau_0$ in \eqref{eq_tauzero}. For $p > 0.49$,
  observing that \eqref{eq_phat_lower} implies $\sqrt{\tfrac{3(C+4)\log N}{\hat{n}}} \le \sqrt{\hat p /3240} \le 1/3240$, we argue that
\[
	\tau_0 \hat p\ge \frac{49}{51}\tau_0 q \By{\eqref{eq_tauq_lower}}{\ge} \frac{49}{51}\cdot 700\cdot 680\sqrt{\frac{3(C+4)\log N}{\hat{n}}} \ge 3000^2\frac{(C+4)\log N}{\hat{n}} ,
\]
whence \eqref{eq_taupn} follows.
\end{proof}
\medskip

Next, we give a proof of Claim~\ref{clm_sumofgammas} which was instrumental in deducing Theorem~\ref{thm_embed} from Lemma~\ref{lem_Stanislaw} in Section~\ref{ss_proof_embed}.

\begin{proof}[Proof of Claim~\ref{clm_sumofgammas}] Writing $X = t_0 - |S|$, we have $\E X = \sum_{t=0}^{t_0 - 1} \gamma_t$. Denoting $\alpha := 1080\hat p^2 \I$ and
	 \begin{equation*}
		\beta :=
	\begin{cases}
	  3240\sqrt{\frac{2(C + 3)\log N}{p\hat n}}, \quad & p \le 0.49, \\
		25000 \sqrt{\frac{(C + 3)\log N}{q^2\hat n }  }
	, \quad & p > 0.49,
	\end{cases}
	 \end{equation*}
	 we have
	 \begin{equation*}
		 \gamma_t = \alpha +
	\begin{cases}
	  \beta \tau^{-1/2}, \quad & p \le 0.49, \\
		\beta \tau^{-1}, \quad & p > 0.49.
	\end{cases}
	 \end{equation*}
	 Since, trivially, $\sum_{t = 0}^{t_0 - 1} \alpha = \alpha t_0 \le \alpha M = 1080\hat p^2 \I \cdot M$, to prove \eqref{eq_theta_p}, it suffices to show that
	 \begin{equation}
\label{eq_summy}
		 \sum_{t = 0}^{t_0 - 1} \tau^{-1/2} \le 2M, \quad \text{and} \quad
		 \sum_{t = 0}^{t_0 - 1} \tau^{-1} \le \frac{M}{4}\log \frac{\hat n}{\log N}.
	 \end{equation}
	 Since $\tau = \tau(t) = 1 - t/M$ is positive and decreasing on $[0,M)$,
\begin{equation*}
	\sum_{t = 0}^{t_0 - 1} \tau^{-1/2} \le \sum_{t = 0}^{M - 1} \tau^{-1/2} \le \int_0^M \tau^{-1/2} dt = M \int_0^{1} \tau^{-1/2} d\tau = 2M,
\end{equation*}
implying the first inequality in \eqref{eq_summy}. On the other hand, recalling $t_0 = \lfloor (1-\tau_0)M \rfloor$,
\begin{equation*}
  \sum_{t = 0}^{t_0 - 1} \tau^{-1} \le \int_0^{t_0} \tau^{-1} dt \le \int_0^{(1-\tau_0)M} \tau^{-1} dt  \le M \int_{\tau_0}^{1} \tau^{-1} d\tau = M \log \frac{1}{\tau_0} \le \frac{M}{4} \log \frac{\hat n}{\log N},
\end{equation*}
which implies the second inequality in \eqref{eq_summy}.

Checking \eqref{eq_thetaM} is a dull inspection of definitions using conditions \eqref{eq_ef_qlbound}--\eqref{eq_tau_less}, which the readers might prefer to do themselves. We nevertheless spell out the details starting with the case $p\le 0.49$. Choosing $C^*$ large enough,
\begin{align*}
	\gamma &= C^*\left( p^2 \I +  \sqrt{\frac{\log N}{p\hat n }}\right) \\
	&\ge  1080 p^2 \I +  6480 \sqrt{\frac{(C + 3)\log N}{p\hat n }}
	+ 3240\sqrt{\frac{3(C + 4)\log N}{p\hat n}} + \sqrt{\frac{\log N}{p\hat n}} + \sqrt{\frac{2C\log N}{p\hat n}}\\
	&\ge \theta  + \sqrt{\tau_0} + 2\sqrt{1/M} + \sqrt{\frac{2C\log N}{M}} \By{\eqref{eq_tau_less}}{\ge}  \theta + \tau_0 + 2/M + \sqrt{\frac{2C\log N}{M}}.
\end{align*}
In the case $p > 0.49$, note that  $q^{3/2}\ge q^2\ge \hat p^2$.
Therefore, assuming $C^*$ is large enough,
\[
  \gamma/3 \ge 700\left(3(C+4)\right)^{1/4} \left( q^{3/2}\I + \left(\frac{\log N}{\hat n}\right)^{1/4}\right) = \tau_0,
\]
and
\[
\gamma/3 \ge  1080q^{3/2}\I +  6250\sqrt{\frac{(C + 3)\log N}{q^2\hat n}}\log\frac{\hat n}{\log N} \ge  \theta.
\]
Finally, because \eqref{eq_ef_qlbound} implies $\hat n / (\log N) \ge 1$ and $d_1, d_2 \ge 1$ implies $M \ge \hat n$, for large enough~$C^*$,
\begin{align*}
  \gamma/3 \ge \frac{C^*}{3}\left( \frac{\log N}{\hat n} \right)^{1/4} &\ge \frac{\log N}{\hat n} +  \sqrt{\frac{2C\log N}{\hat n}} \\
  \justify{$p > 0.49$, $M \ge \hat n$} & \ge \frac{2}{M} + \sqrt{\frac{2C\log N}{M}},
\end{align*}
which, with the previous two inequalities, implies \eqref{eq_thetaM}.
\end{proof}

We conclude this technical section with a lower bound on the maximum degree in a binomial random graph.
\begin{proposition}
  \label{prop_maxdeg}
  If $p' \le 1/4$ and $n_2 \le n_1$, then a.a.s.\ the maximum degree of $\G(n_1,n_2,p')$ in $V_1$ is at least $\kappa := \min \{ \ceil{p'n_2 + \sqrt{p'(1-p')n_2\log n_1}}, n_2 \}$.
\end{proposition}
\begin{proof}
  Writing $Z \sim \mathcal{N}(n_2p', n_2p'(1- p'))$, by Slud's inequality~\cite[Theorem 2.1]{S77} we have
	\begin{equation*}
	  r := \prob{\Bi(n_2,p') \ge \kappa} \ge \prob{Z \ge \kappa} \\
	   \ge \prob{Z \ge \sqrt{p'(1-p')n_2\log n_1}},
	\end{equation*}
	hence by a standard approximation of the normal tail we obtain
	\begin{equation*}
		r \ge \frac{1 + o(1)}{\sqrt{2\pi \log n_1}}e^{-\frac{1}{2}\log n_1} =
		e^{ - (1/2 + o(1)) \log n_1 }.
	\end{equation*}
	Therefore, the probability that all vertices in $V_1$ have degrees smaller than $\kappa$ is
\begin{equation*}
  (1-r)^{n_1} \le e^{ - n_1 r} \le \exp \left( -n_1e^{-(1/2 + o(1))\log n_1} \right) \to 0,
\end{equation*}
proving the proposition.
\end{proof}

\section{Application: perfect matchings between subsets of vertices}
\label{sec_PP}

Perarnau and Petridis in~\cite{PePe}, in connection with a problem of Pl\"unnecke, studied the existence of perfect matchings between fixed subsets of vertices in random biregular graphs. In particular, they proved the following result (which we state in our notation to make it easier to apply Theorem~\ref{thm_embed}).

\begin{theorem}[Theorem~2 in~\cite{PePe}]
  \label{thm_PP}
  Let  $k>0$ be a constant, and assume $n_2 = kn_1$ and $pn_2 \le n_1$. Take subsets $A\subseteq V_1$ and $B\subseteq V_2$ of size $pn_2$ and let $\mathcal M_{A,B}$ denote the event that the subgraph of $\R(n_1, n_2, p)$ induced by $A$ and $B$ contains a perfect matching. As $\hat n \to \infty$, we have
  \begin{equation*}
    \prob{\mathcal M_{A,B}} \to
    0, \quad \text{if} \quad p^2n_2 - \log pn_2 \to -\infty \text{ or } pn_2 \text{ is constant},
  \end{equation*}
  and
   \begin{equation}
     \label{eq_PP}
	   \prob{\mathcal M_{A,B}} \to
     1, \quad \text{if} \quad p^2n_2 - \log pn_2 \to \infty.
   \end{equation}
\end{theorem}

Perarnau and Petridis~\cite{PePe} speculated that if the bipartite version of the Sandwich Conjecture was true, then Theorem~\ref{thm_PP} would follow straightforwardly from the classical result of Erd\H os and R\'enyi on perfect matchings in the random bipartite graph (see Theorem 4.1 in~\cite{JLR}). That result, in particular, implies that the bipartite binomial random graph $\G(n',n',p')$ contains a perfect matching a.a.s.\ whenever $p'n'-\log n'\to\infty$ as $n'\to\infty$. We show how this, together with Theorem~\ref{thm_embed}, implies the $1$-statement in \eqref{eq_PP}, provided that condition \eqref{eq:thm_embed_q} of Theorem~\ref{thm_embed} is satisfied, which in particular implies
\begin{equation}
  \label{eq:q_lower}
  q \ge \left( \frac{\log N}{\hat n} \right)^{1/4}.
\end{equation}
By Theorem~\ref{thm_embed}, the random graph $\Rnnp$ a.a.s.\ contains a random graph $\G(n_1,n_2,p')$ with
 \[
   p' =(1 - 2\gamma) p,
 \]
 where $\gamma$ is defined in \eqref{eq_gamma}. In particular, the subgraph of $\Rnnp$ induced by $A$ and~$B$ contains a random graph $\G(pn_2, pn_2, p')$.
 To see that the latter random graph contains a perfect matching a.a.s., let us verify the Erd\H{o}s--R\'enyi condition
 \begin{equation*}
 p'pn_2 - \log pn_2 \to \infty.
 \end{equation*}
 From the assumption in \eqref{eq_PP} it follows that
 \begin{equation}
   \label{eq:p_sqrt}
   p\geq n_2^{-1/2} = \Theta\left( \hat n^{-1/2} \right).
 \end{equation}
 For our purposes, it is sufficient to check that
 \begin{equation}
   \label{eq:gamma_simple}
   \gamma \ll (\log N)^{-1},
 \end{equation}
 since then
\begin{equation*}
  \gamma \log pn_2 \le \gamma \log N \to 0, \quad \text{and} \quad \gamma \to 0,
\end{equation*}
which together with the condition in \eqref{eq_PP} imply
\[
  p'pn_2 - \log pn_2 = (1 - 2\gamma)(p^2n_2 - \log pn_2) - 2\gamma \log pn_2 \to \infty.
\]

To see that \eqref{eq:gamma_simple} holds, first note that $\I = 1$ implies $p = O( (\log N)^{-1})$ for $p \le 0.49$ and $q = O( (\log N)^{-1})$ for $p > 0.49$; hence, regardless of $p$, the first term in the definition of $\gamma$ is $o\left( (\log N)^{-1} \right)$. The remaining terms are much smaller: for $p \le 0.49$ inequality \eqref{eq:p_sqrt} implies that the second term in the definition of $\gamma$ is $O\left( (\log N)^{1/2}\hat n^{-1/4} \right) \ll (\log N)^{-1}$, while in the case $p > 0.49$  assumption~\eqref{eq:q_lower} implies that the last two terms in the definition of $\gamma$ are at most  $(\log N)^{O(1)}\hat n^{-1/4} \ll (\log N)^{-1}$.

\section{Concluding remarks}
\label{sec_concluding}
\begin{remark}
  \label{rem_best_gamma}
	Assume $p \le 1/4$. If a.a.s.\ $\G(n_1, n_2, p') \subseteq \Rnnp$, then we must have
	\begin{equation}
	  \label{eq_p_best}
		p' = p \left(1 - \Omega \left( \min \left\{ \sqrt{(\log N)/(p \hat n)}, 1  \right\}\right)\right).
	\end{equation}
	To see this, assume, without loss of generality, that $n_2 \le n_1$, and note that by Proposition~\ref{prop_maxdeg} we must have $p'n_2 + \sqrt{p'(1 - p')n_2\log n_1} \le pn_2$, and therefore $p/p' \ge 1 + \sqrt{\frac{(1 - p')\log n_1}{p'n_2}} $. Since $p' \le p \le 1/4$ and $n_2 \le n_1$, we have $p/p' = 1 + \Omega\left( \sqrt{\tfrac{\log N}{p\hat n }} \right)$, whence \eqref{eq_p_best} follows.
	\end{remark}

\begin{remark}
  \label{rem_optimality}
  In view of Remark~\ref{rem_best_gamma}, the error $\gamma$ in Theorem~\ref{thm_embed} has optimal order of magnitude, whenever $p \le 1/4$, provided that also $p^2 \I = O\left( \sqrt{\log N / p \hat n} \right)$ (that is, if either $p = O\left( \left( (\log N)/\hat n \right)^{1/5}\right)$ or $\I = 0$).

  From Remark~\ref{rem_best_gamma} it also follows that we cannot have $\gamma = o(1)$ for $\log N = \Omega(p\hat n )$.  Theorem~\ref{thm_embed} does not apply to the case $\log N \ge p\hat n /(C^*)^2$, but we think it would be interesting to find the largest $p'$ for which one can a.a.s.\ embed $\G(n_1,n_2,p')$ into $\Rnnp$ even in this case. The tightest embedding one can expect is the one permitted by the maximum degree. We conjecture that such an embedding is possible.
\end{remark}

\begin{conjecture}
  \label{conj}
	Suppose that a sequence of parameters $(n_1, n_2, p, p')$ is such that a.a.s.\ in $\G(n_1, n_2, p')$ the maximum degree over~$V_1$ is at most $d_1 = n_2p$ and the maximum degree over~$V_2$ is at most $d_2 = n_1p$. There is a joint distribution of $\G(n_1, n_2, p')$ and $\Rnnp$ such that
	\begin{equation*}
		\G(n_1, n_2, p') \subseteq \Rnnp \quad a.a.s. \quad
	\end{equation*}
\end{conjecture}
Note that if Conjecture~\ref{conj} is true, then by taking complements we also have the tightest embedding $\Rnnp \subseteq \G(n_1,n_2,p'')$ that the minimum degrees of $\G(n_1,n_2,p'')$ permit.

We also propose the following strengthening of the Kim--Vu Sandwich Conjecture.
\begin{conjecture}
  Suppose that a sequence of parameters $(n, p, p')$ is such that a.a.s.\ in $\G(n, p')$ the maximum degree is at most $d = (n-1)p$. There is a joint distribution of $\G(n, p')$ and $\R(n,p)$ such that
	\begin{equation*}
	  \G(n, p') \subseteq \R(n,p) \quad a.a.s. \quad
	\end{equation*}
\end{conjecture}

\begin{remark}
  \label{rem_large_p}
  For constant $p$, to obtain $\gamma = o(1)$ in Theorem~\ref{thm_embed}, we need to assume $\I = 0$, which requires a rather restricted ratio $n_1/n_2$. For example, one cannot afford $n_1 = n_2^{1+\delta}$ for any constant $\delta > 0$. This restriction comes from an enumeration result we use in the proof, namely, Theorem~\ref{thm_enum} (see condition~\ref{thm_enum_dense} therein). Should enumeration be proven with a relaxed condition, it would automatically improve our Theorem~\ref{thm_embed}.
\end{remark}
\begin{remark}
  \label{rem_Ip}
  The terms $ p^2\I$ and $q^{3/2}\I$ in \eqref{eq_gamma} are artifacts of the application of switchings in Lemma~\ref{lem_codegs}. In the sparse case ($\hat p \to 0$) it is plausible that the condition $\I = 0$ can be made much milder by using a very recent enumeration result of Liebenau and Wormald~\cite{LW2020} instead of Theorem~\ref{thm_enum}. Due to the schedule of this manuscript, we did not check what this would imply, but readers seeking smaller errors in \eqref{eq_gamma} are encouraged to do so.
\end{remark}

\bibliographystyle{abbrv}

\end{document}